\documentclass[reqno,11pt]{amsart}
\usepackage{amssymb, amsmath,latexsym,amsfonts,amsbsy, amsthm,mathrsfs}
\usepackage{graphicx}
\usepackage{color}
\usepackage{enumerate}
\usepackage{stmaryrd}
\usepackage{appendix}

\usepackage{geometry}
\usepackage[bookmarks=true,colorlinks,linkcolor=red]{hyperref}
\allowdisplaybreaks
\makeatletter

\newcommand{\Rmnum}[1]{\expandafter\@slowromancap\romannumeral #1@}

\newcommand{\NSF}{\mathrm{NSF}}

\def\ddt{\frac{\mathrm{d}}{\mathrm{d}t}}

\newtheorem{athm}{\bf \t}[section]
\newenvironment{thm} [1] {\def\t{#1}\begin{athm} \bf \rm} {\end{athm}}
\newcommand{\bthm}{\begin{thm}}
\newcommand{\ethm}{\end{thm}}

\newtheorem{theorem}{Theorem}[section]
\newtheorem{lemma}[theorem]{Lemma}

\newtheorem{remark}[theorem]{Remark}

\newtheorem{proposition}[theorem]{Proposition}

\newcommand{\beq}{\begin{equation}}
\newcommand{\eeq}{\end{equation}}
\newcommand{\ben}{\begin{eqnarray}}
\newcommand{\een}{\end{eqnarray}}
\newcommand{\beno}{\begin{eqnarray*}}
\newcommand{\eeno}{\end{eqnarray*}}
\newcommand{\bali}{\begin{aligned}}
\newcommand{\eali}{\end{aligned}}

\numberwithin{equation}{section}

\DeclareMathOperator*{\Ker}{Ker}
\DeclareMathOperator*{\Span}{Span}
\newcommand{\al}{\alpha}
\newcommand{\be}{\beta}

\newcommand{\ve}{\varepsilon}
\newcommand{\ga}{\gamma}

\newcommand{\La}{\Lambda}

\newcommand{\ud}{\mathrm{d}}

\newcommand{\PP}{\mathbf{P}}

\newcommand{\one}{\mathbf{1}}
\newcommand{\Id}{\mathrm{Id}}

\newcommand{\CB}{\mathcal{B}}
\newcommand{\CS}{\mathcal{S}}

\newcommand{\CF}{\mathcal{F}}

\newcommand{\CT}{\mathcal{T}}

\newcommand{\CA}{\mathcal{A}}
\newcommand{\CP}{\mathcal{P}}

\newcommand{\CL}{\mathcal{L}}
\newcommand{\CQ}{\mathcal{Q}}

\newcommand{\BS}{{\mathbb{S}^2}}
\newcommand{\BR}{{\mathbb{R}^3}}

\newcommand{\MXe}{\mathscr{X}^{\varepsilon}}

\newcommand{\vp}{v^{\prime}}
\newcommand{\vs}{v_*}
\newcommand{\vsp}{v_*^{\prime}}

\newcommand{\fve}{f^{\varepsilon}}
\newcommand{\gve}{g^{\varepsilon}}

\newcommand{\Fp}{F^{\prime}}

\newcommand{\Fsp}{F^{\prime}_*}
\newcommand{\Fve}{F^{\varepsilon}}

\newcommand{\muq}{\mu_{\mathrm{q}}}
\newcommand{\weight}{\langle\cdot\rangle}
\newcommand{\fin}{f_{\textrm{in}}}
\newcommand{\Fin}{F_{\textrm{in}}}
\newcommand{\KerL}{\text{Ker}\,\mathcal{L}}
\begin{document}
\allowdisplaybreaks[4]

\title[The Incompressible NSF Limits from BFD Equation]{The Incompressible Navier-Stokes-Fourier Limits from Boltzmann--Fermi--Dirac Equation for Low Regularity Data}

\author{Ning Jiang}
\address{School of Mathematics and Statistics, Wuhan University, Wuhan 430070, China}
\email{njiang@whu.edu.cn}
\author{Chenchen Wang
}
\address{School of Mathematics and Statistics, Wuhan University, Wuhan 430070, China}
\email{ccwangmath@whu.edu.cn}

\author[Kai Zhou]{Kai Zhou}
\address[Kai Zhou]
		{\newline School of Mathematical Sciences, Nanjing Normal University, Nanjing 210023, China}
\email{kzyd@njnu.edu.cn}

\begin{abstract}
We consider the hydrodynamic limits of the quantum Boltzmann equation with Fermi-Dirac statistics  for hard sphere and hard potentials in the whole space. By analyzing the spectrum of the linearized collision operator combined with the transport operator and its associated semigroup, the incompressible Navier–Stokes-Fourier limits from the BFD equation is verified rigorously. Compared to the results in \cite{JXZ2022JDE}, this paper works with a lower regularity for the initial data. In addition, the fixed-point arguments together with a time iteration ensure us to obtain the lifespan of kinetic solution coincides with those of limiting fluid solution.\\
\noindent{\bf Keywords:} Boltzmann--Fermi--Dirac equation, Hard potentials, Low regularity, Spectral analysis, Incompressible Navier-Stokes-Fourier limits.
\end{abstract}

\maketitle

\tableofcontents

\section{Introduction}\label{introduction-section}
\subsection{The Boltzmann-Fermi-Dirac Equation}
\sloppy
We study the modified Boltzmann equation for Fermi--Dirac particles which obey Pauli's exclusion principle. The motion of Fermi--Dirac particles satisfies the dimensionless Boltzmann--Fermi--Dirac equation ({\it briefly, BFD equation})
\begin{align}\label{BFD}
\text{St}\,\partial_t F(t,x, v)+v\cdot\nabla_xF(t,x,v)=\tfrac{1}{\text{Kn}}\,\mathcal{C}(F)(t, x,v),\quad(t,x, v)\in \mathbb{R}_+\times\BR\times\BR,
\end{align}
where``St'' (the Strouhal number) describes the mean collision frequency in the equilibrium state, and ``Kn'' (the Knudsen number) is the ratio of the mean free path to the macroscopic length scale. The collision operator $\mathcal{C}$ in \eqref{BFD} is given by
\begin{align}\label{op-C}
\mathcal{C}(F)=\iint_{\BR\times\BS}B( v- v_*,\sigma)\Big(F^{\prime}F_*^{\prime}(1- F)(1- F_*)-FF_*(1-F^{\prime})(1- F_*^{\prime})\Big)\ud\sigma\ud v_*,
\end{align}
where $F(t, x,v):\mathbb{R}_+\times\BR\times\BR\longmapsto\mathbb{R}_+$ describes the probability distribution of particles at a given time $t \geqslant 0$, position $x\in\BR$ with velocity $ v\in\BR$.

Arising in the collision operator $\mathcal{C}(F)$ due to the Pauli exclusion principle, the factor $1-F$ stands for the probability that a fermion (such as an electron) does not lie a range of the velocity range $\ud v$. In contrast, the factor $1+F$ corresponds to bosons and appears in the Boltzmann--Bose--Einstein equation ({\it briefly, BBE equation}). However, qualitative analysis of the BBE equation is more challenging than that of the BFD equation since $1 + F$ leads to a blow-up of solution to the BBE equation in finite time \cite{cai2019spatially} and the condensation phenomenon \cite{escobedo2015finite}.

Here and after, for any function $F(v)$ we always denote by
\begin{align*}
F\equiv F( v),~F_*\equiv F( v_*), ~\Fp\equiv F( v^{\prime}),~ \Fsp\equiv F( v^{\prime}_*),
\end{align*}
and $v$, $\vs$ and $\vp$, $\vsp$ represent the velocities of two particles before and
after their collision, respectively. During the collision, the conservation of momentum and energy for particle pairs can be expressed as
\begin{equation}\label{conservative-law}
    v^{\prime} + v^{\prime}_* = v + v_*,\quad |v^{\prime}|^2 + |v^{\prime}_*|^2 = |v|^2 + |v_*|^2.
\end{equation}
Then the velocities $\vp$ and $\vsp$ after collision can be solved and represented by
\begin{equation*}
    v^{\prime}=\frac{ v+ v_*}{2}+\frac{| v- v_*|\sigma}{2},\quad v^{\prime}_*=\frac{ v+ v_*}{2}-\frac{| v- v_*|\sigma}{2},\quad\sigma\in\BS.
\end{equation*}

The collision kernel $B(v-v_*,\sigma)$ in \eqref{op-C} is a non--negative Borel function depending only on the relative velocity $|v-v_*|$ and the scalar product $\langle v-v_*,\sigma\rangle$. We will consider the case where $B(v-v_*,\sigma)$ has an explicit form as below.

In the weak-coupling regime, $B(| v - \vs|, \sigma)$ can be mathematically expressed as (see \cite{benedetto2004nonlinear,benedetto2006lowdensity,benedetto2007review})
\begin{align}
    B\left(v-v_*, \sigma\right)=\left|v-v_*\right|\left(\widehat{\phi}\left(\left|v-v^{\prime}\right|\right)- \widehat{\phi}\left(\left|v-v_*^{\prime}\right|\right)\right)^2,\nonumber
\end{align}
where $\phi$ is the potential function and $\widehat{\phi}$ denotes the Fourier transform of $\phi$. Physically, we consider an inverse power law \cite{cercignani2006boltzmann} for the interaction potential, namely,
\begin{align*}
    \phi(z)=|z|^{-p},\quad p>1.
\end{align*}
In addition, the critical value $p=1$ corresponds to the Coulomb potential. For $z\in\BR$ and $1<p<3$, we have $\widehat{\phi}(\xi)=\widehat{\phi}(|\xi|)=C|\xi|^{p-3}$, for some $C>0$. For $\cos\theta=\langle\frac{v-v_*}{|v-v_*|},\sigma\rangle$, $\theta\in[0,\frac{\pi}{ 2}]$, the relationships from the $\sigma$-representation can directly lead to $|v-\vp|=|v-\vs|\sin\frac{\theta}{2}$ and $|v-\vsp|=|v-\vs|\cos\frac{\theta}{2}$. Then without loss of generality, the collision kernel can be written as
\begin{align}
    B(v-v_*,\sigma)=|v-\vs|^{\gamma}b(\cos\theta),\quad \gamma=2p-5\in(-3,1),~p\in(1,3),\nonumber
\end{align}
where due to the symmetry structure of the scattering function $b(\cos\theta)$ on $[0,\pi]$, $b(\cos\theta)$ can be written as
\begin{align*}
    b(\cos\theta)=\one_{\theta\in[0,\frac{\pi}{2}]}\Big(\sin^{p-3}\tfrac{\theta}{2}-\cos^{p-3}\tfrac{\theta}{2}\Big)^2.
\end{align*}
In summary, the collision kernel can be classified according to the range of $\gamma$ and $p$:
\begin{enumerate}
    \item[(I)] When $\gamma =1$, the collision kernel is called \textit{hard sphere} collision kernel and $b(\cos\theta) = 1$;
    \item[(II)] When $\gamma \in (0,1)$, i.e., $p \in \left( \frac{5}{2}, 3 \right)$, the collision kernel is called \textit{hard potential} kernel;
     \item[(III)] When $\gamma = 0$, i.e., $p = \frac{5}{2}$, the collision kernel is called \textit{Maxwell collision} kernel;
    \item[(IV)] When $\gamma \in (-1,0)$, i.e., $p \in \left(2, \frac{5}{2} \right)$, the collision kernel is called \textit{soft potential} kernel;
    \item[(V)] When $\gamma \in (-3,-1)$, i.e., $p \in \left(1,2\right)$, the collision kernel is called \textit{very soft potential} or \textit{non-cutoff soft potential} kernel. See \cite{JiangZhou2024Global,YLZhou2023ADV} for more details.
\end{enumerate}

In this work, we only consider the hard sphere and hard potentials, namely, cases (I) and (II) above. For $\gamma\in(0,1)$, integrating the corresponding scattering function $b(\cos\theta)$ on $\BS$, we get
\begin{align}
\int_{\mathbb{S}^2} b(\cos\theta) \, \ud\sigma
\leqslant 2\pi \int_0^{\pi/2} \left( \sin^{2p - 6} \tfrac{\theta}{2} + \cos^{2p - 6} \tfrac{\theta}{2} \right) \sin\theta \, \ud\theta
= \tfrac{2^{4 - p} \pi}{p - 2}.\nonumber
\end{align}
On the other hand, since $\sin \frac{\theta}{2} \leqslant \frac{\sqrt{3}}{3} \cos \frac{\theta}{2}$ for $\theta \in \left[0, \frac{\pi}{3} \right]$, we get for some constant $C_p > 0$,
\begin{align*}
\int_{\mathbb{S}^2} b(\cos\theta) \, \ud\sigma
\geqslant 2\pi \int_0^{\pi/3} \left( 3^{(3 - p)/2} - 1 \right) \cos^{2p - 6} \tfrac{\theta}{2} \sin\theta \, d\theta
= C_p.
\end{align*}
Thus we can define
\begin{align}
A_p = \int_{\mathbb{S}^2} b(\cos\theta) \, d\sigma,\nonumber
\end{align}
and $
0 < C_p < A_p < \frac{2^{4 - p} \pi}{p - 2}.
$
Therefore, this case for $\gamma\in(0,1)$ can be seen as the \textit{angular cutoff}.

\subsection{Hydrodynamic Limits of BFD Equation}
In the thesis of Zakrevskiy \cite{zakrevskiy2015kinetic}, several fluid equations such as compressible Euler equations, compressible and incompressible Navier-Stokes-Fourier equations were formally derived from the scaled BFD equation under appropriate time and spatial scales. In particular, the incompressible Navier-Stokes-Fourier limits was discussed in Chapter 3 of \cite{zakrevskiy2015kinetic}. The current paper is devoted to rigorously justifying this limit.

We shall work with the small parameters $\text{Kn}=\text{St}=\ve \ll 1$ and we start from the corresponding scaled equation:
\begin{equation}\label{BFD-scaling}
  \left\{
  \begin{split}
     \partial_t F^{\ve}(t,x, v)+\frac{1}{\ve}v\cdot\nabla_x \Fve(t,x,v)= &\frac{1}{\ve^2}\mathcal{C}(\Fve)(t, x,v),\quad\text{in}~(0,\infty)\times\BR\times\BR, \\
     \Fve(0,x,v)= & \Fin^{\ve}(x,v), \qquad\qquad\,\text{in}~\BR\times\BR.
  \end{split}
  \right.
\end{equation}
Generally, incompressible flows can be derived in the near-equilibrium framework. We then set the perturbation around the global Fermi-Dirac distribution $\mu$:
\begin{align*}
    \Fve(t,x,v)=\mu+\ve\sqrt{\muq}\fve(t,x,v),\quad \Fin^{\ve}(x,v)=\mu+\ve\sqrt{\muq}\fin^{\ve}(x,v),
\end{align*}
where
\begin{align}\label{mu}
    \mu(v):=\tfrac{1}{1+e^{\frac{|v|^2}{2}-1}}, \quad \mu_{\text{q}}=\mu(1-\mu).
\end{align}

Now, for the simplicity of presentation, we introduce
\begin{align}
    \mu_{\text{obtain}}:=\mu^{\prime}\mu_*^{\prime}(1-\mu)(1-\mu_*),\quad\mu_{\text{loss}}:=\mu\mu_*(1-\mu^{\prime})(1-\mu^{\prime}_*).\nonumber
\end{align}
Notice that $\mu_{\text{obtain}}=\mu_{\text{loss}}$ due to the conservative law \eqref{conservative-law}, we define the symmetric function by
\begin{align}
   \omega_{\mu}:=\sqrt{\mu_{\text{obtain}}}\sqrt{\mu_{\text{loss}}}=(\sqrt{\mu_{\text{q}}})(\sqrt{\mu_{\text{q}}})_*(\sqrt{\mu_{\text{q}}})^{\prime}(\sqrt{\mu_{\text{q}}})_*^{\prime}.\nonumber
\end{align}
Putting the perturbation into the scaled BFD equation \eqref{BFD-scaling}, we obtain the Cauchy problem
\begin{equation}\label{BFD-perturbation}
    \left\{\begin{split}
        \partial_t \fve+ \frac{1}{\ve} v\cdot\nabla_x\fve = & \,\frac{1}{\ve^2}\CL\fve + \frac{1}{\ve} \CQ(\fve,\fve) + \CT(\fve,\fve,\fve),\\
        \fve(0,x,v)=& \,\fin^{\ve}(x,v),
    \end{split}\right.
\end{equation}
where the linear operator $\CL$, bilinear and trilinear operators $\CQ$ and $\CT$ are respectively given by
    \begin{enumerate}
        \item[\textbullet] The linear operator $\CL g$:
        \begin{align}\label{linear-operator}
            \mathcal{L} g=\iint_{\BR\times\BS} \tfrac{B\left(v-v_*, \sigma\right) \omega_{\mu}}{\sqrt{\mu_{\text{q}}}}\left[\left(\tfrac{g}{\sqrt{\mu_{\text{q}}}}\right)_*^{\prime}+\left(\tfrac{g}{\sqrt{\mu_{\text{q}}}}\right)^{\prime} -\left(\tfrac{g}{\sqrt{\mu_{\text{q}}}}\right)_*-\tfrac{g}{\sqrt{\mu_{\text{q}}}}\right] \mathrm{d} \sigma \mathrm{d} v_*.
        \end{align}
        \item[\textbullet] The bilinear operator $\CQ(f,g)$:
        \begin{align*}
            \mathcal{Q}(f, g)=\mathcal{Q}_1(f, g)+\mathcal{Q}_2(f, g)+\cdots+\mathcal{Q}_6(f, g),
        \end{align*}
        with $\displaystyle\mathcal{Q}_i(f, g)=\iint_{\BR\times\BS} \tfrac{B\left(v-v_*, \sigma\right)}{\sqrt{\mu_{\text{q}}}} q_i(f, g) \mathrm{d} \sigma \mathrm{d} v_*, i=1,2, \cdots, 6$, and
     \begin{align}\label{qi-definition}
\begin{aligned}
&q_1(f, g) = \left(\sqrt{\mu_{\text{q}}} f\right)_*^{\prime} \left(\sqrt{\mu_{\text{q}}} g\right)^{\prime} (1 - \mu_* - \mu),\\
&q_2(f, g) = - \left(\sqrt{\mu_{\text{q}}} f\right)_* \left(\sqrt{\mu_{\text{q}}} g\right) (1 - \mu_*^{\prime} - \mu^{\prime}), \\
&q_3(f, g) = \left(\sqrt{\mu_{\text{q}}} f\right)_*^{\prime} \left(\sqrt{\mu_{\text{q}}} g\right) (\mu_* - \mu^{\prime}),\\
&q_4(f, g) = - \left(\sqrt{\mu_{\text{q}}} f\right)_* \left(\sqrt{\mu_{\text{q}}} g\right)^{\prime} (\mu_*^{\prime} - \mu), \\
&q_5(f, g) = \left(\sqrt{\mu_{\text{q}}} f\right)_*^{\prime} \left(\sqrt{\mu_{\text{q}}} g\right)_* (\mu - \mu^{\prime}),\\
&q_6(f, g) = \left(\sqrt{\mu_{\text{q}}} f\right)^{\prime} \left(\sqrt{\mu_{\text{q}}} g\right) (\mu_* - \mu_*^{\prime}).
\end{aligned}
\end{align}
\item[\textbullet] The trilinear operator $\CT(f,g,h)$:
\begin{align*}
    \mathcal{T}(f, g, h)=\mathcal{T}_1(f, g, h)+\mathcal{T}_2(f, g, h)+\cdots+\mathcal{T}_4(f, g, h),
\end{align*}
 with $\displaystyle\mathcal{T}_i(f, g, h)=\iint_{\BR\times\BS} \tfrac{B\left(v-v_*, \sigma\right)}{\sqrt{\mu_{\text{q}}}} t_i(f, g, h) \mathrm{d} \sigma \mathrm{d} v_*$, $i=1, \cdots, 4$, and
 \begin{align*}
     \begin{aligned}
         & t_1(f, g, h)=\left(\sqrt{\mu_{\text{q}}} f\right)_*\left(\sqrt{\mu_{\text{q}}} g\right)\left(\sqrt{\mu_{\text{q}}} h\right)_*^{\prime},\\
         & t_2(f, g, h)=\left(\sqrt{\mu_{\text{q}}} f\right)_*\left(\sqrt{\mu_{\text{q}}} g\right)\left(\sqrt{\mu_{\text{q}}} h\right)^{\prime}, \\
         & t_3(f, g, h)=-\left(\sqrt{\mu_{\text{q}}} f\right)_*^{\prime}\left(\sqrt{\mu_{\text{q}}} g\right)^{\prime}\left(\sqrt{\mu_{\text{q}}} h\right)_*,\\
         & t_4(f, g, h)=-\left(\sqrt{\mu_{\text{q}}} f\right)_*^{\prime}\left(\sqrt{\mu_{\text{q}}} g\right)^{\prime}\left(\sqrt{\mu_{\text{q}}} h\right).
     \end{aligned}
 \end{align*}
\end{enumerate}
Furthermore, we introduce the symmetrized forms of the bilinear and trilinear operators as
\begin{align}
  \CQ_{\text{sym}}(f,g):=&\tfrac{1}{2}(\CQ(f,g)+\CQ(g,f)),\label{bilinear-sym}\\
  \CT_{\text{sym}}(f,g,h):=&\tfrac{1}{6}\Big(\CT(f,g,h)+\CT(f,h,g)\Big)+\tfrac{1}{6}\Big(\CT(h,f,g)+\CT(h,g,f)\Big),\label{trilinear-sym}\\
  &+\tfrac{1}{6}\Big(\CT(g,h,f)+\CT(g,f,h)\Big), \nonumber
\end{align}
then we have $\CQ_{\text{sym}}(f,f) = \CQ(f,f)$ and $\CT_{\text{sym}}(f,f,f)=\CT(f,f,f)$.

As is shown in \cite{zakrevskiy2015kinetic}, we take inner products by \eqref{BFD-perturbation} with $[1,v,|v|^2]\sqrt{\mu_{\text{q}}}$ and let $\ve\to 0$ to formally obtain the incompressible Navier–Stokes-Fourier equations ({\it NSF equations for short})
\begin{align}\label{NSF}
    \begin{aligned}
    &\left\{\begin{array}{l}
    E_2 \partial_t u+E_2 u \cdot \nabla_x u+\nabla_x p=\nu_* \Delta_x u, \\
    \nabla_x \cdot u=0, \quad\nabla(\rho+\vartheta)=0,\\
    C_A \partial_t \vartheta+C_A u \cdot \nabla_x \vartheta=\kappa_* \Delta_x \vartheta,
    \end{array}\right.
\end{aligned}
\end{align}
where
\begin{align*}
 \nu_*=\int_{\BR}\beta_{\mathcal{L}}(|v|) v_1^2 v_2^2\muq\ud v, \quad \kappa_*=\int_{\BR}\alpha_{\mathcal{L}}(|v|)\left(\frac{|v|^2}{2}-K_A\right)^2 v_1^2\muq\ud v,
\end{align*}
and the definitions of constants $E_2$, $C_A$, $K_A$ and the positive functions $\alpha_{\mathcal{L}}(|v|)$ and $\beta_{\mathcal{L}}(|v|)$ can be found in Section \ref{sec-ntt-mr}.

\subsection{History and Literature}
As a modification of classical Boltzmann equation, the quantum Boltzmann equation was first obtained by heuristic arguments of Nordheim \cite{Nordheim1928} and Uehling-Uhlenbeck \cite{UU1933}. However, unlike the classical ones, the rigorous derivation of the quantum Boltzmann equation from particle system is not yet well-established. Partial results considering the derivation of the quantum Boltzmann equation are presented in \cite{Spohn-2010,ESY-2004JSP,benedetto2004nonlinear,benedetto2005weak,benedetto2006lowdensity,benedetto2007review}.

We review some results concerning the fluid limits from quantum Boltzmann equation. One historical motivation for rigorously deriving fluid models from kinetic equations via asymptotic expansions comes from Hilbert’s sixth problem \cite{Hilbert1900Problems}. The earlier formal derivations of the fluid limits from kinetic equation can be found in \cite{Chapman1960,Grad1963}. Although there are numerous research results on rigorous verification of hydrodynamic limits from the classical Boltzmann equation, such as \cite{Bardos1991JSP,Bardos1991M3AS,Caflisch1980CPAM,Golse2002CPAM,Golse2004Invent,Golse2009JMPA,Guo2006CPAM,Jiang2010CPDE,JiangXuZhao2018IUMJ} and literature cited therein, the study on hydrodynamic limits from the quantum Boltzmann equation has only emerged in the past decade. For the fluid limits of BFD equation, Zakrevskiy formally derived in his thesis \cite{zakrevskiy2015kinetic} the compressible Euler and Navier-Stokes limits and incompressible Navier-Stokes limits. We also mention that, Filbet-Hu-Jin \cite{FilbetHuJin2012MMNA} introduced a new scheme for quantum Boltzmann equation to capture the Euler limit by numerical computations. Motivated by Zakrevskiy's formal analysis, Jiang-Xiong-Zhou \cite{JXZ2022JDE} was firstly devoted to rigorously justifying the incompressible NSF limits from BFD equation for the hard sphere case in the near-equilibrium framework. In their work, they handled the trilinear structure of the collision term arising from the quantum effects and obtained the global existence of BFD equation (for hard potential case, see \cite{JiangZhou2024Global}). In the same framework, Jiang-Zhou \cite{JiangZhou2024Acoustic} verified the acoustic limit of BFD equation. Later on, by analysing the nonlinear implicit relation between the BFD equation and the compressible Euler system, Jiang-Zhou \cite{JiangZhou2024QBE} employed the Hilbert expansion, together with carefully establishing Grad–Caflisch type decay estimates for the inverse of linearized operator, to establish the convergence from BFD equation to the Euler system. In a same spirit of \cite{JiangZhou2024QBE}, Jiang-Yang-Zhou \cite{JYZ2025JDE} proved the compressible Euler-Poisson limits from Vlasov-Poisson-BFD equation. For BBE equation, He-Jiang-Zhou \cite{HeJiangZhou2023BBE} proved its incompressible NSF limits under high temperature assumption.

The above results on fluid limits of kinetic equation are mainly based on the nonlinear energy estimates uniformly in Knudsen number $\ve$ and are in the classical/weak solution framework. Another way to obtain the fluid limits relies on the spectral analysis of the perturbed linearized collision operator (the scaled linear operator $\ve^{-2}\CL + \ve^{-1}v\cdot\nabla_x$) and sharp estimates for its associated semigroup. For classical Boltzmann equation, this approach was developed by Grad \cite{Grad1965PSAM} and Bardos--Ukai \cite{BU1991MMMAS}. In Bardos--Ukai’s work, the key feature is that they only needed the smallness of the initial data, and did not assume the smallness of the Knudsen number $\ve$. Meanwhile, Ellis--Pinsky \cite{ellis1975fluid} carefully studied the perturbed linearized collision operator mentioned above and performed its spectral decomposition. More results on related spectral theory can be found in \cite{Nicolaenko1971,Cercignani1994}.

In recent years, certain research on fluid limits have resorted to spectral decomposition for classical Boltzmann equation. Gallagher-Tristani \cite{Gallagher2020AHL} employed this approach to find the connection between strong solutions of the Boltzmann and the Navier--Stokes equations. They shown that the life span of the solutions of the rescaled Boltzmann equation is bounded from below by that of the Navier–Stokes system. Gervais--Lods \cite{gervais2024hydrodynamic} presented a more modern spectral approach to achieve strong convergence from various kinetic equations for not too soft potentials to the NSF system under structural assumptions. Especially, based on the fact that the solutions of Navier--Stokes equation exist for initial data in $H^{1/2}(\mathbb{T}_x^3)$, Carrapatoso--Gallagher-Tristani \cite{carrapatoso2025navier} firstly established a rigorous correspondence between kinetic and Navier-–Stokes solutions with low spatial regularity initial data and without any smallness assumption on the fluid initial data.

To our best knowledge, there are no results concerning the corresponding spectral analysis for quantum case. In the present paper, we investigate the properties of spectrum of the perturbed linearized collision operator for BFD equation with hard sphere and hard potentials collision kernel. Then in a same spirit of \cite{carrapatoso2025navier}, we set out to understand the hydrodynamic limits of the BFD equation in the whole space $\mathbb{R}_x^3$. Compared to the previous results \cite{JXZ2022JDE} by Jiang–Xiong–Zhou, this paper works without the smallness assumption on the macroscopic part of the kinetic initial data and provides a possible blow-up of the microscopic part of the kinetic initial data, which is characterized by $\ve^{-\beta}$ for some $\beta<1/2$. In addition, we need more refined estimates for trilinear terms (see Lemma \ref{lem-tri-estimates}) than those in \cite{JXZ2022JDE}. By using a fixed-point argument together with a time iteration, it turns out that in this work the smallness assumption on macroscopic part of the initial data is not needed.

\subsection{Notations and Main Results}\label{sec-ntt-mr}
Firstly, we make some conventions on symbols for convenience.
\begin{itemize}
\item Unless specified, capital $C$ is a generic constant and independent of any parameters. In addition, we use $C_{\alpha_1,\alpha_2,\cdots}\equiv C(\alpha_1,\alpha_2,\cdots)>0$ when the constant $C$ depends on the parameters $\alpha_1,\alpha_2,\cdots$.
  \item The relation `$a\sim b$' means $C_1 b \leqslant a \leqslant C_2 b$.  $a\lesssim(\gtrsim)\, b$ means $a \leqslant (\geqslant)\, C b$ for some constant $C>0$.
\item For $l\in\mathbb{R}$, we denote the weight function by $\langle v \rangle^{l}:=(1+|v|^2)^{\frac{l}{2}}$. Particularly for $l=\gamma$, it is convenient to define the weighted Lebesgue space in $v$:
\begin{align*}
    L_\gamma^2 = \big\{f(v)\in L^2(\mathbb{R}^3)\,|\,\|f\|_{L_\gamma^2}^2\equiv\int_{\mathbb{R}^3} |f|^2\langle v \rangle^{\gamma}\ud v< +\infty\big\},
\end{align*}
which is endowed with the norm $\|\cdot\|_{L_\gamma^2}$.
\item In this paper, the Fourier transformation always acts on $x$ variable, that is, $\widehat{f}(t, \xi, v)$ is the Fourier transformation of $f(t,x,v)$. Without causing confusion, we slightly abuse the notation $\widehat{f}(\xi)$ for the Fourier transformation of $f(\cdot,x,\cdot)$. When more convenient, we will sometimes use the notation $\mathcal{F}_x f$ for $\widehat{f}$. 
\item $\langle D_x \rangle = (\mathrm{Id} - \Delta_x)^{1/2}$ is the Fourier multiplier with symbol $(1 + |\xi|^2)^{1/2}$. The standard Hilbert space $H_x^s$ can be endowed with the equivalent norm
\begin{equation*}
  \|f\|_{H_x^s} = \|\langle D_x\rangle^s f\|_{L_x^2}.
\end{equation*}
We also slightly abuse notation and use
\begin{align*}
    |D_x| = (- \Delta_x)^{1/2},
\end{align*}
to denote the Fourier multiplier with symbol $|\xi|$. Similarly, the homogeneous Hilbert space $\dot{H}_x^s$ is endowed with the equivalent norm
\begin{equation*}
  \|f\|_{\dot{H}_x^s} = \left\{\begin{split}
  &\||D_x|^s f\|_{L_x^2},\quad s\in(0,1],\\
  &\sum_{k=1}^{[s]}\||D_x|^k f\|_{L_x^2} + \||D_x|^s f\|_{L_x^2},\quad s>1.
  \end{split}\right.
\end{equation*}
\item For any $m \geqslant 0$ and $T>0$, we define the space $\widetilde{L}^{\infty}\left([0, T], H_x^m L^2_v\right)$ through its norm
\begin{align*}
\|f\|_{\widetilde{L}^{\infty}\left([0, T], H_x^m L^2_v\right)}^2:=\int_{\BR}\langle\xi\rangle^{2m}\|\widehat{f}(\cdot, \xi, \cdot)\|_{L^{\infty}\left([0, T], L^2_v\right)}^2\ud\xi.
\end{align*}
To simplify the notation, we shall often write $L_I^p H_x^m L^2_v$ for $L^p\left(I, H_x^m L^2_v\right)$ and similarly $\widetilde{L}_I^{\infty} H_x^m L^2_v$ for $\widetilde{L}^{\infty}\left(I, H_x^m L^2_v\right)$ when $I$ is an interval of $\mathbb{R}^{+}$. If $I=[0, T]$ we will simply write $L_T^p H_x^m L^2_v$ and $\widetilde{L}_T^{\infty} H_x^m L^2_v$. Finally, if $T=\infty$ and in the absence of ambiguity, we write $L_t^p H_x^m L^2_v$ for $L^p\left(\mathbb{R}^{+}, H_x^m L^2_v\right)$.
\item For convenience, we shall introduce some constants (see \cite[Section 4]{JXZ2022JDE})
\begin{align}\label{const-coefficents}
    \begin{aligned}
     &E_0 =\int_{\BR}\muq\ud v, \quad
E_2 =\int_{\BR}|v_1|^2\muq\ud v ,\\
&E_4 =\int_{\BR}|v_1|^4\muq\ud v, \quad
E_{22} =\int_{\BR}|v_1 v_2|^2\muq\ud v,\\
&C_A =\int_{\BR}\left( \frac{|v|^2}{2} - K_A \right)^2 v_1^2(1 - 2\mu) \muq\ud v, \quad
K_g + 1 = K_A = \frac{E_4 + 2 E_{22}}{2 E_2}.
    \end{aligned}
\end{align}
For the following two functions
\begin{align*}
A(v)=\left(K_A-\frac{|v|^2}{2}\right) v, \quad B(v)=\frac{|v|^2}{3} \text{Id}-v \otimes v,
\end{align*}
there exist unique functions $A^{\prime}(v)$ and $B^{\prime}(v)$ such that
\begin{align*}
    \mathcal{L}\left(A_i^{\prime}\right)=A_i, \quad \mathcal{L}\left(B_{i j}^{\prime}\right)=B_{i j}, \quad i, j=1,2,3 .\nonumber
\end{align*}
In fact, $A^{\prime}(v), B^{\prime}(v)$ can be solved and represented as
\begin{align}\label{func-efficients}
    A^{\prime}(v)=-\alpha_{\mathcal{L}}(|v|) A(v), \quad B^{\prime}(v)=-\beta_{\mathcal{L}}(|v|) B(v),
\end{align}
for two positive functions $\alpha_{\mathcal{L}}(|v|)$ and $\beta_{\mathcal{L}}(|v|)$, one can see the Appendix in \cite[Chapter 2]{zakrevskiy2015kinetic} for more details.
\end{itemize}

Now we are ready to present our main theorem, that is, the limits from BFD equation \eqref{BFD-perturbation} to the incompressible NSF system \eqref{NSF} in the low regularity settings.
\begin{theorem}\label{thm-main}
    Given the well-prepared initial data $(\rho_{\mathrm{in}},u_{\mathrm{in}},\vartheta_{\mathrm{in}})\in H_x^{\frac{1}{2}}(\BR)$ satisfying the incompressibility and Boussinesq relation:
    \begin{equation*}
      \nabla_x\cdot u_{\mathrm{in}}=0,\qquad \rho_{\mathrm{in}}+\vartheta_{\mathrm{in}}=0.
    \end{equation*}
    Let $T$ be given in Lemma \ref{lem-fluid-solution} and $(\rho,u,\vartheta)(t, x)\in \tilde{L}^{\infty}_T H^{\frac{1}{2}}_x\cap L^2_T H^{\frac{3}{2}}_x$ be the unique solution on time interval $[0,T]$ to the NSF equation \eqref{NSF} with the above initial data. Let $\alpha < \frac{1}{4}$ and $\beta < \frac{1}{2}$ be two real numbers. Consider the well-prepared initial data
\begin{align}\label{def-initial-data}
    g_{\mathrm{in}}:= \Big\{u_{\mathrm{in}}\cdot v+(\tfrac{K_g}{1+K_g}\vartheta_{\mathrm{in}}-\tfrac{1}{1+K_g}\rho_{\mathrm{in}})(\tfrac{|v|^2}{2}-K_g-1)\Big\}\sqrt{\muq},
\end{align}
where the constant $K_g$ is given in \eqref{const-coefficents}, and we denote by
\begin{align}\label{def-solution-g}
    g:=\{\rho+u\cdot v+\vartheta(\tfrac{|v|^2}{2}-K_g)\}\sqrt{\muq}.
\end{align}
Then for a family of functions $f^{\varepsilon}_{\mathrm{in}} = \mathbf{P}_0f^{\varepsilon}_{\mathrm{in}} + \mathbf{P}_0^{\perp}f^{\varepsilon}_{\mathrm{in}}$ satisfying
    \begin{itemize}
        \item $\PP_0\fve_{\mathrm{in}}=\psi(\ve^{\al}|D_x|)g_{\mathrm{in}}$, for some smooth compactly supported function $\psi$;
        \item For any $\PP_0^{\perp}\fve_{\mathrm{in}}\in (\Ker\,\CL)^{\perp}$,
        \begin{align}
          \left\|\mathbf{P}_0^{\perp} f_{\mathrm{in}}^{\varepsilon}\right\|_{H_x^{\frac{1}{2}} L_v^2}+\varepsilon^\beta\left\|\mathbf{P}_0^{\perp} f_{\mathrm{in}}^{\varepsilon}\right\|_{H_x^{\ell} L_v^2} \longrightarrow 0,\text{ as } \ve\to 0,\nonumber
        \end{align}
    \end{itemize}
    there is a $\varepsilon_0>0$ such that for any $\varepsilon \leqslant \varepsilon_0$, the Cauchy problem \eqref{BFD-perturbation} admits a unique solution
    \begin{equation*}
        f^{\varepsilon} \in \widetilde{L}_T^{\infty} H_x^{\ell} L_v^2,\quad \tfrac{3}{2}<\ell\leqslant 2,
    \end{equation*}
    with $\mathbf{P}_0 f^\ve \in L_T^2 \dot{H}_x^{\ell} L^2_{\gamma}$ and $\mathbf{P}_0^{\perp} f^\ve \in L_T^2 H_x^{\ell} L^2_{\gamma}$. Moreover, the solution $f^{\varepsilon}$ satisfies
 \begin{align*}
     \left\|f^{\varepsilon}-g\right\|_{\widetilde{L}_T^{\infty} H_x^{\frac{1}{2}} L_v^2} + \left\|\mathbf{P}_0 (f^{\varepsilon}-g)\right\|_{L_T^2 \dot{H}_x^{\frac{3}{2}} L^2_v} + \left\|\mathbf{P}_0^{\perp} (f^{\varepsilon}-g)\right\|_{L_T^2 H_x^{\frac{3}{2}} L^2_{\gamma}} \longrightarrow 0,\, \text{ as }\varepsilon \to 0.
 \end{align*}
\end{theorem}
\begin{remark}
  Theorem \ref{thm-main} shows that based on fluid solutions $(\rho,u,\theta)$ with low regularity, we can construct a kinetic solution $f^\ve$ with higher regularity. Furthermore, as the Knudsen number $\ve$ tends to zero, we prove that $f^\ve$ converges strongly in $\widetilde{L}_T^{\infty} H_x^{\frac{1}{2}} L_v^2$ to the infinitesimal Maxwellian $g$ generated by $(\rho,u,\vartheta)$.
\end{remark}

\begin{remark}
    This work focuses mainly on the hydrodynamic limits from the BFD equation under low regularity initial data setting. To lay stress on low regularity setting, we treat on purpose the case that the initial data is well prepared for simplicity. Indeed, there is mature method to handle the case that the initial data is not well prepared.
\end{remark}

\subsection{Organization of this Paper}
In the next Section, we recall the properties of the linearized collision operator $\CL$. We then analyze the spectral structure of its small perturbations in Fourier space and study the associated semigroup. In Section 3, we outline the strategy for the proof of the main theorem. Specifically, we perform spectral analysis and semigroup decomposition and study the difference equation obtained by taking the difference between the Duhamel formulations of two systems satisfied by the kinetic solution $f^\ve$ and the fluid solution $g^\ve$ respectively. A basic fixed-point lemma is then applied to this difference equation to obtain existence. The main difficulty lies in verifying the contraction condition of the linear operator. In fact, the smallness of $\ve$ cannot be freely used. To address this, we construct a time-local iterative scheme, which ensures that the conditions in the fixed-point lemma are satisfied. In the last section, we focus on proving Proposition \ref{prop-fixed-point-lem-conditions}. We show that the linear operator, the quadratic and the cubic nonlinear operators satisfy the assumptions of the fixed-point lemma.

\section{Preliminary work}
In this section, we mainly focus on the perturbative linear operator $\CL - iv \cdot \nabla_x$ and denote its Fourier transform in the space variable by
\begin{equation*}
  \widehat{\Lambda}(\xi):=\CL - iv \cdot \xi.
\end{equation*}
Then we will investigate the spectrum of $\widehat{\Lambda}(\xi)$ and its associated semi-group. These preparations lay the foundation for proving the strong convergence. Firstly, we recall some known properties of the linear operator $\CL$ for hard sphere and hard potentials.

\subsection{The Properties of the Linear Operator \texorpdfstring{$\CL$}{L}}
Decomposing the linear operator $\CL$ given by \eqref{linear-operator} into
\begin{equation*}
\mathcal{L}g = - \nu(v)g + \mathcal{K}g,
\end{equation*}
where the collision frequency $\nu(v)$ is
\begin{equation}
\nu(v) = \iint_{\mathbb{R}^3 \times \mathbb{S}^2} \tfrac{B(v - v_*, \sigma)\, \omega_{\mu}}{\mu_{\text{q}}} \, \ud\sigma\, \ud v_*,
\end{equation}
and $\mathcal{K}$ is given by
\begin{align*}
  \mathcal{K} g = \iint_{\mathbb{R}^3 \times \mathbb{S}^2} \frac{B(v - v_*, \sigma)\, \omega_{\mu}}{\sqrt{\mu_{\text{q}}}}\,\left[\left(\frac{g}{\sqrt{\mu_q}}\right)_*^\prime + \left(\frac{g}{\sqrt{\mu_q}}\right)^\prime - \left(\frac{g}{\sqrt{\mu_q}}\right)_*\right] \, \ud\sigma\, \ud v_*.
\end{align*}

\begin{lemma}[\cite{JXZ2022JDE,JiangZhou2024Global}]\label{lem-linear-operator-propreties}
    For $\gamma\in (0,1]$, there hold that
\begin{itemize}
    \item There are two constants $C_1, C_2 > 0$ such that
    \begin{equation*}
        C_1(1 + |v|)^{\gamma} \leqslant \nu(v) \leqslant C_2(1 + |v|)^{\gamma};
    \end{equation*}

    \item The operator
    \[
        \mathcal{K} : L^2_v \longrightarrow L^2_v
    \]
    is compact;

    \item The linearized collision operator $\mathcal{L}$ is symmetric, non-positive, and its null space is
    \begin{equation*}
        \operatorname{Ker} \CL= \Span \left\{ 1,\, v,\, |v|^2 \right\}\sqrt{\muq};
    \end{equation*}
    Moreover, let $\mathbf{P}_0:L^2_v\to \Ker\CL$ be the projection onto $\Ker\CL$. $\mathbf{P}_0$ can be expressed as
    \begin{align*}
        \mathbf{P}_0g=a_g+b_g\cdot v+c_g|v|^2,
    \end{align*}
    where $a_g,c_g\in\mathbb{R}$, $b_g\in\BR$ are projection coefficient under the projection $\mathbf{P}_0$;
    \item Let $\mathbf{P}_0^{\perp}:=\Id-\mathbf{P}_0$ be the orthogonal projection onto the orthogonal complement $(\operatorname{Ker} \CL)^{\perp}$ in $L^2_v$. There exists a constant $\lambda > 0$ such that
    \begin{equation*}
        -\langle \mathcal{L}g,\, g \rangle \geqslant \lambda\|\PP_0^{\perp} g\|^2_{L^2_{\gamma}},
        \quad \text{for any}~ g \in D(\mathcal{L}) = \left\{ g \in L^2_v \,\middle|\, \nu g \in L^2_v\right\}.
    \end{equation*}
\end{itemize}
\end{lemma}
\begin{remark}
  Any point in right plane of $\mathrm{Re}(z) = -\nu(0)$ is either resolvent or eigenvalue of finite multiplicity of multiplication operator $\nu(v)$. Then for $\widehat{\Lambda}(\xi) = \CL - iv \cdot \xi$, if $|\xi|\to 0$, the isolate points in spectrum are well-behavior and the resolve is bounded. One can refer to \cite{ellis1975fluid,gervais2024hydrodynamic} for more detailed arguments on spectral analysis of classical linearized Boltzmann operator $\CL_{\mathrm{c}}$ (`c' for classical), which hold for $\CL$ due to the self-adjointness, rotational invariant and spectral gap of $\CL$ as shown in the above lemma.
\end{remark}

\subsection{Analysis on the Scaled BFD Equation in Fourier Space}
We start with denoting by the linearized operator
\begin{align*}
\La^{\ve}:=\ve^{-2}\CL-\ve^{-1}v\cdot\nabla_x,
\end{align*}
and $U^{\ve}(t)$ the semi-group generated by $\La^{\ve}$. In the Fourier space for the space variable $x\in\BR$, the Fourier transform of the operator $\La^{\ve}$ will be
\begin{align*}
\widehat{\La}^{\ve}(\xi):=\ve^{-2}\CL-\ve^{-1} i v\cdot \xi,\quad \xi\in\BR,
\end{align*}
and its associated semi-group is denoted by $\widehat{U}^{\ve}(t):=e^{t\widehat{\La}^{\ve}(\xi)}$. Particularly, $\widehat{\La}^{1}(\xi) = \widehat{\La}(\xi)$. Then we get
\begin{align}
U^{\ve}(t)=\CF_x^{-1}\widehat{U}^{\ve}(t)\CF_x. \nonumber
\end{align}

Next, recall the nonlinear terms $\CQ_{\text{sym}}(f,g)$ and $\CT_{\text{sym}}(f,g,h)$ defined in \eqref{bilinear-sym}-\eqref{trilinear-sym}, we denote by the bilinear form $\widehat{\mathfrak{Q}}^{\ve}$ and trilinear form $\widehat{\mathfrak{T}}^{\ve}$ in the Fourier space
\begin{align*}
    \widehat{\mathfrak{Q}}^{\ve}[f,g](t,\xi):=&\ve^{-1}\int_0^t\widehat{U}^{\ve}(t-s)\widehat{\CQ}_{\text{sym}}(f(s),g(s))(\xi)\ud s,\\
    \widehat{\mathfrak{T}}^{\ve}[f,g,h](t,\xi):=&\int_0^t\widehat{U}^{\ve}(t-s)\widehat{\CT}_{\text{sym}}(f(s),g(s),h(s))(\xi)\ud s,
\end{align*}
with
\begin{align*}
 \widehat{\CQ}_{\text{sym}}(f,g)(\xi) :=&\int_{\BR}\CQ_{\text{sym}}(\widehat{f}(\xi-\xi^{\prime}),\widehat{g}(\xi^{\prime}))\ud\xi^{\prime},\\
 \widehat{\CT}_{\text{sym}}(f,g,h)(\xi) :=&\iint_{\BR\times\BR}\CT_{\text{sym}} (\widehat{f}(\xi-\xi^{\prime}-\xi^{\prime\prime}),\widehat{g}(\xi^{\prime}),\widehat{h}(\xi^{\prime\prime})) \ud\xi^{\prime}\ud\xi^{\prime\prime},
\end{align*}
we also denote by
 \begin{align}
  \mathfrak{Q}^{\ve}[f,g](t)=\CF_x^{-1}\widehat{\mathfrak{Q}}^{\ve}[f,g](t,\cdot)\CF_x,\quad \mathfrak{T}^{\ve}[f,g,h](t)=\CF_x^{-1}\widehat{\mathfrak{T}}^{\ve}[f,g,h](t,\cdot)\CF_x.\nonumber
 \end{align}

With the above notations, the BFD equation \eqref{BFD-perturbation} in Fourier space takes the \textit{Duhamel form}
 \begin{align}
  \widehat{f}^\varepsilon=\widehat{U}^{\ve}(t)\widehat{f}_{\mathrm{in}}^{\ve}+\widehat{\mathfrak{Q}}^{\ve}[\fve,\fve](t)+\widehat{\mathfrak{T}}^{\ve}[\fve,\fve,\fve](t).\nonumber 
 \end{align}
 For $\widehat{U}^{\ve}(t)$, we have by scaling that
 \begin{align*}
     \widehat{U}^{\ve}(t,\xi)=\widehat{U}^1(\ve^{-2} t,\ve \xi),
 \end{align*}
 then we consider the case $\ve = 1$. For $|\xi|$ small enough, $\La^1(\xi)$ can be regarded as a small perturbation of $\CL$. Hence we have the following lemma concerning the localization of the spectrum of the operator $\La^1(\xi)$, which is a simple analogy of Theorem 1.8 of \cite{gervais2024hydrodynamic} in the low-frequency regime (see also \cite{ellis1975fluid}).
\begin{lemma}[]\label{lem-spectral-thm}
Denote by $\mathfrak{S}(\widehat{\La}^1)$ the spectrum of the operator $\widehat{\La}^1$ on $L^2_v$. There is a constant $\kappa>0$ small enough such that the following spectral and dynamical properties hold.
    \begin{enumerate}
        \item\textbf{Eigenvalues.}
        \begin{itemize}
            \item For $|\xi|>\kappa$, there exists $\lambda_0>0$ such that
            \begin{align*}
                \mathfrak{S}(\widehat{\La}^1)\cap\{z\mid \mathrm{Re}\,z>-\lambda_0\}=\emptyset.
            \end{align*}
            \item For $|\xi|\leqslant\kappa$, there exists a $\lambda_1>0$ such that the spectrum is at a positive distance from ${\operatorname{Re} z \geqslant 0}$, except for a finite number of small eigenvalues:
            \begin{align*}
                \mathfrak{S}(\widehat{\La}^1)\cap\{z\mid \mathrm{Re}\,z>-\lambda_1\}=\left\{\lambda_{\mathrm{NS}}(\xi), \lambda_{\mathrm{heat }}(\xi), \lambda_{ \mathrm{wave}-}(\xi), \lambda_{ \mathrm{wave}+}(\xi)\right\},
            \end{align*}
            and these eigenvalues $\lambda_{\star}(\xi)$ expand for $\xi \rightarrow 0$ as
        \begin{align*}
    \begin{aligned}
    \lambda_{\star}(\xi)=& -\nu_{\star}|\xi|^2+\mathcal{O}\left(|\xi|^3\right), \quad \star=\mathrm { NS, heat, }\\
    \lambda_{ \mathrm {wave }\pm }(\xi)= & -\nu_{\mathrm {wave }}|\xi|^2\pm i c|\xi|+\mathcal{O}\left(|\xi|^3\right),
\end{aligned}
\end{align*}
where $c$ is the sound speed, and the coefficients $0<\nu_{\star}<\infty$ for $\star=\mathrm{NS,\,heat,\,}\mathrm{wave}\pm$,
\begin{align}
\begin{aligned}
&\nu_{\mathrm{NS}}:=-\tfrac{1}{8\sqrt{3E_2}}\int_{\BR}B^{\prime}(v):B(v)\muq\ud v, \\
&\nu_{\mathrm{heat}}:=-\tfrac{1}{6E\sqrt{K(K-1)}}\int_{\BR}A^{\prime}(v)\cdot A(v)\muq\ud v, \\
& \nu_{\mathrm{wave}}:=\tfrac{1}{3} \nu_{\mathrm{NSF}}+\tfrac{E^2(K-1)}{2} \nu_{\mathrm{heat}},
\end{aligned}\nonumber
\end{align}
\end{itemize}
where the constants $E=3E_2,K=\frac{3E_4+6E_{22}}{E^2}$ and constants $E_2,\,E_4,\,E_{22}$, as well as the functions $A^{\prime}(v),\,A(v),\,B^{\prime}(v),\,B(v)$ are given in \eqref{const-coefficents}-\eqref{func-efficients}.
\item \textbf{Spectral projectors.} The spectral projectors $\mathcal{P}_{\star} : L_\gamma^2 \to (L_\gamma^2)^{\prime}$ for $\star \in \{\mathrm{NS},\, \mathrm{heat},\, \mathrm{wave}\pm\}$ are bounded linear operators associated with the corresponding small eigenvalues. These projectors can be written as
\begin{align*}
\mathcal{P}_{\star}(\xi)=\mathcal{P}_{\star}^{(0)}\left(\tfrac{\xi}{|\xi|}\right)+|\xi| \mathcal{P}_{\star}^{(1)}\left(\tfrac{\xi}{|\xi|}\right)+|\xi|^2 \mathcal{P}_{\star}^{(2)}(\xi)\,,
\end{align*}
where the zero order coefficients are
\begin{align}
    \begin{aligned}
    &\mathcal{P}_{\mathrm {NS}}^{(0)}\left(\tfrac{\xi}{|\xi|}\right) \hat{f}(\xi, v)=\tfrac{1}{\sqrt{E_2}}\Big(\Pi_\xi\cdot\int_{\BR}\widehat{f}(\xi,v)v\sqrt{\muq}\ud v\Big)\cdot v\sqrt{\muq},\quad \Pi_\xi=\Id-\tfrac{\xi\otimes \xi}{|\xi|^2}\,,\\
     &\mathcal{P}_{\star}^{(0)}\left(\tfrac{\xi}{|\xi|}\right) \hat{f}(\xi, v)=\psi_{\star}(\tfrac{\xi}{|\xi|})\int_{\BR}\widehat{f}(\xi,v)\psi_{\star}(\tfrac{\xi}{|\xi|})\ud v\,,\quad\star \in\{\mathrm{heat},\,\mathrm {wave}\pm\}.
\end{aligned}\nonumber
\end{align}
Here, the zeroth order eigenfunctions $\psi_{\mathrm {heat }}$ and $\psi_{\mathrm {wave}\pm}$ are given by
\begin{align}
    \begin{aligned}
       &\psi_{\mathrm{heat}}(v):=\tfrac{1}{\sqrt{K(K-1)}}\left(K-\tfrac{|v|^2}{E}\right) \sqrt{\muq}\,,\\
 &\psi_{ \mathrm{wave}\pm }(\tfrac{\xi}{|\xi|}, v):=\tfrac{1}{\sqrt{2 K}}\left(1 \pm \sqrt{\tfrac{3 K}{E}} \tfrac{\xi}{|\xi|} \cdot v+\tfrac{1}{E}\left(|v|^2-E\right)\right) \sqrt{\muq}\,.
    \end{aligned}\nonumber
\end{align}
\item \textbf{Decomposition.} Moreover, if $\star \neq \star^{\prime}$, then $\mathcal{P}_{\star}^{(0)} \mathcal{P}_{\star^{\prime}}^{(0)}=0$ and the orthogonal projector $\mathbf{P}_0$ onto $\operatorname{Ker} \CL$ satisfies
\begin{align}
    \mathbf{P}_0=\sum_{\star \in\{\mathrm{NS},\, \mathrm{heat},\,\mathrm {wave} \pm\}} \mathcal{P}_{\star}^{(0)}\left(\tfrac{\xi}{|\xi|}\right) .\nonumber
\end{align}
\item \textbf{Decay estimates.} Let $\delta_0:=\min\{\lambda_0,\lambda_1\}$ and $\chi$ be a smooth cutoff function. We denote by the spectral projection
\begin{align}
    \CP(\xi):=\chi(\tfrac{\xi}{\kappa})\sum_{\star\in\{\mathrm{NS,\,heat,\,wave}\pm\}}\CP_{\star}(\xi).\nonumber
\end{align}
Then the $C^0$ semi-group $\widehat{U}^1(t)$ generated by $\widehat{\La}^1$ satisfies that for any $\delta \in\left(0, \delta_0\right)$, $\xi \in \BR$ and $f \in L^2_v$,
\begin{align}
\sup _{t \geqslant 0} e^{2 \delta_0 t} \| \widehat{U}^1(t) & (\Id-\CP(\xi))\widehat{f} \|_{L^2_v}^2\leqslant C_\delta\|(\operatorname{Id}-\CP(\xi)) \widehat{f}\|_{L^2_v}^2. \label{ineq-decay}
\end{align}
    \end{enumerate}
\end{lemma}

The fourth property in Lemma \ref{lem-spectral-thm} implies that we can decompose $U^{\ve}(t)$ into
\begin{align}
U^{\varepsilon}(t)=U^{\varepsilon, \flat}(t)+U^{\varepsilon, \sharp}(t),   \nonumber
\end{align}
where in Fourier variables,
\begin{align}\label{def-U-b-Fourier}
    \widehat{U}^{\varepsilon, \flat}(t, \xi):=\chi\left(\tfrac{\varepsilon|\xi|}{\kappa}\right) \sum_{\star \in\{\mathrm{NS}, \text { heat }, \text { wave}\pm\}} e^{\lambda_{\star}(\varepsilon \xi) \tfrac{t}{\varepsilon^2}} \mathcal{P}_{\star}(\varepsilon \xi),
\end{align}
and also from \eqref{ineq-decay}, it follows that for any $t \geqslant 0$,
\begin{align*}
    \left\|\widehat{U}^{\varepsilon, \sharp}(t, \xi)\right\|_{L_v^2 \rightarrow L_v^2} \lesssim e^{-\tfrac{\delta_0}{2\varepsilon^2}t},
\end{align*}
uniformly for $\xi\in\BR$. In the study of the limit $\varepsilon \rightarrow 0$ of \eqref{BFD-perturbation}, it will be useful to decompose $U^{\varepsilon, \flat}(t)$ into a part independent of $\varepsilon$ and a remainder, which will be shown to go to zero in a sense to be made precise later:
\begin{align*}
    U^{\varepsilon, \flat}=U_{\mathrm{NSF}}+\widetilde{U}_{\mathrm{NSF}}^{\varepsilon}+U_{\mathrm{wave}}^{\varepsilon, \flat},
\end{align*}
where in Fourier variables,
\begin{align}
    \begin{aligned}\label{def-U-NSF-wave-Fourier}
& \widehat{U}_{\mathrm{NSF}}(t, \xi):=e^{-\nu_{\mathrm{NS}}|\xi|^2 t} \mathcal{P}_{\mathrm{NS}}^{(0)}\left(\tfrac{\xi}{|\xi|}\right)+e^{-\nu_{\mathrm {heat }}|\xi|^2 t} \mathcal{P}_{\mathrm {heat }}^{(0)}\left(\tfrac{\xi}{|\xi|}\right), \\
& \widehat{U}_{\mathrm {wave }}^{\varepsilon, \flat}(t, \xi):=\chi\left(\tfrac{\varepsilon|\xi|}{\kappa}\right) \sum_{\star =\mathrm{wave}\pm} e^{\lambda_\star(\varepsilon \xi) \tfrac{t}{\varepsilon^2}} \mathcal{P}_\star(\varepsilon \xi).
\end{aligned}
\end{align}

Similarly, we also have the following decomposition for the bilinear term,
\begin{align}
  &\mathfrak{Q}^{\ve}[f,g](t)=\mathfrak{Q}^{\ve,\flat}[f,g](t)+\mathfrak{Q}^{\ve,\sharp}[f,g](t),\label{decom-bi}
\end{align}
where
\begin{align}
    \begin{aligned}
       &\widehat{\mathfrak{Q}}^{\ve,\flat}[f,g](t):=\ve^{-1}\int_0^t\widehat{U}^{\ve,\flat}(t-s)\widehat{\CQ}_{\text{sym}}(f(s),g(s))(\xi)\ud s,\\
       &\widehat{\mathfrak{Q}}^{\ve,\sharp}[f,g](t):=\ve^{-1}\int_0^t\widehat{U}^{\ve,\sharp}(t-s)\widehat{\CQ}_{\text{sym}}(f(s),g(s))(\xi)\ud s,
    \end{aligned}\nonumber
\end{align}
with $\widehat{U}^{\ve,\flat}$ defined in \eqref{def-U-b-Fourier}. In a same way, the source terms can be decomposed as
\begin{align}
\mathfrak{Q}^{\ve,\flat}[f,g](t)=\mathfrak{Q}_{\mathrm{NSF}}[f,g](t)+\widetilde{\mathfrak{Q}}^{\ve,\flat}_{\mathrm{NSF}}[f,g](t)+\mathfrak{Q}^{\ve,\flat}_{\mathrm{wave}}[f,g](t), \nonumber
\end{align}
where
\begin{align}
    &\widehat{\mathfrak{Q}}_{\mathrm{NSF}}[f,g](t):=\sum_{\star\in\{\mathrm{NS},\mathrm{heat}\}}\int_0^t e^{-\nu_{\star}\left(t-s\right)|\xi|^2}|\xi| \mathcal{P}_{\star}^{(1)}\left(\frac{\xi}{|\xi|}\right) \widehat{\CQ}_{\mathrm {sym }}\left(f\left(s\right), g\left(s\right)\right)(k) \mathrm{d} s,\nonumber\\
       &\widehat{\mathfrak{Q}}^{\ve,\flat}_{\mathrm{wave}}[f,g](t):=\ve^{-1}\int_0^t\widehat{U}^{\ve,\flat}_{\mathrm{wave}\pm}(t-s)\widehat{\CQ}_{\text{sym}}(f(s),g(s))(\xi)\ud s,\nonumber
\end{align}
with $\widehat{U}_{\mathrm{NSF}}(t)$ and $\widehat{U}^{\ve}_{\mathrm{wave}}(t)$ defined in \eqref{def-U-NSF-wave-Fourier}. For the trilinear source term, we can also decompose
\begin{align}
    &\mathfrak{T}^{\ve}[f,g,h](t)=\mathfrak{T}^{\ve,\flat}[f,g,h](t)+\mathfrak{T}^{\ve,\sharp}[f,g,h](t),\label{def-tri-decom}
\end{align}
where
\begin{align}
    &\widehat{\mathfrak{T}}^{\ve,\flat}[f,g,h](\xi,t):=\int_0^t\widehat{U}^{\ve,\flat}(t-s)\widehat{\CT}_{\text{sym}}(f(s),g(s),h(s))(\xi)\ud s,\nonumber\\
       &\widehat{\mathfrak{T}}^{\ve,\sharp}[f,g,h](\xi,t):=\int_0^t\widehat{U}^{\ve,\sharp}(t-s)\widehat{\CT}_{\text{sym}}(f(s),g(s),h(s))(\xi)\ud s.\nonumber
\end{align}
$\mathfrak{T}^{\ve,\flat}[f,g,h](t)$ can be further decomposed into
\begin{align}
\mathfrak{T}^{\ve,\flat}[f,g,h](t)=\widetilde{\mathfrak{T}}^{\ve,\flat}_{\mathrm{NSF}}[f,g,h](t)+\mathfrak{T}^{\ve,\flat}_{\mathrm{wave}}[f,g,h](t), \nonumber
\end{align}
where
\begin{align}
       &\widehat{\mathfrak{T}}^{\ve,\flat}_{\mathrm{wave}}[f,g,h](t):=\int_0^t\widehat{U}^{\ve,\flat}_{\mathrm{wave}}(t-s)\widehat{\CT}_{\text{sym}}(f(s),g(s),h(s))(\xi)\ud s.\nonumber
\end{align}

For the limiting function $g$ constructed in \eqref{def-solution-g}, by using the above notations, it can be verified that
\begin{align}\label{eq-g}
    g(t)=U_{\mathrm{NSF}}(t)g_{\mathrm{in}}+\mathfrak{Q}_{\mathrm{NSF}}[g,g](t).
\end{align}

\subsection{Some Results Concerning the Convergence}
In order to understand the limiting progress from BFD equation \eqref{BFD-perturbation} to NSF equation \eqref{NSF}, we shall recall some classical results related to the solutions of NSF equation.
\begin{lemma}[\cite{gallagher2003asymptotics}]\label{lem-fluid-solution}
    Consider the NSF equation \eqref{NSF} with initial data $(\rho_{\mathrm{in}},u_{\mathrm{in}},\vartheta_{\mathrm{in}})(x)\in H^{\frac{1}{2}}_x$ and $g_{\mathrm{in}}(x,v)\in H^{\frac{1}{2}}_xL^2_v$ being given by \eqref{def-initial-data}. Then there exists a maximal fluid solution $g \in \widetilde{L}_T^{\infty} H_x^{\frac{1}{2}} L_v^2 \cap L_T^2 H_x^{\frac{3}{2}} L_v^2$ for $T<T^{\star}$, where the maximal life span $T^{\star}>0$ satisfies
    \begin{align}\label{blow-up}
        \lim _{T \rightarrow T^{\star}}\|g\|_{L_T^2 H_x^{\frac{3}{2}} L_v^2}=\infty,
    \end{align}
    and
    \begin{align}
        \|g\|_{\tilde{L}_T^{\infty} H_x^{\frac{1}{2}} L_v^2}+\|g\|_{L_T^2 H_x^{\frac{3}{2}} L_v^2} \lesssim\left\|g_{\mathrm{in}}\right\|_{H_x^{\frac{1}{2}} L_v^2},\label{ineq-lower-bound}
    \end{align}
   where the constant may depend on $T^{\star}$ but is uniform if $T^{\star}=\infty$.
\end{lemma}

In \cite{JXZ2022JDE}, the authors show the convergence of $f^\ve$ in $L_t^\infty H_x^3 L_v^2$. However, when the case comes to the lower regularity framework here, there is the possibility of blow-up if we want to upgrade the regularity of the kinetic solutions, by noting \eqref{blow-up}. Thus, we define a smooth function
\begin{align}
    g^{\varepsilon}(t, x, v):=\big\{\rho^{\varepsilon}(t, x)+u^{\varepsilon}(t, x) \cdot v+\theta^{\varepsilon}(t, x)(\tfrac{|v|^2}{2}-K_g)\big\} \sqrt{\muq},\nonumber
\end{align}
 where $(\rho^{\varepsilon}, u^{\varepsilon}, \theta^{\varepsilon})$ solves NSF equation \eqref{NSF} with the initial data $\psi\left(\varepsilon^\alpha\left|D_x\right|\right)\left(\rho_{\mathrm {in }}, u_{\mathrm {in }}, \theta_{\mathrm {in }}\right)$. We then make a detour to show $f^\ve \to g$ in $\widetilde{L}_T^{\infty} H_x^{\frac{1}{2}} L_v^2$. That is, we firstly prove $g^\ve \to g$ then we prove $f^\ve \to g^\ve$, as $\ve\to 0$. To this end, we have the following lemma.
 \begin{lemma}[\cite{gallagher2020convergence}]\label{lem-smooth-fluid-solution}
     Let $T^*$ be given in above lemma. For any $T<T^*$ and $\ve$ sufficiently small, it holds that
     \begin{align}
         \left\|g^{\varepsilon}-g\right\|_{\widetilde{L}_T^{\infty} H_x^{\frac{1}{2}} L_v^2}+\left\|g^{\varepsilon}-g\right\|_{L_T^2 H_x^{\frac{3}{2}} L_\gamma^2} {\longrightarrow} 0,\quad \ve\to 0.\label{limit-smooth-appoximation}
     \end{align}
     Moreover, for $m>\frac{1}{2}$ there are estimates for $g^{\ve}$:
     \begin{align}
\left\|g^{\varepsilon}\right\|_{\widetilde{L}_T^{\infty} H_x^m L_v^2}+\left\|g^{\varepsilon}\right\|_{L_T^2 H_x^{m+1} L_\gamma^2}  \lesssim\left\|\mathbf{P}_0 f_{\mathrm{in}}^{\varepsilon}\right\|_{H_x^m L_v^2} \exp \left(C\left\|g_{\mathrm{in}}\right\|_{H_x^{\frac{1}{2}} L_v^2}^2\right).\nonumber
     \end{align}
In particular,
\begin{align}
  &\left\|g^{\varepsilon}\right\|_{\widetilde{L}_T^{\infty} H_x^m L_v^2}+\left\|g^{\varepsilon}\right\|_{L_T^2 H_x^{m+1} L_\gamma^2} \lesssim \varepsilon^{-\alpha\left(m-\frac{1}{2}\right)}\left\|g_{\mathrm{in}}\right\|_{H_x^{\frac{1}{2}} L_v^2} \exp \left(C\left\|g_{\mathrm{in}}\right\|_{H_x^{\frac{1}{2}} L_v^2}^2\right),\label{ineq-gve-higher-bounds}\\
  &\left\|g^{\varepsilon}\right\|_{\widetilde{L}_T^4 H_x^{m+\frac{1}{2}} L_\gamma^2} \lesssim \varepsilon^{-\alpha\left(m-\frac{1}{2}\right)}\left\|g_{\mathrm{in}}\right\|_{H_x^{\frac{1}{2}} L_v^2} \exp \left(C\left\|g_{\mathrm{in}}\right\|_{H_x^{\frac{1}{2}} L_v^2}^2\right) .\label{ineq-L4-g-ve}
\end{align}
 \end{lemma}
 \begin{proof}
   We refer to Page 12 in \cite{carrapatoso2025navier} for the proof, where we have used that $\PP_0:L^2_v\to L_\gamma^2\subset L^2_v$ is bounded and hence
   \begin{equation}\label{es-kerL-L2v-H*v-ineq}
     \|\gve\|_{L_\gamma^2}=\|\PP_0\gve\|_{L_\gamma^2}\lesssim|\PP_0\gve\|_{L^2_v}.
   \end{equation}
   The estimate \eqref{ineq-L4-g-ve} is a consequence of a simple interpolation inequality, that is,
    \begin{align*}
        \|h\|_{\widetilde{L}_T^4 H_x^n} \lesssim\|h\|_{\widetilde{L}_T^{\infty} H_x^{n-\frac{1}{2}}}^{\frac{1}{2}}\|h\|_{L_T^2 H_x^{n+\frac{1}{2}}}^{\frac{1}{2}}, \; n\geqslant \frac{1}{2}.
    \end{align*}
 \end{proof}

\begin{remark}
    In fact, Lemma \ref{lem-fluid-solution} shows that the lifespan of $g^\ve$ is $[0,T^\ve]$ for some $T^\ve > 0$. However, one can show that the lifespan of $g^\ve$ is actually at least that of the limit $g$. This statement is presented at the very beginning of Section 2 in \cite{gallagher2020convergence}.
\end{remark}

\section{Sketch of the Proof of Theorem \ref{thm-main}}
The aim of this work is to rigorously justify the convergence from the quantum kinetic solution $\fve$ to the fluid solution $g$ in space $\widetilde{L}^{\infty}_TH^{\frac{1}{2}}_xL^2_v$. Since we already have \eqref{limit-smooth-appoximation}, it is enough to show
\begin{align}
 \left\|\fve-g^{\varepsilon}\right\|_{\widetilde{L}_T^{\infty} H_x^{\frac{1}{2}} L_v^2}{\longrightarrow} 0,\quad \ve\to0,\label{aim-reduce-limit}
\end{align}
for the given well-prepared initial data $g_{\mathrm{in}}$ defined by \eqref{def-initial-data}.

Now we define the difference $\delta^{\ve}$ by $\delta^{\ve}(t):= \fve(t) - g^{\ve}(t)$, where $\fve$ satisfies
\begin{align}
    \fve(t)=U^{\ve}(t)\fin^{\ve}+\mathfrak{Q}^{\ve}[f^{\ve},f^{\ve}](t)+\mathfrak{T}^{\ve}[f^{\ve},f^{\ve},f^{\ve}](t), \label{f-Duhamel-form}
\end{align}
and just like \eqref{eq-g}, $g^{\ve}(t)$ also satisfies
\begin{align}
  g^{\ve}(t)=U_{\mathrm{NSF}}(t)\PP_0f^{\ve}_{\mathrm{in}}+\mathfrak{Q}_{\mathrm{NSF}}[g^{\ve},g^{\ve}](t).\label{g-Duhamel-form}
\end{align}
Taking the difference between \eqref{f-Duhamel-form} and \eqref{g-Duhamel-form}, we get the equation for $\delta^{\ve}(t)$: 
\begin{align}\label{different-equation}
\delta^{\ve}(t)=D_0^{\ve}(t)+\mathcal{S}^{\ve}(t)+\mathcal{A}^{\ve}[\delta^{\ve}](t)+\mathcal{B}^{\ve}[\delta^{\ve},\delta^{\ve}](t)+\mathfrak{T}^{\ve}[\delta^{\ve},\delta^{\ve},\delta^{\ve}](t),
\end{align}
where
\begin{align}
    \begin{aligned}
        &D_0^{\ve}(t):=(U^{\ve}(t)-U_{\mathrm{NSF}}(t))\PP_0\fin^{\ve}+U^{\ve}(t)\PP_0^{\perp}\fin^{\ve},\\
        &\mathcal{S}^{\ve}(t):=\mathfrak{Q}^{\ve}[\gve,\gve](t)-\mathfrak{Q}_{\NSF}[\gve,\gve](t)+\mathfrak{T}^{\ve}[\gve,\gve,\gve](t),\\
        &\mathcal{A}^{\ve}[\delta^{\ve}](t):=2\mathfrak{Q}^{\ve}[\delta^{\ve},\gve](t)+3\mathfrak{T}^{\ve}[\delta^{\ve},\gve,\gve](t),\\
        &\mathcal{B}^{\ve}[\delta^{\ve},\delta^{\ve}](t): =\mathfrak{Q}^{\ve}[\delta^{\ve},\delta^{\ve}](t)+3\mathfrak{T}^{\ve}[\delta^{\ve},\delta^{\ve},\gve](t).
    \end{aligned}\label{def-diffi-eq-part}
\end{align}

Then we should prove the difference equation \eqref{different-equation} has a unique solution and the solution $\delta^{\ve}\to 0$ as $\ve\to 0$. For the existence, we have the following fixed point lemma.
\begin{lemma}\label{lem-fixed-point}
    Let $\mathscr{X}$ be a Banach space, $L:\mathscr{X}\to\mathscr{X}$ be a continuous linear operator, $B:\mathscr{X}\times\mathscr{X}\to\mathscr{X}$ be a bilinear operator, and $\Psi:\mathscr{X}\times\mathscr{X}\times\mathscr{X}\to\mathscr{X}$ be a trilinear operator. Define norms of operators:
    \begin{align}
        &\|L\|:=\sup_{\|y\|=1}\|L(y)\|,\quad \|B\|:=\sup_{\|y\|=\|z\|=1}\|B(y,z)\|,\nonumber\\
       & \|\Psi\|:=\sup_{\|y\|=\|z\|=\|u\|=1}\|\Psi(y,z,u)\|.\nonumber
    \end{align}
    If $\|L\|<1/2$, then for any $y_0\in\mathscr{X}$ such that
    \begin{equation*}
    \|y_0\|\leqslant \frac{1/2-\|L\|}{4\|B\| + \sqrt{6\|\Psi\|}},
  \end{equation*}
  the equation $y=y_0+L(y)+B(y,y)+\Psi(y,y,y)$ admits a unique solution in the ball of center $0$ and radius $\frac{1-2\|L\|}{4\|B\| + \sqrt{6\|\Psi\|}}$. In addition, there holds that $\|y\|\leqslant C\|y_0\|$ for some constant $C>0$.
\end{lemma}

\begin{proof}
  For any fixed $y_0\in\mathscr{X}$, we define the map $G:\,\mathscr{X}\to \mathscr{X}$ as
  \begin{equation}
    G(y)=y_0+L(y)+B(y,y)+\Psi(y,y,y).\nonumber
  \end{equation}

  Let $s>0$ and $\|y_0\|\leqslant s/2$, then for any $y\in \mathscr{X}$ with $\|y\|\leqslant s$ we have
  \begin{equation*}
    \|G(y)\| \leqslant \frac{s}{2} + \|L\|s + \|B\|s^2 + \|\Psi\|s^3.
  \end{equation*}
  It is easy to see that for any $s>0$ with
  \begin{equation*}
    s \leqslant \frac{1-2\|L\|}{2\|B\| + \sqrt{2\|\Psi\|}},
  \end{equation*}
  one can get $\|G(y)\|\leqslant s$. Then $G$ maps $B_0(s)$ to $B_0(s)$ for $s$ satisfying the above condition.

  Next for any $y,\,z\in B_0(r)$ with $r\in(0, \frac{1-2\|L\|}{2\|B\| + \sqrt{2\|\Psi\|}}]$ to be determined, we have
  \begin{align*}
    G(y) - G(z) = &L(y-z) + B(y-z,y) + B(z,y-z) \\
    &+ \Psi(y-z,y,y) + \Psi(z,y-z,y) +\Psi(z,z,y-z),
  \end{align*}
  it follows that
  \begin{align*}
    \|G(y) - G(z)\| \leqslant & \Big\{\|L\| + \|B\|(\|y\| + \|z\|) +\|\Psi\|(\|y\|^2 + \|y\|\|z\| + \|z\|^2)\Big\} \|y-z\|\\
    \leqslant & (\|L\| + 2\|B\| r + 3\|\Psi\| r^2)\|y-z\|,
  \end{align*}
  Now we choose $r>0$ with
  \begin{equation*}
    r \leqslant \frac{1-2\|L\|}{4\|B\| + \sqrt{6\|\Psi\|}} < \frac{1-2\|L\|}{2\|B\| + \sqrt{2\|\Psi\|}},
  \end{equation*}
  to get $\|L\| + 2\|B\| r + 3\|\Psi\| r^2\leqslant 1/2$ and hence $G$ is a contraction map on a ball $B_0(r)$. We complete the proof by the fixed point theorem.
\end{proof}

To progress, for any interval $I\subset\mathbb{R}^+$ we define an appropriate functional space $\mathscr{X}_I^{\ve}$ as
\begin{align}
\mathscr{X}_I^{\ve}:=\left\{f(t,x,v) \mid \|f\|_{\mathscr{X}_I^{\ve}}<+\infty\right\},\label{def-space-MXe}
\end{align}
though the norm
\begin{align}
    \begin{aligned}
\|f\|_{\mathscr{X}_I^{\ve}}:= & \|f\|_{\widetilde{L}_I^{\infty} H_x^{\frac{1}{2}} L_v^2}+\left\|\mathbf{P}_0 f\right\|_{L_I^2 \dot{H}_x^{\frac{3}{2}} L_v^{2}}+\frac{1}{\sqrt{\varepsilon}}\left\|\mathbf{P}_0^{\perp} f\right\|_{L_I^2 H_x^{\frac{3}{2}} L_\gamma^{2}} \\
& +\varepsilon^\beta\left(\|f\|_{\widetilde{L}_I^{\infty} H_x^{\ell} L_v^2}+\left\|\mathbf{P}_0 f\right\|_{L_I^2 \dot{H}_x^{\ell} L_v^{ 2}}+\tfrac{1}{\sqrt{\varepsilon}}\left\|\mathbf{P}_0^{\perp} f\right\|_{L_I^2 H_x^{\ell} L_\gamma^{2}}\right).
\end{aligned}\nonumber
\end{align}
In order to apply the fixed point argument described above to prove the existence of \eqref{different-equation} in the functional space $\mathscr{X}_I^{\ve}$, it suffices to establish the following estimates, which are proved in the next section.

\begin{proposition}\label{prop-fixed-point-lem-conditions}
Recall the notations $D^{\ve}_0,\CS^{\ve},\CA^{\ve},\CB^{\ve}$ given by \eqref{def-diffi-eq-part}. Under the assumptions of Theorem \ref{thm-main}, the following estimates hold:
\begin{enumerate}
    \item For any $t\in(0,T)$ and any function $F\in \mathscr{X}_{[0, t]}^{\varepsilon}$, we have
    \begin{align*}
        \left\|U^{\varepsilon}(\cdot-t) F(t)\right\|_{\mathscr{X}_{[t, T]}^{\varepsilon}} \lesssim\|F\|_{\mathscr{X}_{[0, t]}^{\varepsilon}}.
    \end{align*}
     \item For $D_0^{\ve}(t)$, it holds that
        \begin{align}
            \|D_0^{\ve}\|_{\MXe_{\infty}}\lesssim D^{\ve}_{\mathrm{in}}\to 0,\quad\ve\to 0,\nonumber
        \end{align}
        where
        \begin{align*}
          D^{\ve}_{\mathrm{in}}:=\ve^{\frac{1}{2}-\al}\|g_{\mathrm{in}}\|_{H^{\frac{1}{2}}_xL^2_v}+\|\PP_0^{\perp}\fin^{\ve}\|_{H^{\frac{1}{2}}_xL^2_v}+\ve^{\be}\|\PP_0^{\perp}\fin^{\ve}\|_{H^{\ell}_xL^2_v}.
        \end{align*}
    \item For the source term $\CS^{\ve}(t)$, there exists a nonnegative increasing function $\Phi$ such that
    \begin{align}
       \left\|\CS^{\varepsilon}\right\|_{\mathscr{X}_T^{\varepsilon}} \lesssim \ve^{\frac{1}{2}-2\al}\Phi\big( \left\|g_{\mathrm {in }}\right\|_{H_x^{1/2} L_v^2}\big)\to 0,\quad\ve\to0 .\nonumber
    \end{align}
    \item For all intervals $I$, the linear operator $\CA^{\ve}$ satisfies that for sufficiently small $\ve>0$,
    \begin{align}
        \begin{aligned}
           \left\|\mathcal{A}^{\varepsilon}[f]\right\|_{\mathscr{X}_I^{\varepsilon}} &\lesssim\|f\|_{\mathcal{X}_I^{\varepsilon}}(\left\|g^{\varepsilon}\right\|_{\widetilde{L}_I^4 H_x^1 L_v^2}+\left\|g^{\varepsilon}\right\|_{L_I^2 H_x^{\frac{3}{2}} L_v^2}+\varepsilon^\beta\left\|g^{\varepsilon}\right\|_{\widetilde{L}_I^{\infty} H_x^{\ell} L_v^2}+\varepsilon^\beta\left\|g^{\varepsilon}\right\|_{L_I^2 H_x^{\ell} L_v^2})\\
        &\quad+ \sqrt{\ve}\|f\|_{\MXe_I}\|\gve\|_{\widetilde{L}^{\infty}_IH^{\frac{1}{2}}_xL^2_v}(\|\gve\|_{\widetilde{L}^{\infty}_IH^{\frac{5}{2}}_xL^2_v}+\|\gve\|_{L^2_IH^{\frac{7}{2}}_xL^2_v})\\
 &\quad+\sqrt{\ve}\|f\|_{\MXe_I}(\|\gve\|_{L^2_IH^{\frac{3}{2}}_xL^2_v}\|\gve\|_{\widetilde{L}^{\infty}_IH^{\frac{5}{2}}_xL^2_v}+\|\gve\|_{L^2_IH^{\ell}_xL^2_v}^2)\\
  &\quad+\sqrt{\ve}\|f\|_{\MXe_I}\|\gve\|_{\widetilde{L}^{\infty}_IH^{\ell}_xL^2_v}(\|\gve\|_{L^2_IH^{\ell+1}_xL^2_v}+\|\gve\|_{\widetilde{L}^{\infty}_IH^{\ell}_xL^2_v})\\
  &\quad+\sqrt{\ve}\|f\|_{\MXe_I}\|\gve\|_{\widetilde{L}^{\infty}_IH^{1}_xL^2_v}^2.
        \end{aligned}\label{linear-CA-estimate}
    \end{align}
   \item For all intervals $I$, the bilinear and trilinear operators are continuous:
   \begin{align}
      & \|\mathfrak{Q}^{\ve}[f_1,f_2]\|_{\MXe_I}\lesssim\|f_1\|_{\MXe_I}\|f_2\|_{\MXe_I},\label{bi-non-esti}\\
      &\|\mathfrak{T}^\ve[f_1,f_2,f_3]\|_{\MXe_I}\lesssim\|f_1\|_{\MXe_I}\|f_2\|_{\MXe_I}\|f_3\|_{\MXe_I}\label{tri-non-esti}.
   \end{align}
\end{enumerate}
\end{proposition}

With the above tools and estimates at hand, we now prove our main result.
\begin{proof}[\bf Proof of Theorem \ref{thm-main}]
Applying Lemma \ref{lem-fixed-point} to our case, we set there that
\begin{equation*}
    y=\delta^\ve \textrm{ and } y_0=D_0^{\ve}(t)+\CS^{\ve}(t).
\end{equation*}
Then from the second and third statements in Proposition \ref{prop-fixed-point-lem-conditions} we know $y_0$ is bounded and $y_0\to0$, as $\ve\to0$. The estimates in Proposition \ref{prop-fixed-point-lem-conditions}-(5) correspond to the bilinear and trilinear conditions in Lemma \ref{lem-fixed-point}. It remains to verify the contraction of the linear operator $\CA^{\ve}$. 

For the scaling parameter $\ve$, we can choose $\ve_0>0$ such that for any $\ve\leqslant\ve_0$,
\begin{align}
    \|\CA^{\ve}[f]\|_{\MXe_I}&\lesssim\Big(\ve^{\be-\al(\ell-\frac{1}{2})}+\ve^{\be-\al(\ell-\frac{3}{2})}+\ve^{\frac{1}{2}-2\al}+\ve^{\frac{1}{2}-2\al(\ell-\frac{1}{2})}+\ve^{\frac{1}{2}-\al}\Big)\|f\|_{\MXe_I}\nonumber\\
    &\quad+\Big(\left\|g\right\|_{\widetilde{L}_I^4 H_x^1 L_v^2}+\left\|g\right\|_{L_I^2 H_x^{\frac{3}{2}} L_v^2}\Big)\|f\|_{\MXe_I}\nonumber\\
    &\leqslant C\Big(\tfrac{1}{8C}+\underbrace{\left\|g\right\|_{\widetilde{L}_I^4 H_x^1 L_v^2}+\left\|g\right\|_{L_I^2 H_x^{\frac{3}{2}} L_v^2}}_{=: G}\Big)\|f\|_{\MXe_I},\label{CA-esti}
\end{align} 
where we have used the estimate \eqref{linear-CA-estimate}. It should be pointed out that the term $G$ is explicitly bounded as in Lemma \ref{lem-smooth-fluid-solution} and clearly, $G$ is not small. To obtain the contraction of $\mathcal{A}^\ve$, we estimate the term $G$ by an irrational method on time, that is to show $G\leqslant \tfrac{1}{8C}$ on each small time interval but not the whole interval $[0,T]$. 

According to Lemma \ref{lem-smooth-fluid-solution}, there exists $t_1$ with $0<t_1<T$ such that
\begin{align}
  \left\|g\right\|_{\widetilde{L}^4([0,t_{1}];H_x^1 L_v^2)}+\left\|g\right\|_{L^2([0,t_{1}]; H_x^{\frac{3}{2}} L_v^2)}\lesssim\tfrac{1}{8C}.\label{G-esti}
\end{align}
 Inserting \eqref{G-esti} into \eqref{CA-esti} for $t\in[0,t_1]$, we obtain
 \begin{align}
    \|\CA^{\ve}[f]\|_{\MXe_{t_1}}\lesssim\tfrac{1}{4}\|f\|_{\MXe_{t_1}}.
 \end{align}
Then by the fixed-point Lemma \ref{lem-fixed-point}, there exists a unique solution $\delta^{\ve}$ to equation
\eqref{different-equation} on $[0,t_1]$ satisfying
\begin{align}
    \|\delta^{\ve}\|_{\MXe_{t_1}}\lesssim\|D_0^{\ve}\|_{\MXe_{t_1}}+\|\CS^{\ve}\|_{\MXe_{t_1}}\to 0,\quad\ve\to0.\label{delta-t2}
\end{align}

We now proceed to establish the irrational relationship. In view of the nonlinear terms in \eqref{different-equation}, by the property of semi-group, then for any $t\geqslant t_1$ we can rewrite them as
\begin{align}
\begin{aligned}
      &\int_0^t U^{\ve}(t-s)\CQ_{\mathrm{sym}}(f_1,f_2)(s)\ud s\\
      =&\int_0^{t_1} U^{\ve}(t-s)\CQ_{\mathrm{sym}}(f_1,f_2)(s)\ud s+\int_{t_1}^t U^{\ve}(t-s)\CQ_{\mathrm{sym}}(f_1,f_2)(s)\ud s\\
   =&U^{\ve}(t-t_1)\int_0^{t_1} U^{\ve}(t_1-s)\CQ_{\mathrm{sym}}(f_1,f_2)(s)\ud s+\underbrace{\int_{t_1}^t U^{\ve}(t-s)\CQ_{\mathrm{sym}}(f_1,f_2)(s)\ud s}_{=:\mathfrak{Q}^{\ve}[f_1,f_2](t_1;t)}.
\end{aligned}
\end{align}
Similarly,
\begin{align}
\begin{aligned}
  &\int_0^t U^{\ve}(t-s)\CT_{\mathrm{sym}}(f_1,f_2,f_3)(s)\ud s\\
  =&U^{\ve}(t-t_1)\int_0^{t_1} U^{\ve}(t_1-s)\CT_{\mathrm{sym}}(f_1,f_2,f_3)(s)\ud s+\underbrace{\int_{t_1}^t U^{\ve}(t-s)\CT_{\mathrm{sym}}(f_1,f_2,f_3)(s)\ud s}_{=:\mathfrak{T}^{\ve}[f_1,f_2,f_3](t_1;t)}.
\end{aligned}\nonumber
\end{align}
With this notation, the equation \eqref{different-equation} can be written as
\begin{align*}
    \delta^{\ve}(t)=&U^{\ve}(t-t_1)\delta^{\ve}(t_1)+D_{02}^{\ve}(t)+\CS^{\ve}_2(t)+\CA^{\ve}[\delta^{\ve}](t_1;t)\\
    &+\CB^{\ve}[\delta^{\ve},\delta^{\ve}](t_1;t)+\mathfrak{T}^{\ve}[\delta^{\ve},\delta^{\ve},\delta^{\ve}](t_1;t),
\end{align*}
where
\begin{align*}
 D_{02}^{\ve}(t)&:=D_{0}^{\ve}(t)-U^{\ve}(t-t_1)D_{0}^{\ve}(t_1),\\
 \CS^{\ve}_2(t)&:=\CS^{\ve}(t)-U^{\ve}(t-t_1)\CS^{\ve}(t_1).
\end{align*}
Again by Lemma \ref{lem-smooth-fluid-solution}, we can choose $t_2\in(t_1,T)$ such that
\begin{align*}
  \left\|g\right\|_{\widetilde{L}^4([t_{1},t_2];H_x^1 L_v^2)}+\left\|g\right\|_{L^2([t_{1},t_2]; H_x^{\frac{3}{2}} L_v^2)}\lesssim\tfrac{1}{8C}.
\end{align*}
Then Proposition \ref{prop-fixed-point-lem-conditions} (1)-(3) and \eqref{delta-t2} imply
\begin{align}
\begin{aligned}
    & \|U^{\ve}(\cdot-t_1)\delta^{\ve}(t_1)+D_{02}^{\ve}(\cdot)+\CS^{\ve}_2(\cdot)\|_{\MXe_{[t_1,t_2]}}\\
    \lesssim&\|\delta^{\ve}\|_{\MXe_{t_1}}+\|D^{\ve}_0\|_{\MXe_{t_1}}+\|\CS^{\ve}\|_{\MXe_{t_1}}\\
  \lesssim&\|D^{\ve}_0\|_{\MXe_{t_1}}+\|\CS^{\ve}\|_{\MXe_{t_1}}\\
  \leqslant& C_1\Big(D_{\mathrm{in}}^{\ve}+\ve^{\frac{1}{2}-2\al}\Phi\big( \left\|g_{\mathrm {in }}\right\|_{H_x^{1/2} L_v^2}\big)\Big)\to0,\quad\ve\to 0,
\end{aligned}\label{delta-t2-esti}
  \end{align}
where constant  $C_1>0$. By Proposition \ref{prop-fixed-point-lem-conditions}-(4), we have
\begin{align}
    \|\CA^\ve[\delta^{\ve}]\|_{\MXe_{[t_1,t_2]}}\leqslant\frac{1}{4}\|\delta^{\ve}\|_{\MXe_{[t_1,t_2]}}.\nonumber
\end{align}
From Lemma \ref{lem-fixed-point} and \eqref{delta-t2-esti}, there exists a positive constant $C_2$ such that
\begin{align}
\begin{aligned}
    \|\delta^{\ve}\|_{\MXe_{[t_1,t_2]}}&\leqslant C_2\|U^{\ve}(\cdot-t_1)\delta^{\ve}(t_1)+D_{02}^{\ve}(\cdot)+\CS^{\ve}_2(\cdot)\|_{\MXe_{[t_1,t_2]}}\\
    &\leqslant C_2C_1\Big(D_{\mathrm{in}}^{\ve}+\ve^{\frac{1}{2}-2\al}\Phi\big( \left\|g_{\mathrm {in }}\right\|_{H_x^{1/2} L_v^2}\big)\Big)\to0,\quad\ve\to 0.
\end{aligned}\nonumber
\end{align}

We can repeat the process above $N$ times for some $N\in\mathbb{Z}^+$ such that $0:=t_0<t_1<\cdots<t_N:= T$ with
\begin{align*}
  \left\|g\right\|_{\widetilde{L}^4([t_{N-1},t_{N}];H_x^1 L_v^2)}+\left\|g\right\|_{L^2([t_{N-1},t_N]; H_x^{\frac{3}{2}} L_v^2)}\lesssim\tfrac{1}{8C},
\end{align*}
and
\begin{equation*}
    \|\CA^\ve[\delta^{\ve}]\|_{\MXe_{[t_i,t_{i+1}]}}\leqslant\frac{1}{4}\|\delta^{\ve}\|_{\MXe_{[t_i,t_{i+1}]}},\;i = 0,1,2,\cdots,N-1.
\end{equation*}
In summary, there exists a unique solution $\delta^{\ve}(t)$ to the difference equation \eqref{different-equation} and a positive constant $C(N):=\prod_{i=1}^NC_i$ such that 
\begin{align*}
    \|\delta^{\ve}\|_{\MXe_T}\leqslant C(N)\Big(\|D_0^{\ve}\|_{\MXe_T}+\|\CS^{\ve}\|_{\MXe_T}\Big)\to 0,\quad\ve\to0.
\end{align*}
The convergence follows from Proposition \ref{prop-fixed-point-lem-conditions} (2)-(3) and thus the proof of Theorem \ref{thm-main} is complete.
\end{proof}
\begin{remark}
    Note that in \eqref{CA-esti}, $G$ is not small enough to guarantee the contraction of operator $\mathcal{A}^\ve$ on $\mathcal{X}_I^\ve$, unless the initial data $g_{\mathrm{in}}$ is sufficiently small. By dividing the interval $I = [0,T]$ into finite segments, one can get around this difficulty if there is no smallness assumption on $g_{\mathrm{in}}$. This idea comes from Carrapatoso-Gallagher-Tristani \cite{carrapatoso2025navier}.
\end{remark}

\section{Proof of Proposition \ref{prop-fixed-point-lem-conditions}}
In this section, we prove Proposition \ref{prop-fixed-point-lem-conditions}. The key ingredient in the proof is the nonlinear estimates for bilinear operator $\mathfrak{Q}^{\ve}$. In subsections \ref{subsec-es-QT}-\ref{subsec-Rfdes-QT}, we give the estimates for bilinear operator $\CQ$ and trilinear operator $\CT$. These estimates will lead to the continuity estimates of semigroup $U^\ve$ and convergence of $D_0^{\ve}$, that is, Proposition \ref{prop-fixed-point-lem-conditions}-(1) and (2). In subsection \ref{subsec-conv-S}, we obtain the convergence of $\CS^{\ve}$ and hence prove Proposition \ref{prop-fixed-point-lem-conditions}-(3). In the last three subsections, we show that $\CA^{\ve}$ and $\CB^{\ve}$ are continuous on functional space $\MXe_I$, and then Proposition \ref{prop-fixed-point-lem-conditions}-(4) and (5) are proved.

\subsection{\texorpdfstring{Estimates for $\mathcal{Q}$ and $\mathcal{T}$}{Estimates for Q and T}}\label{subsec-es-QT}
We now give a class of useful estimates for the bilinear terms $\CQ$ and the trilinear terms $\CT$. Firstly, we invoke the $L_v^2$ estimates, which were proved in \cite{JXZ2022JDE,JiangZhou2024Global}.

\begin{lemma}
  For $f_1,\,f_2,\,f_3$ and $h$ in $D(\CL)$, we have
  \begin{align*}
        |\langle\CQ(f_1,f_2),h\rangle_{L^2_v}|\lesssim(\|f_1\|_{L^2_v}\|f_2\|_{L_\gamma^2}+\|f_2\|_{L^2_v}\|f_1\|_{L_\gamma^2})\|h\|_{L_\gamma^2},
    \end{align*}
    and
    \begin{align*}
        |\langle\CT(f_1,f_2,f_3),h\rangle_{L^2_v}|\lesssim(\|f_1\|_{L^2_v}\|f_2\|_{L_\gamma^2}+\|f_2\|_{L^2_v}\|f_1\|_{L_\gamma^2})\|f_3\|_{L^2_v} \|h\|_{L_\gamma^2}.
    \end{align*}
    As a result,
    \begin{equation}\label{es-Q}
      \begin{split}
      \|\CQ(f_1,f_2)\|_{(L_\gamma^2)^{\prime}}:&=\sup_{\|\phi\|_{L_\gamma^2}\leqslant1}|\langle\CQ(f_1,f_2),\phi\rangle_{L^2_v}|\\
      &\lesssim \|f_1\|_{L^2_v}\|f_2\|_{L_\gamma^2}+\|f_2\|_{L^2_v}\|f_1\|_{L_\gamma^2},
      \end{split}
    \end{equation}
    and
    \begin{equation}\label{es-T}
      \begin{split}
      \|\CT(f_1,f_2,f_3)\|_{(L_\gamma^2)^{\prime}}&:=\sup_{\|\phi\|_{L_\gamma^2}\leqslant1}|\langle\CT(f_1,f_2,f_3),\phi\rangle_{L^2_v}|\\
      &\lesssim(\|f_1\|_{L^2_v}\|f_2\|_{L_\gamma^2}+\|f_2\|_{L^2_v}\|f_1\|_{L_\gamma^2})\|f_3\|_{L^2_v}.
      \end{split}
    \end{equation}
\end{lemma}

In the following, we define for any $a\in\mathbb{R}$,
\begin{equation}\label{def-max}
  (a)_+:=\max\{a,0\}.
\end{equation}
Based on the above lemma, we have the following estimates for $\mathcal{Q}$ and $\mathcal{T}$ in functional spaces $L_T^2 H_x^m\left(L_\gamma^2\right)^{\prime}$ and $\widetilde{L}^{\infty}_T H^m_x L^2_v$.
\begin{lemma}\label{lem-bilinear-ope}
   Let $m\geqslant0$. For any $r_1, r_2 \neq\frac{3}{2}$, it holds that
   \begin{enumerate}
       \item  For any $p_1, q_1, p_2, q_2 \in[1, \infty]$ such that $\frac{1}{p_1}+\frac{1}{q_1}=\frac{1}{p_2}+\frac{1}{q_2}=\frac{1}{2}$, we have
       \begin{equation}\label{es-Q-f1f2}
         \begin{split}
            &\left\|\CQ\left(f_1, f_2\right)\right\|_{L_T^2 H_x^m\left(L_\gamma^2\right)^{\prime}} \\
            \lesssim& \left\|f_1\right\|_{\widetilde{L}_T^{p_1} H_x^{m+\left(\frac{3}{2}-r_1\right)_{+}} L_v^2}\left\|f_2\right\|_{\widetilde{L}_T^{q_1} H_x^{r_1} L_\gamma^2} +\left\|f_1\right\|_{\widetilde{L}_T^{p_2} H_x^{r_2} L_v^2}\left\|f_2\right\|_{\widetilde{L}_T^{q_2} H_x^{m+\left(\frac{3}{2}-r_2\right)_{+}} L_\gamma^2}.
         \end{split}
       \end{equation}
\item  In particular,
\begin{equation}\label{es-Q-Pf1f2-1}
  \begin{split}
     \left\| \CQ(\mathbf{P}_0 f_1, \mathbf{P}_0 f_2) \right\|_{\widetilde{L}^{\infty}_T H^m_x L^2_v}
     &\lesssim \left\| \mathbf{P}_0 f_1 \right\|_{\widetilde{L}^{\infty}_T H^{r_1}_x L^2_v} \left\| \mathbf{P}_0 f_2 \right\|_{\widetilde{L}^{\infty}_T H^{m + (\frac{3}{2} - r_1)_+}_x L^2_v} \\
     &\quad + \left\| \mathbf{P}_0 f_1 \right\|_{\widetilde{L}^{\infty}_T H^{r_2}_x L^2_v}\left\| \mathbf{P}_0 f_2 \right\|_{\widetilde{L}^{\infty}_T H^{m + (\frac{3}{2} - r_2)_+}_x L^2_v}.
  \end{split}
\end{equation}
\item In addition, for any $a_1, b_1, a_2, b_2 \in [1, \infty]$ such that
\begin{align}
\tfrac{1}{a_1} + \tfrac{1}{b_1} = \tfrac{1}{a_2} + \tfrac{1}{b_2} = \tfrac{1}{4},\nonumber
\end{align}
there holds
\begin{equation}\label{es-Q-Pf1f2-2}
  \begin{split}
     \left\| \CQ(\mathbf{P}_0 f_1, \mathbf{P}_0 f_2) \right\|_{\widetilde{L}^{4}_T H^m_x L^2_v}
&\lesssim \left\| \mathbf{P}_0 f_1 \right\|_{\widetilde{L}^{a_1}_T H^{m + (\frac{3}{2} - r_1)_+}_x L^2_v} \left\| \mathbf{P}_0 f_2 \right\|_{\widetilde{L}^{b_1}_T H^{r_1}_x L^2_v} \\
&\quad +\left\| \mathbf{P}_0 f_1 \right\|_{\widetilde{L}^{a_2}_T H^{r_2}_x L^2_v} \left\| \mathbf{P}_0 f_2 \right\|_{\widetilde{L}^{b_2}_T H^{m + (\frac{3}{2} - r_2)_+}_x L^2_v}.
  \end{split}
\end{equation}
   \end{enumerate}
\end{lemma}
\begin{proof}
    Since we have \eqref{es-Q}, the inequality \eqref{es-Q-f1f2} is the integral version (in $x$) of the conclusion of Proposition 4.7 in \cite{carrapatoso2025navier} and the inequality \eqref{es-Q-Pf1f2-1} and \eqref{es-Q-Pf1f2-2} are the integral versions (in $x$) of the conclusion of Proposition 4.8 in \cite{carrapatoso2025navier}. The details of proof are similar to those of next lemma.
\end{proof}
Now we treat the trilinear term $\CT$. We give a similar estimate for $\CT$ to those in \eqref{es-Q-f1f2}.

\begin{lemma}\label{lem-tri-estimates}
 Assume that  $m\geqslant0$ and the functions $f_1\,,f_2\,,f_3$ are sufficiently smooth. Define
 \begin{align}
 \begin{aligned}
  &m_1:=m+(\tfrac{3}{2}-r_{1})_+,\\
        &m_{2}:=m+(\tfrac{3}{2}-r_{2})_++(\tfrac{3}{2}-\delta_{2})_+,\\
        &m_3:=m+(\tfrac{3}{2}-r_{3})_+,\\
         &m_{4}:=m+(\tfrac{3}{2}-r_{4})_++(\tfrac{3}{2}-\delta_{4})_+.
 \end{aligned}
        \end{align}
 Suppose that $r_i,\delta_i\neq\frac{3}{2},\,i=1,\cdots,4$,  for any $\al_j,\be_j,\ga_j \in[1, \infty],\,j=1,\cdots,4$, such that
     \begin{align*}
         &\tfrac{1}{\alpha_j}+\tfrac{1}{\beta_j}+\tfrac{1}{\gamma_j}=\tfrac{1}{2},\quad j=1,\cdots,4.
     \end{align*}
     Then
     \begin{align}
     \begin{aligned}
      &\left\|\CT\left(f_1, f_2,f_3\right)\right\|_{L_T^2 H_x^m\left(L_\gamma^2\right)^{\prime}} \\
      \lesssim& \|f_1\|_{\widetilde{L}^{\alpha_1}_TH^{m_1}_xL^2_v}\|f_2\|_{\widetilde{L}^{\beta_1}_TH^{r_1+(\frac{3}{2}-\delta_1)_+}_xL^2_v} \|f_3\|_{\widetilde{L}^{\gamma_1}_TH^{\delta_1}_x L_\gamma^2} \\
&\quad+\|f_1\|_{\widetilde{L}^{\alpha_2}_TH^{r_2}_xL^2_v}\|f_2\|_{\widetilde{L}^{\beta_2}_TH^{m_2}_xL^2_v} \|f_3\|_{\widetilde{L}^{\gamma_2}_TH^{\delta_2}_x L_\gamma^2},\\
            &\quad+\|f_1\|_{\widetilde{L}^{\alpha_3}_TH^{m_3}_xL^2_v}\|f_2\|_{\widetilde{L}^{\beta_3}_TH^{\delta_3}_xL^2_v} \|f_3\|_{\widetilde{L}^{\gamma_3}_TH^{r_3+(\frac{3}{2}-\delta_3)_+}_x L_\gamma^2},\\
            &\quad+\|f_1\|_{\widetilde{L}^{\alpha_4}_TH^{r_4}_xL^2_v}\|f_2\|_{\widetilde{L}^{\beta_4}_TH^{\delta_4}_xL^2_v} \|f_3\|_{\widetilde{L}^{\gamma_4}_TH^{m_4}_x L_\gamma^2}.
     \end{aligned}\label{ineq-tri-L2}
     \end{align}
In addition, for $m>0$, it holds that
  \begin{align}\label{ineq-tri-homo-L2}
     \begin{aligned}
      &\left\|\CT\left(f_1, f_2,f_3\right)\right\|_{L_T^2 \dot H_x^m\left(L_\gamma^2\right)^{\prime}} \\
      \lesssim& \|f_1\|_{\widetilde{L}^{\alpha_1}_T\dot H^{m_1}_xL^2_v}\|f_2\|_{\widetilde{L}^{\beta_1}_T\dot H^{r_1+(\frac{3}{2}-\delta_1)_+}_xL^2_v} \|f_3\|_{\widetilde{L}^{\gamma_1}_T\dot H^{\delta_1}_x L_\gamma^2} \\
&\quad+\|f_1\|_{\widetilde{L}^{\alpha_2}_T\dot H^{r_2}_xL^2_v}\|f_2\|_{\widetilde{L}^{\beta_2}_T\dot H^{m_2}_xL^2_v} \|f_3\|_{\widetilde{L}^{\gamma_2}_T\dot H^{\delta_2}_x L_\gamma^2},\\
            &\quad+\|f_1\|_{\widetilde{L}^{\alpha_3}_T\dot H^{m_3}_xL^2_v}\|f_2\|_{\widetilde{L}^{\beta_3}_T\dot H^{\delta_3}_xL^2_v} \|f_3\|_{\widetilde{L}^{\gamma_3}_T\dot H^{r_3+(\frac{3}{2}-\delta_3)_+}_x L_\gamma^2},\\
            &\quad+\|f_1\|_{\widetilde{L}^{\alpha_4}_T\dot H^{r_4}_xL^2_v}\|f_2\|_{\widetilde{L}^{\beta_4}_T\dot H^{\delta_4}_xL^2_v} \|f_3\|_{\widetilde{L}^{\gamma_4}_T\dot H^{m_4}_x L_\gamma^2}.
     \end{aligned}
     \end{align}
\end{lemma}

\begin{proof}
    For simplicity, we denote by $F_1(t,\xi):=\|\widehat{f}_1(t,\xi)\|_{L^2_v}$, $F_2(t,\xi):=\|\widehat{f}_2(t,\xi)\|_{L^2_v}$, and $F_3(t,\xi):=\|\widehat{f}_3(t,\xi)\|_{L_\gamma^2}$.
    From the arguments in \cite[Lemma 7.3]{Lemarie2016NS}, we have
    \begin{align}
        &\int_{\BR}F_2(\eta-\zeta)F_3(\zeta)\ud\zeta\nonumber\\
        \lesssim&\int_{|\zeta|<|\eta-\zeta|}F_2(\eta-\zeta)F_3(\zeta)\ud\zeta+\int_{|\zeta|<|\eta-\zeta|}F_2(\zeta)F_3(\eta-\zeta)\ud\zeta\nonumber\\
        =:&J_{23}(\eta)+J^{\prime}_{23}(\eta).\nonumber
    \end{align}
    Applying the arguments in \cite[Lemma 7.3]{Lemarie2016NS} once more to the above inequality, we derive
    \begin{align}
    \begin{aligned}
     &\int_{\BR}F_1(t,\xi-\eta)\int_{\BR}F_2(\eta-\zeta)F_3(\zeta)\ud\zeta\ud\eta\\
     \lesssim&\int_{\BR}F_1(t,\xi-\eta)(J_{23}(\eta)+J^{\prime}_{23}(\eta))\ud\eta\\
        \lesssim&\int_{|\eta|<|\xi-\eta|}F_1(t,\xi-\eta)J_{23}(\eta)\ud\eta+\int_{|\eta|<|\xi-\eta|}F_1(t,\eta)J_{23}(\xi-\eta)\ud\eta\\
        &\quad+\int_{|\eta|<|\xi-\eta|}F_1(t,\xi-\eta)J_{23}^{\prime}(\eta)\ud\eta+\int_{|\eta|<|\xi-\eta|}F_1(t,\eta)J^{\prime}_{23}(\xi-\eta)\ud\eta \\
        =:&J_1(\xi)+J_2(\xi)+J_3(\xi)+J_4(\xi).
    \end{aligned}\nonumber
    \end{align}
    As a consequence, we conclude that
    \begin{align}\nonumber
        \begin{aligned}
         &\|\widehat{\CT}(f_1,f_2,f_3)(\xi)\|_{L^2_T(L_\gamma^2)^{\prime}}\\
         &\lesssim\left\{\int_0^T\left(\int_{|\eta|<|\xi-\eta|}F_1(t,\xi-\eta)\int_{|\zeta|<|\eta-\zeta|}F_2(\eta-\zeta)F_3(\zeta)\ud\zeta\ud\eta\right)^2 \mathrm{~d} t\right\}^{\frac{1}{2}}\\
         &\lesssim\sum_{i=1}^4\|J_i(\xi)\|_{L^2_T}.
        \end{aligned}
        \end{align}
        To estimate the term $\|J_1(\xi)\|_{L^2_T}$, we apply Minkowski's inequality and Young's inequality. For any $\alpha_1,\beta_1,\gamma_1\in[1,\infty]$ such that $\frac{1}{\al_1}+\frac{1}{\beta_1}+\frac{1}{\gamma_1}=\frac{1}{2}$, we have
        \begin{align}
        \begin{aligned}
              \|J_1(\xi)\|_{L^2_T}&\lesssim\int_{|\eta|<|\xi-\eta|}\int_{|\zeta|<|\eta-\zeta|}\|F_1(\xi-\eta)F_2(\eta-\zeta)F_3(\zeta)\|_{L^2_T}\ud\zeta\ud\eta,\\
           &\lesssim\int_{|\eta|<|\xi-\eta|}\int_{|\zeta|<|\eta-\zeta|}\|F_1(\xi-\eta)\|_{L^{\al_1}_T}\|F_2(\eta-\zeta)\|_{L^{\beta_1}_T}\|F_3(\zeta)\|_{L^{\gamma_1}_T}\ud\zeta\ud\eta.
        \end{aligned}\label{ineq-J-1}
        \end{align}
        We denote
        \begin{align*}
            I_1(\eta):=\int_{|\zeta|<|\eta-\zeta|}\|F_2(\eta-\zeta)\|_{L^{\beta_1}_T}\|F_3(\zeta)\|_{L^{\gamma_1}_T}\ud\zeta,
        \end{align*}
        then for any $r_1\neq\frac{3}{2}$, applying Young's inequality to \eqref{ineq-J-1} we have
        \begin{align}
          \|J_1(\xi)\|_{L^2_T}^2\lesssim\|\weight^{r_1}I_1(\cdot)\|_{L^2}^2\int_{|\eta|<|\xi-\eta|}\langle\eta\rangle^{-2r_1}\|F_1(\xi-\eta)\|_{L^{\al_1}_T}^2\ud\xi.\nonumber
        \end{align}
        Then, arguing as in the proof of \cite[Proposition 4.7]{carrapatoso2025navier}. In fact, it based on standard inequalities such as $\one_{\{|\eta|<|\xi-\eta|\}}|\xi|\lesssim\one_{\{|\eta|<|\xi-\eta|\}}|\xi-\eta|$, change of variables $\xi-\eta\to\xi^{\prime}$, and
        \begin{align}
            \int_{|\eta|<|\xi-\eta|}\langle\eta\rangle^{-2r_1}\ud\eta\lesssim\begin{cases}
              \langle\xi^{\prime}\rangle^{3-2r_1}\quad &\text{if}\quad r_1<\tfrac{3}{2},\\
              1\quad&\text{if}\quad r_1>\tfrac{3}{2}.
            \end{cases}\nonumber
        \end{align}
        Then we have
        \begin{align}
            \int_{\BR}\langle\xi\rangle^{2m}\|J_1(\xi)\|_{L^2_T}^2\ud\xi\lesssim \|\weight^{r_1}I_1(\cdot)\|_{L^2}^2\|f_1\|_{\widetilde{L}^{\alpha_1}_TH^{m+(\frac{3}{2}-r_1)_+}L^2_v}^2.\label{ineq-J-1-up}
        \end{align}
        Repeating the above progress again, for any $\delta_1\neq\frac{3}{2}$, we get
        \begin{align}
          \begin{aligned}
          \|\weight^{r_1}I_1(\cdot)\|_{L^2}^2&=\int_{\BR}\langle\xi\rangle^{2r_1}\Big(\int_{|\zeta|<|\eta-\zeta|}\|F_2(\eta-\zeta)\|_{L^{\beta_1}_T}\|F_3(\zeta)\|_{L^{\gamma_1}_T}\ud\zeta\Big)^{2}\ud\xi\\
          &\lesssim\int_{\BR}\langle\zeta\rangle^{2\delta_1}\|F_3(\zeta)\|_{L^{\gamma_1}_T}^2\ud\zeta\times\int_{\BR}\langle\xi\rangle^{2r_1}\Big(\int_{|\zeta|<|\eta-\zeta|}\langle\zeta\rangle^{-2\delta_1}\|F_2(\eta-\zeta)\|_{L^{\beta_1}_T}\ud\zeta\Big)^{2}\ud\xi\\
          &\lesssim\|f_3\|_{L^{\gamma_1}_TH^{\delta_1}_x L_\gamma^2}^2\|f_2\|_{L^{\beta_1}_TH^{r_1+(\delta_1-\frac{3}{2})_+}_xL^2_v}^2.
        \end{aligned}\label{ineq-I-1}
        \end{align}
        Putting \eqref{ineq-I-1} into \eqref{ineq-J-1-up} gives
        \begin{align}
         \int_{\BR}\langle\xi\rangle^{2m}\|J_1(\xi)\|_{L^2_T}^2\ud\xi\lesssim \|f_1\|_{\widetilde{L}^{\alpha_1}_TH^{m+(\frac{3}{2}-r_1)_+}L^2_v}^2\|f_2\|_{\widetilde{L}^{\beta_1}_TH^{r_1+(\delta_1-\frac{3}{2})_+}_xL^2_v}^2 \|f_3\|_{\widetilde{L}^{\gamma_1}_TH^{\delta_1}_x L_\gamma^2}^2. \label{ineq-J-1-uup}
        \end{align}
        In the same way, for any $r_i,\delta_i\neq\frac{3}{2},\,i=2,3,4$ and any $\alpha_j,\beta_j,\gamma_j\in[1,\infty],\, j=2,3,4$ satisfy
        \begin{align*}
            \tfrac{1}{\al_2}+\tfrac{1}{\be_2}+\tfrac{1}{\ga_2}=\tfrac{1}{\al_3}+\tfrac{1}{\be_3}+\tfrac{1}{\ga_3}=\tfrac{1}{\al_4}+\tfrac{1}{\be_4}+\tfrac{1}{\ga_4}=\tfrac{1}{2},
        \end{align*}
        Then, we get the other estimates for remaining terms $J_2,J_3,J_4$, respectively,
        \begin{align}
            &\int_{\BR}\langle\xi\rangle^{2m}\|J_2(\xi)\|_{L^2_T}^2\ud\xi\lesssim \|f_1\|_{\widetilde{L}^{\alpha_2}_TH^{r_2}_xL^2_v}^2\|f_2\|_{\widetilde{L}^{\beta_2}_TH^{m+(\tfrac{3}{2}-r_{2})_++(\tfrac{3}{2}-\delta_{2})_+}_xL^2_v}^2 \|f_3\|_{\widetilde{L}^{\gamma_2}_TH^{\delta_2}_x L_\gamma^2}^2,\label{ineq-J-2}\\
            &\int_{\BR}\langle\xi\rangle^{2m}\|J_3(\xi)\|_{L^2_T}^2\ud\xi\lesssim \|f_1\|_{\widetilde{L}^{\alpha_3}_TH^{m+(\tfrac{3}{2}-r_{3})_+}_xL^2_v}^2\|f_2\|_{\widetilde{L}^{\beta_3}_TH^{\delta_3}_xL^2_v}^2 \|f_3\|_{\widetilde{L}^{\gamma_3}_TH^{r_3+(\frac{3}{2}-\delta_3)_+}_x L_\gamma^2}^2,\label{ineq-J-3}\\
            &\int_{\BR}\langle\xi\rangle^{2m}\|J_4(\xi)\|_{L^2_T}^2\ud\xi\lesssim \|f_1\|_{\widetilde{L}^{\alpha_4}_TH^{r_4}_xL^2_v}^2\|f_2\|_{\widetilde{L}^{\beta_4}_TH^{\delta_4}_xL^2_v}^2 \|f_3\|_{\widetilde{L}^{\gamma_4}_TH^{m+(\tfrac{3}{2}-r_{4})_++(\tfrac{3}{2}-\delta_{4})_+}_x L_\gamma^2}^2.\label{ineq-J-4}
        \end{align}
        Combining \eqref{ineq-J-1-uup}-\eqref{ineq-J-4} gives \eqref{ineq-tri-L2}.

        The proof of \eqref{ineq-tri-homo-L2} is analogous to that of \eqref{ineq-tri-L2}, with the only difference that in the homogeneous Sobolev setting, derivatives correspond in Fourier variables to multiplication by $|\xi|$.
\end{proof}

 The proof of the next lemma follows along the same lines as that of Lemma~\ref{lem-bilinear-ope}, and will therefore be only sketched here.

\begin{lemma}\label{lem-homo-bi}
  Assume that  $m>0$ and the functions $F_1\,,F_2$ are sufficiently smooth.
  Suppose that $p_1, q_1, p_2, q_2 \in[1, \infty]$ satisfy $\frac{1}{p_1}+\frac{1}{q_1}=\frac{1}{p_2}+\frac{1}{q_2}=\frac{1}{2}$. For any $r_1,r_2\neq\frac{3}{2}$, then
  \begin{align}
  \begin{aligned}
      \|F_1F_2\|_{L^2_TH^m_x} \lesssim & \left\|F_1\right\|_{\widetilde{L}_T^{p_1} H_x^{m+\left(\frac{3}{2}-r_1\right)_{+}} }\left\|F_2\right\|_{\widetilde{L}_T^{q_1} H_x^{r_1} } +\left\|F_1\right\|_{\widetilde{L}_T^{p_2} H_x^{r_2} }\left\|F_2\right\|_{\widetilde{L}_T^{q_2} H_x^{m+\left(\frac{3}{2}-r_2\right)_{+}}}.
  \end{aligned}\nonumber
  \end{align}
\end{lemma}

\subsection{\texorpdfstring{Refined Estimates for $\mathcal{Q}$ and $\mathcal{T}$}{Refined Estimates for Q and T}}\label{subsec-Rfdes-QT}
Firstly, we establish the estimates for the semigroup $U^\ve$, which serve as a fundamental tool for deriving nonlinear estimates. We define the fluid variables for any function $f$ as follows
\begin{equation*}
  \left\{\begin{split}
            \rho[f]:=&\tfrac{1}{K_A E_0-3 E_2 / 2}\int_{\mathbb{R}^3}\big[1+\big(\tfrac{2 E_0}{3 E_2} K_g-1\big) \tfrac{|v|^2}{2}\big]f\sqrt{\muq}\,\ud v, \\
       u[f]:=&\tfrac{1}{E_2}\int_{\BR}vf\sqrt{\muq}\,\ud v,\\
       \vartheta[f]:=&\tfrac{1}{K_A E_0-3 E_2 / 2}\int_{\BR}\big(\tfrac{E_0}{3 E_2}|v|^2-1\big)f\sqrt{\muq}\,\ud v.
         \end{split}
  \right.
\end{equation*}
Furthermore, for some constants $0<\delta_1\ll\delta_2\ll\delta_3\ll1$ and each $\xi\in\BR$, we define a bilinear symmetric form $\mathscr{B}[\cdot,\cdot]$ depending on $\xi$ by
\begin{align*}
 \mathscr{B}&[f_1(\xi), f_2(\xi)] :=\delta_1E_2\frac{i\xi}{\langle \xi\rangle}\cdot \Big(u[\widehat{f}_1(\xi)]\rho[\widehat{f}_2(\xi)]+u[\widehat{f}_2(\xi)]\rho[\widehat{f}_1(\xi)]\Big)\\
 &+\delta_2\Big(\tfrac{3E_4}{2}+3E_{22}-3E_2K_g\Big)\frac{i\xi}{\langle \xi\rangle}\cdot \left(u[\widehat{f}_1(\xi)]\vartheta[\widehat{f}_2(\xi)] + u[\widehat{f}_2(\xi)]\vartheta[\widehat{f}_1(\xi)]\right)\\
 &+\delta_2\left\{\frac{i\xi}{\langle \xi\rangle}\otimes u[\widehat{f}_1(\xi)]:\int_{\BR}v\otimes v\PP_0^{\perp}\widehat{f}_2(\xi) \sqrt{\muq}\,\ud v + \frac{i\xi}{\langle \xi\rangle}\otimes u[\widehat{f}_2(\xi)]:\int_{\BR}v\otimes v\PP_0^{\perp}\widehat{f}_1(\xi) \sqrt{\muq}\,\ud v\right\}\\
 &+\delta_3\frac{i\xi}{\langle \xi\rangle}\cdot \left\{\vartheta[\widehat{f}_1(\xi)]\int_{\BR}\tfrac{|v|^2}{2}v\PP_0^{\perp}\widehat{f}_2(\xi) \sqrt{\muq}\,\ud v +\vartheta[\widehat{f}_2(\xi)]\int_{\BR}\tfrac{|v|^2}{2}v\PP_0^{\perp}\widehat{f}_1(\xi) \sqrt{\muq}\,\ud v\right\}.
\end{align*}
In addition, we also define a modified inner product $\langle\langle\cdot,\cdot\rangle\rangle_{L^2_v}$ on $L^2_v$ by
\begin{align}\label{def-inner-product}
 \langle\langle f_1(\xi),f_2(\xi)\rangle\rangle_{L^2_v}:=\langle f_1(\xi),f_2(\xi)\rangle_{L^2_v}+\ve\mathscr{B}[f_1(\xi), f_2(\xi)],
\end{align}
with its associated norm $|\!|\!|f|\!|\!|_{L^2_v}:=\langle\langle f,f\rangle\rangle_{L^2_v}^{\frac{1}{2}}$.
\begin{remark}\label{rmk-norm}
  We can choose $\delta_1,\,\delta_2,\delta_3$ small enough such that for any $\ve<1$ and $f\in L_v^2$
  \begin{equation}\label{eql-norm}
  |\!|\!|f|\!|\!|_{L^2_v}\sim\|f\|_{L^2_v}.
  \end{equation}
  We refer to Theorem 2.3 in \cite{Strain2012KRM} for this arguments.
\end{remark}

Now we consider the Cauchy problem:
\begin{equation}\label{eq-CP}
  \begin{cases}
    \partial_t \widehat{f}(\xi)=\widehat{\Lambda}^{\ve}(\xi)\widehat{f}(\xi) + \widehat{w}(\xi),\\
    \widehat{f}(\xi)\mid_{t=0}=\widehat{f}_{\mathrm{in}}(\xi).
\end{cases}
\end{equation}
We then have
\begin{lemma}
  For functional $\operatorname{Re}\mathscr{B}[\widehat{f}(\xi),\widehat{f}(\xi)]$, there is a constant $\lambda_1>0$ such that for any $\xi\in\mathbb{R}^3$,
  \begin{equation}\label{es-mac}
    \ve\partial_t\operatorname{Re}\mathscr{B}[\widehat{f}(\xi),\widehat{f}(\xi)] + \frac{\lambda_1|\xi|^2}{1 + |\xi|^2}\|\mathbf{P}_0\widehat{f}(\xi,\cdot)\|_{L_v^2}^2 \lesssim \frac{1}{\ve^2}\|\mathbf{P}_0^{\perp}\widehat{f}(\xi,\cdot)\|_{L_v^2}^2 + \ve^2 \|\widehat{w}(\xi,\cdot)\|_{(L_\gamma^2)'}^2,
  \end{equation}
  and
  \begin{equation}\label{es-mic}
    \ddt\|\widehat{f}(t,\xi,\cdot)\|_{L_v^2}^2 + \frac{\lambda}{\ve^2}\|\mathbf{P}_0^{\perp}\widehat{f}(t,\xi,\cdot)\|_{L_\gamma^2}^2 \lesssim |\operatorname{Re}\langle \widehat{w}(\xi,\cdot),\,\widehat{f}(\xi,\cdot)\rangle_{L_v^2}|.
  \end{equation}
\end{lemma}
\begin{proof}
  The first estimate called macroscopic estimate is given by a similar way to \cite[Lemma 4.1]{Duan2011NON}, and the second estimate called microscopic estimate is given by a similar way to \cite[(2.8)]{Strain2012KRM}.
\end{proof}

By using the lemma above, we can obtain the following estimates.
\begin{lemma}\label{lem-semi-hypocoercivity}
    Give $m\geqslant 1$ and $T>0$, it holds that
    \begin{enumerate}[(1)]
\item For any $f \in H^m_x L^2_v$, we have for all $t>0$,
\begin{align*}
\|U^\varepsilon(\cdot) f\|_{\widetilde{L}^\infty_t H^m_x L^2_v}
+ \|\mathbf{P}_0 U^\varepsilon(\cdot) f\|_{L^2_t \dot{H}^m_x L^2_v}
+ \frac{1}{\varepsilon} \|\mathbf{P}_0^\perp U^\varepsilon(\cdot) f\|_{L^2_t H^m_x L_\gamma^2}
\lesssim \|f\|_{H^m_x L^2_v}.
\end{align*}

\item Let $S = S(t,x,v)$ satisfy $\mathbf{P}_0 S = 0$ and $S \in L^2_T H^m_x (L_\gamma^2)^{\prime}$. Then for any $t \leqslant T$,
\begin{align*}
&\left\| \int_0^t U^\varepsilon(t - s) S(s) \, \mathrm{d}s \right\|_{\widetilde{L}^\infty_T H^m_x L^2_v} + \left\| \mathbf{P}_0 \int_0^t U^\varepsilon(t - s) S(s) \, \mathrm{d}s \right\|_{L^2_T \dot{H}^m_x L^2_v} \\
&\hspace{2cm}+ \frac{1}{\varepsilon} \left\| \mathbf{P}_0^\perp \int_0^t U^\varepsilon(t - s) S(s) \, \mathrm{d}s \right\|_{L^2_T H^m_x L_\gamma^2}
\lesssim \varepsilon \|S\|_{L^2_T H^m_x (L_\gamma^2)^{\prime}}.
\end{align*}
\end{enumerate}
\end{lemma}
\begin{proof}
  In order to study the properties of semigroup $\widehat{U}^\ve(t) = e^{t\widehat{\Lambda}^{\ve}(\xi)}$, we study the Cauchy problem \eqref{eq-CP} with $w\equiv 0$. Then its solution can be written as
  \begin{equation*}
    \widehat{f}(t,\xi,v) = \widehat{U}^\ve(t)\widehat{f}_{\mathrm{in}}(\xi,v).
  \end{equation*}
  Noting Remark \ref{rmk-norm}, $\delta_1$, $\delta_2$ and $\delta_3$ are small enough. Combining \eqref{es-mac} with \eqref{es-mic} for $w\equiv 0$ to get for $\lambda_2 = \min\{\lambda,\lambda_1\}$,
  \begin{equation}\label{es-homo-CP}
    \begin{split}
       &\ddt\left\{\|\widehat{f}(t,\xi,\cdot)\|_{L_v^2}^2 + \ve\operatorname{Re}\mathscr{B}[\widehat{f}(\xi),\widehat{f}(\xi)]\right\} \\
       &+ \lambda_2\left\{ \frac{1}{\ve^2}\|\mathbf{P}_0^{\perp}\widehat{f}(t,\xi,\cdot)\|_{L_\gamma^2}^2 + \frac{|\xi|^2}{1 + |\xi|^2}\|\mathbf{P}_0\widehat{f}(\xi,\cdot)\|_{L_v^2}^2\right\} \leqslant 0
    \end{split}
  \end{equation}
   Integrating over $[0,t]$, we can get
  \begin{align*}
    &\|\widehat{f}(t,\xi,\cdot)\|_{L_v^2}^2 + \frac{1}{\varepsilon^2}\int_{0}^{t}\|\mathbf{P}_0^{\perp} \widehat{f}(s,\xi,\cdot)\|_{L_\gamma^2}^2\ud s\\
    &\hspace{2cm}+ \int_{0}^{t}\frac{|\xi|^2}{1 + |\xi|^2}\|\mathbf{P}_0\widehat{f}(s,\xi,\cdot)\|_{L_v^2}^2\ud s \lesssim \|\widehat{f}_{\mathrm{in}}(\xi,\cdot)\|_{L_v^2}^2,
  \end{align*}
  where the equivalence \eqref{eql-norm} is used. Next, we take $L^\infty$ norm in time and then multiply by $(1 + |\xi|^2)^m$  to yield
  \begin{align*}
    &(1 + |\xi|^2)^m\|\widehat{f}(\xi)\|_{L_t^\infty L_v^2}^2 + \frac{(1 + |\xi|^2)^m}{\varepsilon^2}\|\mathbf{P}_0^{\perp} \widehat{f}(\xi)\|_{L_t^2 L_\gamma^2}^2\\
    &\qquad + |\xi|^2(1 + |\xi|^2)^{m-1}\|\mathbf{P}_0\widehat{f}(\xi)\|_{L_t^2 L_v^2}^2 \lesssim (1 + |\xi|^2)^m\|\widehat{f}_{\mathrm{in}}(\xi)\|_{L_v^2}^2,
  \end{align*}
  which implies the first estimate.

  For the second one, we denote by $h = \displaystyle \int_{0}^{t}U^\ve(t-s)S(s)\ud s$, and hence $\widehat{h}(\xi)$ is the solution of the Cauchy problem
  \begin{equation*}
    \partial_t \widehat{h}(\xi)=\widehat{\Lambda}^{\ve}(\xi)\widehat{h}(\xi) + \widehat{S}(\xi),\quad \widehat{h}(\xi)\mid_{t=0}=0,
  \end{equation*}
  Same as \eqref{es-homo-CP}, we can use \eqref{es-mac}-\eqref{es-mic} to obtain for constant $\lambda_3>0$
  \begin{equation}\label{es-inhomo-CP}
    \begin{split}
       &\ddt\left\{\|\widehat{h}(t,\xi,\cdot)\|_{L_v^2}^2 + \ve\operatorname{Re}\mathscr{B}[\widehat{h}(\xi),\widehat{h}(\xi)]\right\}\\
       &\quad+ \lambda_3\left\{ \frac{1}{\ve^2}\|\mathbf{P}_0^{\perp}\widehat{f}(t,\xi,\cdot)\|_{L_\gamma^2}^2 + \frac{|\xi|^2}{1 + |\xi|^2}\|\mathbf{P}_0\widehat{f}(\xi,\cdot)\|_{L_v^2}^2\right\}\\
       &\hspace{2cm}\leqslant \ve^2\|\widehat{S}(\xi,\cdot)\|_{(L_\gamma^2)'}^2 + |\operatorname{Re}\langle \widehat{S}(\xi,\cdot),\,\widehat{h}(\xi,\cdot)\rangle_{L_v^2}|.
    \end{split}
  \end{equation}
  By duality and $\widehat{S}(\xi)\in \Ker(\CL)^\perp$, one can use a simple Cauchy-Schwarz's inequality to acquire that
  \begin{align*}
    |\operatorname{Re}\langle \widehat{S}(\xi,\cdot),\,\widehat{h}(\xi,\cdot)\rangle_{L_v^2}| \leqslant -\frac{\lambda_3}{2}\dfrac{1}{\varepsilon^2}\|\mathbf{P}_0^{\perp} \widehat{h}(\xi)\|_{L_\gamma^2}^2 + C\ve^2\|\widehat{S}(\xi)\|_{(L_\gamma^2)'}^2,
  \end{align*}
  Gathering the above two estimates and then integrating over $[0,t]$ to get
  \begin{align*}
    \|\widehat{h}(t,\xi,\cdot)\|_{L_v^2}^2 + &\frac{1}{\varepsilon^2}\int_{0}^{t}\|\mathbf{P}_0^{\perp} \widehat{h}(s,\xi,\cdot)\|_{L_\gamma^2}^2\ud s\\
    &+ \int_{0}^{t}\frac{|\xi|^2}{1 + |\xi|^2}\|\mathbf{P}_0\widehat{h}(s,\xi,\cdot)\|_{L_v^2}^2\ud s \lesssim \ve^2\int_{0}^{t}\|\widehat{S}(s,\xi,\cdot)\|_{(L_\gamma^2)'}^2\ud s.
  \end{align*}
  Finally, taking the $L^\infty$ norm in time and then multiplying by $(1 + |\xi|^2)^m$  to yield the second estimate.
\end{proof}
As arguing in \cite{carrapatoso2025navier}, we also have following estimates for the components of $\mathfrak{Q}^{\ve}[f_1,f_2]$ by applying Lemma \ref{lem-semi-hypocoercivity}.
\begin{lemma}\label{lem-bilinear-semi}
   Suppose $f_1,f_2$ are such that $\mathfrak{Q}^{\ve}[f_1,f_2]\in L^2_T H^m_x (L_\gamma^2)^{\prime}$ for some $T>0$. Then the following estimates hold:
   \begin{enumerate}
       \item  Let $m\geqslant 1$, it holds that
       \begin{align*}
           \|\mathfrak{Q}^\varepsilon[f_1, f_2]\|_{\widetilde{L}^\infty_T H^m_x L^2_v} + \|\PP_0\mathfrak{Q}^\varepsilon[f_1, f_2]\|_{L^2_T \dot{H}^m_x L^2_v} + \frac{1}{\varepsilon} \|\PP_0^\perp \mathfrak{Q}^\varepsilon&[f_1, f_2]\|_{L^2_T H^m_x L_\gamma^2} \\
           &\lesssim \|\CQ_{\mathrm{sym}}(f_1, f_2)\|_{L^2_T H^m_x (L_\gamma^2)^{\prime}},
  \end{align*}
  and
  \begin{align*}
          \|\mathfrak{Q}^{\varepsilon,\sharp}[f_1, f_2]\|_{\widetilde{L}^\infty_T H^m_x L^2_v}
    + \frac{1}{\varepsilon} \|\mathfrak{Q}^{\varepsilon,\sharp}[f_1, f_2]\|_{L^2_T H^m_x L_\gamma^2}
    \lesssim \|\CQ_{\mathrm{sym}}(f_1, f_2)\|_{L^2_T H^m_x (L_\gamma^2)^{\prime}}.
 \end{align*}
 \item Let $m\geqslant0$. For any $f \in H_x^m L_v^2$ there holds
 \begin{align*}
  \left\|U^{\varepsilon, \flat}(\cdot) f\right\|_{L_t^2 H_x^{m+1} L_\gamma^2} \lesssim\|f\|_{H_x^m L_v^2}
 \end{align*}
\item Consider $T>0$ and $m \in \mathbb{R}$. For any smooth enough functions $f_1$ and $f_2$, we have
\begin{align*}
    & \left\|\mathfrak{Q}^{\varepsilon, \flat}\left[f_1, f_2\right]\right\|_{\widetilde{L}_T^{\infty} H_x^m L_v^2} \lesssim\left\|\CQ_{\operatorname{sym}}\left(f_1, f_2\right)\right\|_{\widetilde{L}_T^{\infty} H_x^{m-1} L_v^2}, \\
& \left\|\mathfrak{Q}^{\varepsilon, \flat}\left[f_1, f_2\right]\right\|_{\widetilde{L}_T^{\infty} H_x^m L_v^2} \lesssim\left\|\CQ_{\operatorname{sym}}\left(f_1, f_2\right)\right\|_{\widetilde{L}_T^4 H_x^{m-\frac{1}{2}} L_v^2},\\
&\left\|\mathfrak{Q}^{\varepsilon, \flat}\left[f_1, f_2\right]\right\|_{L_T^2 H_x^m L_\gamma^2} \lesssim\left\|\CQ_{\mathrm{sym }}\left(f_1, f_2\right)\right\|_{L_T^2 H_x^{m-1}\left(L_\gamma^2\right)^{\prime}},\\
&\left\|\mathfrak{Q}^{\varepsilon, \sharp}\left[f_1, f_2\right]\right\|_{\widetilde{L}_T^{\infty} H_x^m L_v^2} \lesssim \varepsilon\left\|\CQ_{\operatorname{sym}}\left(f_1, f_2\right)\right\|_{\widetilde{L}_T^{\infty} H_x^m L_v^2}.
\end{align*}
   \end{enumerate}
\end{lemma}

\begin{remark}
    The arguments in Lemmas \ref{lem-bilinear-ope}-\ref{lem-bilinear-semi} lead to the continuity of $U^{\ve}(t)$ and the convergence of the data term $D_0^{\ve}$ as in Proposition \ref{prop-fixed-point-lem-conditions} (1)-(2). The proof follows the same strategy as in \cite[Sections~5.1-5.2]{carrapatoso2025navier}, and we omit the details here.
\end{remark}

\subsection{\texorpdfstring{Convergence of the Source Term $\CS^{\ve}$}{Convergence of the Source Term S}}\label{subsec-conv-S}
Based on a class of useful nonlinear estimates above, we now prove the statement (3) in Proposition \ref{prop-fixed-point-lem-conditions}. Recall that $\CS^{\ve}$ is given by \eqref{def-diffi-eq-part} and $\mathfrak{Q}^{\ve}$ is decomposed as in \eqref{decom-bi}, we have the following lemma to show $\CS^{\ve}\to 0$ as $\ve\to 0$.

\begin{lemma}
For any interval $I\subset (0,T)$, there exists a nonnegative increasing function $\Phi_1$ such that
\begin{align}\label{convergence-source-bi}
    \|\mathfrak{Q}^{\ve}[\gve,\gve]-\mathfrak{Q}_{\NSF}[\gve,\gve]\|_{\MXe_{I}}\leqslant\ve^{\frac{1}{2}-2\al}\Phi_1\left(\left\|g_{\mathrm{in}}\right\|_{H_x^{\frac{1}{2}} L_v^2}\right) \rightarrow 0, \quad \varepsilon \rightarrow 0,
\end{align}
and
\begin{align}
\|\mathfrak{T}^{\ve}[\gve,\gve,\gve]\|_{\MXe_I}\lesssim\ve^{\frac{1}{2}-\al(\ell+\frac{1}{2})} \left\|g_{\mathrm{in}}\right\|_{H_x^{\frac{1}{2}} L_v^2}^3 \exp \left(2C\left\|g_{\mathrm{in}}\right\|_{H_x^{\frac{1}{2}} L_v^2}^2\right). \label{convergence-source-tri}
\end{align}
\end{lemma}

\begin{proof}
  Following the approach in \cite[Section 5.3]{carrapatoso2025navier}, we apply  the bilinear estimates in Lemmas \ref{lem-bilinear-ope} and \ref{lem-bilinear-semi}, and derive similar bounds.

  It remains to deal with the trilinear part of the source term. By the definition of the space $\MXe_I$ in \eqref{def-space-MXe}, we can write the norm as
\begin{align*}
    &\|\mathfrak{T}^{\ve}[\gve,\gve,\gve]\|_{\MXe_I}\\
    =&\|\mathfrak{T}^{\ve}[\gve,\gve,\gve]\|_{\widetilde{L}_I^{\infty} H_x^{\frac{1}{2}} L_v^2}+\left\|\mathbf{P}_0 \mathfrak{T}^{\ve}[\gve,\gve,\gve]\right\|_{L_I^2 \dot H_x^{\frac{3}{2}} L_\gamma^2}\\
    &+\tfrac{1}{\sqrt{\varepsilon}}\left\|\mathbf{P}_0^{\perp} \mathfrak{T}^{\ve}[\gve,\gve,\gve]\right\|_{L_I^2 H_x^{\frac{3}{2}} L_\gamma^2}+\varepsilon^\beta\|\mathfrak{T}^{\ve}[\gve,\gve,\gve]\|_{\widetilde{L}_I^{\infty} H_x^{\ell} L_v^2}\\
& +\varepsilon^\beta\left(\left\|\mathbf{P}_0 \mathfrak{T}^{\ve}[\gve,\gve,\gve]\right\|_{L_I^2 \dot H_x^{\ell} L_\gamma^2}+\tfrac{1}{\sqrt{\varepsilon}}\left\|\mathbf{P}_0^{\perp} \mathfrak{T}^{\ve}[\gve,\gve,\gve]\right\|_{L_I^2  H_x^{\ell} L_\gamma^2}\right).
\end{align*}
Since $\CT_{\mathrm{sym}}[\gve,\gve,\gve]=\CT[\gve,\gve,\gve]\in(\KerL)^{\perp}$, from Lemma \ref{lem-semi-hypocoercivity} it implies that
\begin{align}
\begin{aligned}
\|\mathfrak{T}^{\ve}[\gve,\gve,\gve]\|_{\MXe_I}\lesssim& \ve\|\CT[\gve,\gve,\gve]\|_{L^2_IH^{\frac{1}{2}}(L_\gamma^2)^{\prime}}+(\ve+\ve^{\frac{3}{2}})\|\CT[\gve,\gve,\gve]\|_{L^2_IH^{\frac{3}{2}}_x(L_\gamma^2)^{\prime}}\\
 &+\ve^{\beta}(\ve+\ve^{\frac{3}{2}})\|\CT[\gve,\gve,\gve]\|_{L^2_IH^{\ell}_x(L_\gamma^2)^{\prime}}.
\end{aligned}\label{ineq-semi-tri-ker}
\end{align}
Next, we estimate the terms on the right hand side of  \eqref{ineq-semi-tri-ker} one by one. Applying Lemma \ref{lem-tri-estimates}-(1) with $m=\frac{1}{2}$, and setting
\begin{align*}
    (\al_i,\be_i,\ga_i,r_i,\delta_i)=\begin{cases}
       (\infty,\infty,2,\tfrac{1}{2},\ell),\quad i=1,2,\\
       (\infty,2,\infty,\ell,\tfrac{1}{2}),\quad i=3,4.
    \end{cases}
\end{align*}
Then we have
\begin{align}
\begin{aligned}
\|\CT[\gve,\gve,\gve]\|_{L^2_IH_x^{\frac{1}{2}}(L_\gamma^2)^{\prime}}&\lesssim 3\|\gve\|_{\widetilde{L}^{\infty}_IH^{\frac{1}{2}}_xL^2_v}\|\gve\|_{\widetilde{L}^{\infty}_IH^{\frac{3}{2}}_xL^2_v} \|\gve\|_{\widetilde{L}^{2}_TH^{\ell}_x L_\gamma^2}\\
 &\quad+\|\gve\|_{\widetilde{L}^{\infty}_IH^{\frac{1}{2}}_xL^2_v}\|\gve\|_{\widetilde{L}^{2}_IH^{\ell}_xL^2_v}\|\gve\|_{\widetilde{L}^{\infty}_TH^{\frac{3}{2}}_x L_\gamma^2}.
\end{aligned}\label{ineq-semi-tri-ker-right-11st}
\end{align}

Therefore, noting \eqref{es-kerL-L2v-H*v-ineq}, the estimate \eqref{ineq-semi-tri-ker-right-11st} can be further refined as follows:
\begin{align*}
\|\CT[\gve,\gve,\gve]\|_{L^2_IH_x^{\frac{1}{2}}(L_\gamma^2)^{\prime}}&\lesssim\|\gve\|_{\widetilde{L}^{\infty}_IH^{\frac{1}{2}}_xL^2_v}\|\gve\|_{\widetilde{L}^{\infty}_IH^{\frac{3}{2}}_xL^2_v} \|\gve\|_{\widetilde{L}^{2}_TH^{\ell}_x L_\gamma^2}.
\end{align*}
Further to apply \eqref{ineq-lower-bound} and \eqref{ineq-gve-higher-bounds} we get
\begin{align}
  \|\CT[\gve,\gve,\gve]\|_{L^2_IH_x^{\frac{1}{2}}(L_\gamma^2)^{\prime}}\lesssim \ve^{-\al(\ell-\frac{3}{2})-\al}\left\|g_{\text {in }}\right\|_{H_x^{\frac{1}{2}} L_v^2}^3 \exp \left(2C\left\|g_{\text {in }}\right\|_{H_x^{\frac{1}{2}} L_v^2}^2\right).\label{ineq-semi-tri-ker-right-1st}
\end{align}
In the same way, we have that
\begin{align}
&\|\CT[\gve,\gve,\gve]\|_{L^2_IH^{\frac{3}{2}}_x(L_\gamma^2)^{\prime}}\lesssim \ve^{-\al(\ell-\frac{3}{2})-2\al}\left\|g_{\text {in }}\right\|_{H_x^{\frac{1}{2}} L_v^2}^3 \exp \left(2C\left\|g_{\text {in }}\right\|_{H_x^{\frac{1}{2}} L_v^2}^2\right),\\
&\|\CT[\gve,\gve,\gve]\|_{L^2_IH^{\ell}_x(L_\gamma^2)^{\prime}}\lesssim\ve^{-\al(\ell-\frac{3}{2})-\al(\ell+\frac{1}{2})}\left\|g_{\text {in }}\right\|_{H_x^{\frac{1}{2}} L_v^2}^3 \exp \left(2C\left\|g_{\text {in }}\right\|_{H_x^{\frac{1}{2}} L_v^2}^2\right)\label{ineq-semi-tri-ker-right-3rd}.
\end{align}
Putting \eqref{ineq-semi-tri-ker-right-1st}-\eqref{ineq-semi-tri-ker-right-3rd} into \eqref{ineq-semi-tri-ker} we get
\begin{align}
\begin{aligned}
     &\|\mathfrak{T}^{\ve}[\gve,\gve,\gve]\|_{\MXe_I}\\
     \lesssim&\Big(\ve^{1-\al(\ell-\frac{1}{2})}+(\ve+\ve^{\frac{3}{2}})(\ve^{-\al(\ell+\frac{1}{2})}+\ve^{\be-\al(2\ell-1)})\Big)\left\|g_{\text {in }}\right\|_{H_x^{\frac{1}{2}} L_v^2}^3 \exp \left(2C\left\|g_{\text {in }}\right\|_{H_x^{\frac{1}{2}} L_v^2}^2\right).
\end{aligned}\nonumber
\end{align}
Recalling the assumptions for $\al<\frac{1}{4},\,\beta<\frac{1}{2}$, and $\ell\in(\frac{3}{2},2]$, which imply $\beta>\al(\ell-\frac{1}{2})$ then
\begin{align*}
1-\al(\ell-\tfrac{1}{2})>0,\quad\tfrac{1}{2}-\al(\ell+\tfrac{1}{2})>0,\quad\tfrac{1}{2}+\be-\al(2\ell-1)>0.
\end{align*}
For $\ve$ sufficiently small, we have
\begin{align*}
\|\mathfrak{T}^{\ve}[\gve,\gve,\gve]\|_{\MXe_I}\lesssim\ve^{\frac{1}{2}-\al(\ell+\frac{1}{2})} \left\|g_{\text {in }}\right\|_{H_x^{\frac{1}{2}} L_v^2}^3 \exp \left(2C\left\|g_{\text {in }}\right\|_{H_x^{\frac{1}{2}} L_v^2}^2\right).
\end{align*}

Finally, combining \eqref{convergence-source-bi} and \eqref{convergence-source-tri}, we conclude the convergence of the source term as stated in Proposition \ref{prop-fixed-point-lem-conditions}-(3). More precisely, there exists a nonnegative increasing function $\Phi$ such that
\begin{align}
\begin{aligned}\|\CS^{\ve}\|_{\MXe_I}&\leqslant\|\mathfrak{Q}^{\ve}[\gve,\gve]-\mathfrak{Q}_{\NSF}[\gve,\gve]\|_{\MXe_{I}}+\|\mathfrak{T}^{\ve}[\gve,\gve,\gve]\|_{\MXe_I}\\
&\lesssim\ve^{\frac{1}{2}-2\al}\Phi( \left\|g_{\text {in }}\right\|_{H_x^{\frac{1}{2}} L_v^2}),
\end{aligned}\nonumber
\end{align}
for $\ell\in(\frac{3}{2},2]$.
\end{proof}

\subsection{\texorpdfstring{Estimates for Linear Term $\CA^{\ve}[f]$}{Estimates for Linear Term A[f]}}
Recall the definition of the operator $\CA^{\ve}$ given in \eqref{def-diffi-eq-part}. In order to prove $\CA^{\ve}$ is continuous in functional space $\mathcal{X}_I^{\varepsilon}$, we show that $\mathfrak{Q}^{\varepsilon}$ and $\mathfrak{T}^{\ve}$ are continuous on $\mathcal{X}_I^{\varepsilon}$.

\begin{lemma}
    Under the assumptions of Theorem \ref{thm-main}, there holds that for $\varepsilon$ small enough and all intervals $I$,
    \begin{equation}\label{ineq-liner-bi-part}
      \begin{split}
         \left\|\mathfrak{Q}^{\varepsilon}[f,\gve]\right\|_{\mathcal{X}_I^{\varepsilon}} \lesssim &\|f\|_{\mathcal{X}_I^{\varepsilon}}\left(\left\|g^{\varepsilon}\right\|_{\widetilde{L}_I^4 H_x^1 L_v^2}+\left\|g^{\varepsilon}\right\|_{L_I^2 H_x^{\frac{3}{2}} L_v^2}\right)\\
         &+\|f\|_{\mathcal{X}_I^{\varepsilon}}\left(\varepsilon^\beta\left\|g^{\varepsilon}\right\|_{\widetilde{L}_I^{\infty} H_x^{\ell} L_v^2}+\varepsilon^\beta\left\|g^{\varepsilon}\right\|_{L_I^2 H_x^{\ell} L_v^2}\right).
      \end{split}
    \end{equation}
    and
    \begin{equation}\label{ineq-liner-tri-part}
      \begin{split}
         &\|\mathfrak{T}^{\ve}[f,\gve,\gve]\|_{\MXe_I}\\\lesssim& \sqrt{\ve}\|f\|_{\MXe_I}\|\gve\|_{\widetilde{L}^{\infty}_IH^{\frac{1}{2}}_xL^2_v}(\|\gve\|_{\widetilde{L}^{\infty}_IH^{\frac{5}{2}}_xL^2_v}+\|\gve\|_{L^2_IH^{\frac{7}{2}}_xL^2_v})\\
 &\quad+\sqrt{\ve}\|f\|_{\MXe_I}(\|\gve\|_{L^2_IH^{\frac{3}{2}}_xL^2_v}\|\gve\|_{\widetilde{L}^{\infty}_IH^{\frac{5}{2}}_xL^2_v}+\|\gve\|_{L^2_IH^{\ell}_xL^2_v}^2)\\
  &\quad+\sqrt{\ve}\|f\|_{\MXe_I}\|\gve\|_{\widetilde{L}^{\infty}_IH^{\ell}_xL^2_v}(\|\gve\|_{L^2_IH^{\ell+1}_xL^2_v}+\|\gve\|_{\widetilde{L}^{\infty}_IH^{\ell}_xL^2_v})\\
  &\quad+\sqrt{\ve} \|f\|_{\MXe_I}\|\gve\|_{\widetilde{L}^{\infty}_IH^1_xL^2_v}^2.
      \end{split}
    \end{equation}
\end{lemma}

\begin{proof}
  Based on Lemma \ref{lem-bilinear-ope} and the second statement in Lemma \ref{lem-semi-hypocoercivity}, we can obtain \eqref{ineq-liner-bi-part} by a similar way to \cite[Section 5.4]{carrapatoso2025navier}.

  It remains to derive the continuity estimate of the trilinear term $\mathfrak{T}^{\ve}[f,\gve,\gve]$. By Lemma \ref{lem-semi-hypocoercivity} and the definition of the symmetric form \eqref{trilinear-sym} of the trilinear operator $\mathfrak{T}^{\ve}$, we have
\begin{align}
\begin{aligned}
&\|\mathfrak{T}^{\ve}[f,\gve,\gve]\|_{\MXe_I}\\
\lesssim&\, \ve\|\CT_{\mathrm{sym}}[f,\gve,\gve]\|_{L^2_IH^{\frac{1}{2}_x}(L_\gamma^2)^{\prime}}+(\ve+\ve^{\frac{3}{2}})\|\CT_{\mathrm{sym}}[f,\gve,\gve]\|_{L^2_IH^{\frac{3}{2}}_x(L_\gamma^2)^{\prime}}\\
 &+\quad\ve^{\beta}(\ve+\ve^{\frac{3}{2}})\|\CT_{\mathrm{sym}}[f,\gve,\gve]\|_{L^2_IH^{\ell}_x(L_\gamma^2)^{\prime}}\\
 \lesssim&\,(1+\ve^{\beta})(\ve+\ve^{\frac{3}{2}})\CA_0+\ve(\CA_1+\CA_2)+(\ve+\ve^{\frac{3}{2}})(\CA_3+\CA_4+\CA_5)\\
 &\quad+\ve^{\beta}(\ve+\ve^{\frac{3}{2}})(\CA_6+\CA_7+\CA_8),
\end{aligned}\label{ineq-semi-tri-sym-f-gve-gve}
\end{align}
where
\begin{align*}
&\CA_0:=\|\CT[f,\gve,\gve]\|_{L^2_IL^2_x(L_\gamma^2)^{\prime}}+\|\CT[\gve,f,\gve]\|_{L^2_IL^2_x(L_\gamma^2)^{\prime}}+\|\CT[\gve,\gve,f]\|_{L^2_IL^2_x(L_\gamma^2)^{\prime}}\\
    &\CA_1:=\|\CT[f,\gve,\gve]\|_{L^2_I\dot H^{\frac{1}{2}}_x(L_\gamma^2)^{\prime}}+\|\CT[\gve,f,\gve]\|_{L^2_I\dot H^{\frac{1}{2}}_x(L_\gamma^2)^{\prime}},\quad \CA_2:=\|\CT[\gve,\gve,f]\|_{L^2_I\dot H^{\frac{1}{2}}_x(L_\gamma^2)^{\prime}},\\
    &\CA_3:=\|\CT[f,\gve,\gve]\|_{L^2_I\dot H^{\frac{3}{2}}_x(L_\gamma^2)^{\prime}}\quad \CA_4:=\|\CT[\gve,f,\gve]\|_{L^2_I\dot H^{\frac{3}{2}}_x(L_\gamma^2)^{\prime}},\quad \CA_5:=\|\CT[\gve,\gve,f]\|_{L^2_I\dot H^{\frac{3}{2}}_x(L_\gamma^2)^{\prime}},\\
    &\CA_6:=\|\CT[f,\gve,\gve]\|_{L^2_I\dot H^{\ell}_x(L_\gamma^2)^{\prime}},\quad \CA_7:=\|\CT[\gve,f,\gve]\|_{L^2_I\dot H^{\ell}_x(L_\gamma^2)^{\prime}},\quad \CA_8:=\|\CT[\gve,\gve,f]\|_{L^2_I\dot H^{\ell}_x(L_\gamma^2)^{\prime}}.
\end{align*}
For the zeroth order term $\CA_0$ in spatial variable $x$, using \eqref{es-T}, Hölder' s inequality and the Grgliardo-Nienberg inequality, we obtain that for all $t\in I$,
\begin{align}
\begin{aligned}
  &\|\CT[f,\gve,\gve](t)\|_{L^2_x(L_\gamma^2)^{\prime}}\\
  \lesssim&\Big(\int_{\BR}\big(\|f(t)\|_{L^2_v}\|\gve(t)\|_{L^2_{\ga}}+\|f(t)\|_{L^2_{\ga}}\|\gve(t)\|_{L^2_v}\big)^2\|\gve(t)\|_{L^2_v}^2\ud x\Big)^{\frac{1}{2}}\\
    \lesssim&\Big(\|f(t)\|_{L^6_xL^2_v}\|\gve(t)\|_{L^6_xL^2_{\ga}}+\|f(t)\|_{L^6_xL^2_{\ga}}\|\gve(t)\|_{L^6_xL^2_v}\Big)\|\gve(t)\|_{L^6_xL^2_v}\\
    \lesssim&\|\nabla_xf(t)\|_{L^2_xL^2_{\ga}}\|\gve(t)\|_{H^1_xL^2_{v}}^2.
\end{aligned}\nonumber
\end{align}
Moreover, by using Hölder's inequality once again, we have
\begin{align}\label{A0-fgg-es}
\begin{aligned}
   \|\CT[f,\gve,\gve]\|_{L^2_IL^2_x(L_\gamma^2)^{\prime}}&\lesssim\Big(\|\nabla_x\PP_0f\|_{L^2_IL^2_xL^2_{\ga}}+\|\PP_0^{\perp}f\|_{L^2_IH^1_xL^2_{\ga}}\Big)\|\gve\|_{\widetilde{L}^{\infty}_IH^1_xL^2_v}\\
 &\lesssim(1+\sqrt{\ve})\|f\|_{\MXe_I}\|\gve\|_{\widetilde{L}^{\infty}_IH^1_xL^2_v}^2.
\end{aligned}
\end{align}
Similarly, it have that
\begin{align}
  \|\CT[\gve,f,\gve]\|_{L^2_IL^2_x(L_\gamma^2)^{\prime}}&\lesssim(1+\sqrt{\ve})\|f\|_{\MXe_I}\|\gve\|_{\widetilde{L}^{\infty}_IH^1_xL^2_v}^2\\
   \|\CT[\gve,\gve,f]\|_{L^2_IL^2_x(L_\gamma^2)^{\prime}}&\lesssim(1+\sqrt{\ve})\|f\|_{\MXe_I}\|\gve\|_{\widetilde{L}^{\infty}_IH^1_xL^2_v}^2.\label{A0-ggf-es}
\end{align}
Combining \eqref{A0-fgg-es}-\eqref{A0-ggf-es}, $\CA_0$ can be controlled as
\begin{align*}
  \CA_0& \lesssim(1+\sqrt{\ve})\|f\|_{\MXe_I}\|\gve\|_{\widetilde{L}^{\infty}_IH^1_xL^2_v}^2.
\end{align*}

We continue with $\CA_1,\cdots,\CA_8$ by applying Lemma \ref{lem-tri-estimates}-(1) with $m=\frac{1}{2}$ and choosing
\begin{align}
    (\al_j,\beta_j,\ga_j,r_j,\delta_j)=\begin{cases}
        (\infty,\infty,2,\ell,\ell),\quad j=1,2,3,\\
        (\infty,2,\infty,\ell,\ell),\quad j=4.
    \end{cases}\label{al-j-assume}
\end{align}
Under the choices \eqref{al-j-assume}, the term $\CA_1$ can be bounded as
\begin{align*}
    \CA_1\lesssim&\,3\|f\|_{\widetilde{L}^{\infty}_I\dot H^{\frac{1}{2}}_xL^2_v}\|\gve\|_{\widetilde{L}^{\infty}_I\dot H^{\ell}_xL^2_v}\|\gve\|_{\widetilde{L}^{2}_I\dot H^{\ell}_x L_\gamma^2}\\
    &+3\|f\|_{\widetilde{L}^{\infty}_I\dot H^{\ell}_xL^2_v}\|\gve\|_{\widetilde{L}^{\infty}_I\dot H^{\frac{1}{2}}_xL^2_v}\|\gve\|_{\widetilde{L}^{2}_I\dot H^{\ell}_x L_\gamma^2}\\
    &+\|f\|_{\widetilde{L}^{\infty}_I\dot H^{\ell}_xL^2_v}\|\gve\|_{\widetilde{L}^{2}_I\dot H^{\ell}_xL^2_v}\|\gve\|_{\widetilde{L}^{\infty}_I\dot H^{\frac{1}{2}}_x L_\gamma^2}\\
    &+\|f\|_{\widetilde{L}^{2}_I\dot H^{\ell}_xL^2_v}\|\gve\|_{\widetilde{L}^{\infty}_I\dot H^{\ell}_xL^2_v}\|\gve\|_{\widetilde{L}^{\infty}_I\dot H^{\frac{1}{2}}_x L_\gamma^2}.
\end{align*}
From \eqref{es-kerL-L2v-H*v-ineq} and Sobolev embedding $L_\gamma^2\hookrightarrow L^2_v$, we further get
\begin{align*}
  \CA_1&\lesssim 3\|f\|_{\MXe_I}\|\gve\|_{\widetilde{L}^{\infty}_IH^{\ell}_xL^2_v}\|\gve\|_{\widetilde{L}^{2}_IH^{\ell}_x L_\gamma^2}+5\|\PP_0f\|_{\widetilde{L}^{\infty}_I\dot H^{\ell}_x L_\gamma^2}\|\gve\|_{\widetilde{L}^{2}_IH^{\ell}_xL^2_v}\|\gve\|_{\widetilde{L}^{\infty}_IH^{\frac{1}{2}}_x L_\gamma^2}\\
  &\quad+5\|\PP_0^{\perp}f\|_{\widetilde{L}^{\infty}_I\dot H^{\ell}_x L_\gamma^2}\|\gve\|_{\widetilde{L}^{2}_IH^{\ell}_xL^2_v}\|\gve\|_{\widetilde{L}^{\infty}_IH^{\frac{1}{2}}_x L_\gamma^2}\\
  &\lesssim(1+\ve^{-\beta}(1+\sqrt{\ve}))\|f\|_{\MXe_I}\|\gve\|_{\widetilde{L}^{2}_IH^{\ell}_xL^2_v}\|\gve\|_{\widetilde{L}^{\infty}_IH^{\frac{1}{2}}_xL^2_v}.
\end{align*}
In the same way, for term $\CA_2$, we choose
\begin{align*}
(\al_k,\be_k,\ga_k,r_k,\delta_k)=\begin{cases}
        (\infty,\infty,2,\ell,\ell),\quad k=1,2,3,\\
        (\infty,\infty,2,\tfrac{1}{2},\ell),\quad k=4.
    \end{cases}
\end{align*}
Then, it holds that
\begin{align*}
    \CA_2&\lesssim3\|\gve\|_{\widetilde{L}^{\infty}_I\dot H^{\frac{1}{2}}_xL^2_v}\|\gve\|_{\widetilde{L}^{\infty}_I\dot H^{\ell}_xL^2_v}\Big(\|\PP_0f\|_{L^2_I\dot H^{\ell}_x L_\gamma^2}+\|\PP_0^{\perp}f\|_{L^2_I\dot H^{\ell}_x L_\gamma^2}\Big)\\
    &\quad+\|\gve\|_{\widetilde{L}^{\infty}_I\dot H^{\frac{1}{2}}_xL^2_v}\|\gve\|_{\widetilde{L}^{\infty}_I\dot H^{\ell}_xL^2_v}\Big(\|\PP_0f\|_{L^2_I\dot H^{\frac{3}{2}}_x L_\gamma^2}+\|\PP_0^{\perp}f\|_{L^2_I\dot H^{\frac{3}{2}}_x L_\gamma^2}\Big)\\
    &\lesssim(\ve^{-\be}(1+\sqrt{\ve})+1+\sqrt{\ve})\|f\|_{\MXe_I}\|\gve\|_{\widetilde{L}^{\infty}_IH^{\frac{1}{2}}_xL^2_v}\|\gve\|_{\widetilde{L}^{\infty}_IH^{\ell}_xL^2_v}.
\end{align*}
Combining term $\CA_1$ and $\CA_2$ and by $\be<\frac{1}{2}$, we obtain
\begin{align}
\begin{aligned}
\ve(\CA_1+\CA_2)
    &\lesssim\ve(\ve^{-\be}(1+\sqrt{\ve})+1+\sqrt{\ve})\|f\|_{\MXe_I}\|\gve\|_{\widetilde{L}^{\infty}_IH^{\frac{1}{2}}_xL^2_v}\|\gve\|_{\widetilde{L}^{\infty}_IH^{\ell}_xL^2_v}\\
    &\lesssim\sqrt{\ve}\|f\|_{\MXe_I}\|\gve\|_{\widetilde{L}^{\infty}_IH^{\frac{1}{2}}_xL^2_v}\|\gve\|_{\widetilde{L}^{\infty}_IH^{\ell}_xL^2_v}.
\end{aligned}\label{ineq-I1-I2}
\end{align}
For term $\CA_3$, we repeat above progress under some suitable choices of $(\al_i,\be_i,\ga_i,r_i,\delta_i),\,i=1,2,3,4$. Then it holds that
\begin{align*}
    \CA_3&\lesssim 2(1+\sqrt{\ve})\|f\|_{\MXe_I}\|\gve\|_{\widetilde{L}^{\infty}_IH^{\frac{1}{2}}_xL^2_v}^2+2\|f\|_{\MXe_I}\|\gve\|_{L^2_IH^{\frac{7}{2}}_xL^2_v}\|\gve\|_{\widetilde{L}^{\infty}_IH^{\frac{1}{2}}_xL^2_v},\\
    \CA_4&\lesssim \|\gve\|_{L^2_IH^{\frac{5}{2}}_xL^2_v}\|f\|_{\MXe_I}\|\gve\|_{\widetilde{L}^{\infty}_IH^{\frac{1}{2}}_xL^2_v}+\|\gve\|_{L^2_IH^{\ell}_xL^2_v}^2(1+\sqrt{\ve})\|f\|_{\MXe_I}\\
    &\quad+\|\gve\|_{\widetilde{L}^{\infty}_IH^{\frac{5}{2}}_xL^2_v}\|f\|_{\MXe_I}\|\gve\|_{L^2_IH^{\frac{3}{2}}_xL^2_v}+\|\gve\|_{\widetilde{L}^{\infty}_IH^{\frac{1}{2}}_xL^2_v}\|f\|_{\MXe_I}\|\gve\|_{L^2_IH^{\frac{7}{2}}_xL^2_v},
\end{align*}
For term $\CA_5$, we use the decomposition $f=\PP_0f+\PP_0^{\perp}f$, then, it holds that
\begin{align*}
    \CA_5&\lesssim \|\CT[\gve,\gve,\PP_0f]\|_{L^2_I\dot H^{\frac{3}{2}}_x(L_\gamma^2)^{\prime}}+\|\CT[\gve,\gve,\PP_0^{\perp}f]\|_{L^2_I\dot H^{\frac{3}{2}}_x(L_\gamma^2)^{\prime}}\\
    &=:\CA_{51}+\CA_{52}.
\end{align*}
Similarly, same as in \eqref{es-kerL-L2v-H*v-ineq}, it have that $\|\PP_0f\|_{L_\gamma^2}\lesssim\|f\|_{L^2_v}$. Then we get
\begin{align*}
    \CA_{51}&\lesssim(\|\gve\|_{\widetilde{L}^{\infty}_I\dot H^{\frac{5}{2}}_xL^2_v}+\|\gve\|_{\widetilde{L}^{\infty}_I\dot H^{\ell}_xL^2_v})\|\gve\|_{L^2_I\dot H^{\frac{3}{2}}_xL^2_v}\|\PP_0f\|_{\widetilde{L}^{\infty}_I\dot H^{\frac{1}{2}}_x L_\gamma^2}\\
    &\quad+(\|\gve\|_{L^2_I\dot H^{\frac{5}{2}}_xL^2_v}\|\PP_0f\|_{\widetilde{L}^{\infty}_I\dot H^{\frac{1}{2}}_x L_\gamma^2}+\|\gve\|_{\widetilde{L}^{\infty}_I\dot H^{\ell}_xL^2_v}\|\PP_0f\|_{L^2_I\dot H^{\frac{3}{2}}_x L_\gamma^2})\|\gve\|_{\widetilde{L}^{\infty}_I\dot H^{\ell}_xL^2_v}\\
    &\lesssim(\|\gve\|_{\widetilde{L}^{\infty}_IH^{\frac{5}{2}}_xL^2_v}+\|\gve\|_{\widetilde{L}^{\infty}_IH^{\ell}_xL^2_v})\|\gve\|_{L^2_IH^{\frac{3}{2}}_xL^2_v}\|f\|_{\MXe_I}\\
    &\quad+(\|\gve\|_{L^2_IH^{\frac{5}{2}}_xL^2_v}+\|\gve\|_{\widetilde{L}^{\infty}_IH^{\ell}_xL^2_v})\|\gve\|_{\widetilde{L}^{\infty}_IH^{\ell}_xL^2_v}\|f\|_{\MXe_I},\\
    \CA_{52}&\lesssim\|\gve\|_{\widetilde{L}^{\infty}_I\dot H^{\frac{5}{2}}_xL^2_v}\|\gve\|_{\widetilde{L}^{\infty}_I\dot H^{\frac{1}{2}}_xL^2_v}(2\|\PP_0^{\perp}f\|_{L^2_I\dot H^{\ell}_x L_\gamma^2}+\|\PP_0^{\perp}f\|_{L^2_I\dot H^{\frac{3}{2}}_x L_\gamma^2})\\
    &\quad+\|\gve\|_{\widetilde{L}^{\infty}_I\dot H^{\ell}_xL^2_v}^2\|\PP_0^{\perp}f\|_{L^2_I\dot H^{\frac{3}{2}}_x L_\gamma^2}\\
    &\lesssim((\ve^{-\be}+1)\|\gve\|_{\widetilde{L}^{\infty}_IH^{\frac{5}{2}}_xL^2_v}\|\gve\|_{\widetilde{L}^{\infty}_IH^{\frac{1}{2}}_xL^2_v}+\|\gve\|_{\widetilde{L}^{\infty}_IH^{\ell}_xL^2_v}^2)\sqrt{\ve}\|f\|_{\MXe_I}.
\end{align*}
For $\ve$ small enough, combining $\CA_3,\,\CA_4,\,\CA_5$, we have
\begin{align}
\begin{aligned}
&(\ve+\ve^{\frac{3}{2}})(\CA_3+\CA_4+\CA_5)\\\lesssim&\sqrt{\ve}\|f\|_{\MXe_I}(\|\gve\|_{\widetilde{L}^{\infty}_IH^{\frac{1}{2}}_xL^2_v}^2
+\|\gve\|_{\widetilde{L}^{\infty}_IH^{\ell}_xL^2_v}^2
+\|\gve\|_{L^2_IH^{\ell}_xL^2_v}^2)\\
&\quad+\sqrt{\ve}\|f\|_{\MXe_I}(\|\gve\|_{L^2_IH^{\frac{7}{2}}_xL^2_v}+\|\gve\|_{L^2_IH^{\frac{5}{2}}_xL^2_v}+\|\gve\|_{\widetilde{L}^{\infty}_IH^{\frac{5}{2}}_xL^2_v})\|\gve\|_{\widetilde{L}^{\infty}_IH^{\frac{1}{2}}_xL^2_v}\\
&\quad+\sqrt{\ve}\|f\|_{\MXe_I}(\|\gve\|_{\widetilde{L}^{\infty}_IH^{\frac{5}{2}}_xL^2_v}+\|\gve\|_{\widetilde{L}^{\infty}_IH^{\ell}_xL^2_v})\|\gve\|_{L^2_IH^{\frac{3}{2}}_xL^2_v}\\
&\quad+\sqrt{\ve}\|f\|_{\MXe_I}\|\gve\|_{\widetilde{L}^{\infty}_IH^{\ell}_xL^2_v}\|\gve\|_{L^2_IH^{\frac{5}{2}}_xL^2_v}.
\end{aligned}\nonumber
\end{align}
Same as $\CA_3$, the terms $\CA_6+\CA_7$ can be bounded by
\begin{align}
\begin{aligned}
 \CA_6+\CA_7&\lesssim5\|f\|_{\widetilde{L}^{\infty}_IH^{\frac{1}{2}}_xL^2_v}\|\gve\|_{L^2_IH^{\ell+1}_xL^2_v}\|\gve\|_{\widetilde{L}^{\infty}_xH^{\ell}_x L_\gamma^2}\\
 &\quad+2\|f\|_{\widetilde{L}^{\infty}_IH^{\ell}_xL^2_v}\|\gve\|_{L^2_IH^{\ell+1}_xL^2_v}\|\gve\|_{\widetilde{L}^{\infty}_IH^{\frac{1}{2}}_x L_\gamma^2}\\
    &\quad+\|f\|_{\widetilde{L}^{\infty}_IH^{\ell}_xL^2_v}\|\gve\|_{L^2_IH^{\ell}_xL^2_v}\|\gve\|_{\widetilde{L}^{\infty}_IH^{\ell}_xL_v^2}\\
    &\lesssim\ve^{-\beta}\|f\|_{\MXe_I}\|\gve\|_{L^2_IH^{\ell+1}_xL^2_v}\|\gve\|_{\widetilde{L}^{\infty}_IH^{\frac{1}{2}}_x L_\gamma^2}\\
    &\quad+\ve^{-\beta}\|f\|_{\MXe_I}\|\gve\|_{L^2_IH^{\ell}_xL^2_v}\|\gve\|_{\widetilde{L}^{\infty}_IH^{\ell}_xL_v^2}\\
    &\quad+\|f\|_{\MXe_I}\|\gve\|_{L^2_IH^{\ell+1}_xL^2_v}\|\gve\|_{\widetilde{L}^{\infty}_xH^{\ell}_x L_\gamma^2}.
\end{aligned}\nonumber
\end{align}
For the last term $\CA_8$, as the decomposition in $\CA_5$, we have
\begin{align}
    \begin{aligned}
      \CA_8&\lesssim  \|\CT[\gve,\gve,\PP_0f]\|_{L^2_I\dot H^{\ell}_x(L_\gamma^2)^{\prime}}+\|\CT[\gve,\gve,\PP_0^{\perp}f]\|_{L^2_I\dot H^{\ell}_x(L_\gamma^2)^{\prime}}\\
      &\lesssim\|\gve\|_{\widetilde{L}^{\infty}_IH^{\ell}_xL^2_v}\|\gve\|_{L^2_IH^{\ell}_xL^2_v}\|\PP_0f\|_{\widetilde{L}^{\infty}_I\dot H^{\ell}_x L_\gamma^2}+\|\gve\|_{\widetilde{L}^{\infty}_I\dot H^{\ell}_xL^2_v}^2\|\PP^{\perp}_0f\|_{L^2_I\dot H^{\ell}_x L_\gamma^2}\\
      &\lesssim(\|\gve\|_{\widetilde{L}^{\infty}_IH^{\ell}_xL^2_v}\|\gve\|_{L^2_IH^{\ell}_xL^2_v}+\sqrt{\ve}\|\gve\|_{\widetilde{L}^{\infty}_IH^{\ell}_xL^2_v}^2)\|f\|_{\MXe_I}.
    \end{aligned}\nonumber
\end{align}
Then combining $\CA_6,\,\CA_7,\,\CA_8$, it holds that
\begin{align}
\begin{aligned}
 &\ve^{\beta}(\ve+\ve^{\frac{3}{2}})(\CA_6+\CA_7+\CA_8)\\
 \lesssim&\sqrt{\ve}\|f\|_{\MXe_I}(\|\gve\|_{L^2_IH^{\ell+1}_xL^2_v}\|\gve\|_{\widetilde{L}^{\infty}_IH^{\frac{1}{2}}_x L_\gamma^2}+\|\gve\|_{L^2_IH^{\ell}_xL^2_v}\|\gve\|_{\widetilde{L}^{\infty}_IH^{\ell}_xL_v^2})  \\
 &\quad+\sqrt{\ve}\|f\|_{\MXe_I}(\|\gve\|_{L^2_IH^{\ell+1}_xL^2_v}\|\gve\|_{\widetilde{L}^{\infty}H^{\ell}_x L_\gamma^2}+\|\gve\|_{\widetilde{L}^{\infty}_IH^{\ell}_xL^2_v}\|\gve\|_{L^2_IH^{\ell}_xL^2_v})\\
 &\quad+\sqrt{\ve}\|f\|_{\MXe_I}\|\gve\|_{\widetilde{L}^{\infty}_IH^{\ell}_xL^2_v}^2.
\end{aligned} \label{ineq-I6-I7-I8}
\end{align}

Putting \eqref{ineq-I1-I2}-\eqref{ineq-I6-I7-I8} into \eqref{ineq-semi-tri-sym-f-gve-gve} and using Sobolev embedding $H^j_x\hookrightarrow H^i_x$ for any $0<i<j<\infty$, we conclude that
\begin{align}
\begin{aligned}
&\|\mathfrak{T}^{\ve}[f,\gve,\gve]\|_{\MXe_I}\\
\lesssim&(1+\ve^{\beta})(\ve+\ve^{\frac{3}{2}})\CA_0+\ve(\CA_1+\CA_2)+(\ve+\ve^{\frac{3}{2}})(\CA_3+\CA_4+\CA_5)\\
 &\quad+\ve^{\beta}(\ve+\ve^{\frac{3}{2}})(\CA_6+\CA_7+\CA_8)\\
 \lesssim&\sqrt{\ve}\|f\|_{\MXe_I}\|\gve\|_{\widetilde{L}^{\infty}_IH^{\frac{1}{2}}_xL^2_v}(\|\gve\|_{\widetilde{L}^{\infty}_IH^{\frac{5}{2}}_xL^2_v}+\|\gve\|_{L^2_IH^{\frac{7}{2}}_xL^2_v})\\
 &\quad+\sqrt{\ve}\|f\|_{\MXe_I}(\|\gve\|_{L^2_IH^{\frac{3}{2}}_xL^2_v}\|\gve\|_{\widetilde{L}^{\infty}_IH^{\frac{5}{2}}_xL^2_v}+\|\gve\|_{L^2_IH^{\ell}_xL^2_v}^2)\\
  &\quad+\sqrt{\ve}\|f\|_{\MXe_I}\|\gve\|_{\widetilde{L}^{\infty}_IH^{\ell}_xL^2_v}(\|\gve\|_{L^2_IH^{\ell+1}_xL^2_v}+\|\gve\|_{\widetilde{L}^{\infty}_IH^{\ell}_xL^2_v})\\
  &\quad+ \sqrt{\ve} \|f\|_{\MXe_I}\|\gve\|_{\widetilde{L}^{\infty}_IH^1_xL^2_v}^2
.
\end{aligned}\label{ineq-liner-tri-part-1}
\end{align}
As a consequence, combining \eqref{ineq-liner-bi-part} and \eqref{ineq-liner-tri-part-1}, we obtain the estimate in Proposition \ref{prop-fixed-point-lem-conditions}-(5).
\end{proof}
\subsection{\texorpdfstring{Nonlinear Estimate of $\CB^{\ve}[f,f]$}{Nonlinear Estimate of CB-epsilon[f,f]}} In this subsection, we focus on proving the fifth statement in Proposition \ref{prop-fixed-point-lem-conditions}, that is, bilinear operator $\CB^{\ve}$ is bounded. It suffices to show that $\mathfrak{Q}^{\ve}$ and $\mathfrak{T}^{\ve}$ are bounded on $\MXe_I\times\MXe_I$. We have the following lemma.

\begin{lemma}
    Under the assumptions of Theorem \ref{thm-main}, there holds that for all intervals $I$,
    \begin{equation}\label{CB-bi-part}
      \|\mathfrak{Q}^{\ve}[f_1,f_2]\|_{\MXe_I}\lesssim\|f_1\|_{\MXe_I}\|f_2\|_{\MXe_I}.
    \end{equation}
    and
    \begin{equation}\label{CB-tri-part}
   \|\mathfrak{T}^{\ve}[f_1,f_2,\gve]\|_{\MXe_I}\lesssim\|f_1\|_{\MXe_I}\|f_2\|_{\MXe_I}.
    \end{equation}
\end{lemma}
\begin{proof}
  We can prove \eqref{CB-bi-part} by a similar way as shown in \cite[Section 5.5]{carrapatoso2025navier}. For simplicity, we omit it. Now we focus on proving \eqref{CB-tri-part}. Since $\mathfrak{T}^{\ve}$ is symmetric in its arguments, we consider $\mathfrak{T}^{\ve}[\gve,f_1,f_2]$. Applying Lemma \ref{lem-semi-hypocoercivity} to get
\begin{align*}
 \|\mathfrak{T}^{\ve}[\gve,f_1,f_2]\|_{\MXe_I}&\lesssim(1+\ve^{\be})(\ve+\ve^{\frac{3}{2}})\|\CT_{\mathrm{sym}}[\gve,f_1,f_2]\|_{L^2_IL^2_x(L_\gamma^2)^{\prime}}\\
 &\quad+\ve\|\CT_{\mathrm{sym}}[\gve,f_1,f_2]\|_{L^2_I\dot H^{\frac{1}{2}}_x(L_\gamma^2)^{\prime}}\\
 &\quad+(\ve+\ve^{\frac{3}{2}})\|\CT_{\mathrm{sym}}[\gve,f_1,f_2]\|_{L^2_I\dot H^{\frac{3}{2}}_x(L_\gamma^2)^{\prime}}\\
 &\quad+\ve^{\beta}(\ve+\ve^{\frac{3}{2}})\|\CT_{\mathrm{sym}}[\gve,f_1,f_2]\|_{L^2_I\dot H^{\ell}_x(L_\gamma^2)^{\prime}}\\
 &=:\CB_0+\CB_1+\CB_2+\CB_3.
 \end{align*}
Then we shall handle the term $\CB_0$ first and then handle the terms $\CB_1,\CB_2,\CB_3$. Before this, we introduce an interpolation inequality:
\begin{align}
    \|\nabla_x f\|_{L^2_x}\leqslant\||\mathrm{D}_x|^{\frac{1}{2}}f\|_{L^2_x}^{\frac{1}{2}}\||\mathrm{D}_x|^{\frac{3}{2}}f\|_{L^2_x}^{\frac{1}{2}}.\label{interpolation-nabla-x-f}
\end{align}
Indeed, by Plancherel’s theorem in $L^2$ and the Cauchy–Schwarz  inequality, we have
\begin{align*}
    \|\nabla_x f\|_{L^2_x}&=\Big(\int_{\BR}|\xi|^2|\hat{f(\xi)|^2}\ud\xi\Big)^{\frac{1}{2}}\\
    &=\Big(\int_{\BR}|\xi|^{\frac{3}{2}}|\hat{f(\xi)}||\xi|^{\frac{1}{2}}|\hat{f(\xi)}|\ud\xi\Big)^{\frac{1}{2}}\\
    &\leqslant\Big(\int_{\BR}|\xi|^{3}|\hat{f(\xi)}|^2\ud\xi\Big)^{\frac{1}{4}}\Big(\int_{\BR}|\xi||\hat{f(\xi)}|^2\ud\xi\Big)^{\frac{1}{4}}\\
    &=\||\mathrm{D}_x|^{\frac{1}{2}}f\|_{L^2_x}^{\frac{1}{2}}\||\mathrm{D}_x|^{\frac{3}{2}}f\|_{L^2_x}^{\frac{1}{2}}.
\end{align*}
We now estimate $\CB_0$ and decompose it as
\begin{align*}
    \CB_0\lesssim\CB_{01}+\CB_{02}+\CB_{03}+\CB_{04}+\CB_{05},
\end{align*}
where
\begin{align*}
    &\CB_{01}:=\ve\|\CT[\gve,\PP_0f_1,\PP_0 f_2]\|_{L^2_IL^2_x(L_\gamma^2)^{\prime}}+\ve\|\CT[\gve,\PP_0f_2,\PP_0 f_1]\|_{L^2_IL^2_x(L_\gamma^2)^{\prime}},\\
     &\CB_{02}:=\ve\|\CT[\gve,\PP_0 f_1,\PP_0^{\perp} f_2]\|_{L^2_IL^2_x(L_\gamma^2)^{\prime}}+\ve\|\CT[\gve,\PP_0f_2,\PP_0^{\perp} f_1]\|_{L^2_IL^2_x(L_\gamma^2)^{\prime}},\\
    &\CB_{03}:=\ve\|\CT[\gve,\PP_0^{\perp} f_1,\PP_0^{\perp} f_2]\|_{L^2_IL^2_x(L_\gamma^2)^{\prime}}+\ve\|\CT[\gve,\PP_0^{\perp}f_2,\PP_0^{\perp} f_1]\|_{L^2_IL^2_x(L_\gamma^2)^{\prime}},\\
    &\CB_{04}:=\ve\|\CT[\gve,\PP_0^{\perp} f_1,\PP_0f_2]\|_{L^2_IL^2_x(L_\gamma^2)^{\prime}}+\ve\|\CT[\gve,\PP_0^{\perp}f_2,\PP_0 f_1]\|_{L^2_IL^2_x(L_\gamma^2)^{\prime}},\\
    &\CB_{05}:=\ve\|\CT[f_1, f_2,\gve]\|_{L^2_IL^2_x(L_\gamma^2)^{\prime}}+\ve\|\CT[f_2, f_1,\gve]\|_{L^2_IL^2_x(L_\gamma^2)^{\prime}}.
\end{align*}
For the first part of $\CB_{01}$, applying H\"{o}lder's inequality, we have
\begin{align}
    \begin{aligned}
        &\|\CT[\gve,\PP_0f_1,\PP_0 f_2]\|_{L^2_IL^2_x(L_\gamma^2)^{\prime}}\\
        \lesssim&\|\gve\|_{\widetilde{L}^{\infty}_IL^6_xL^2_v} \|\PP_0 f_1\|_{L^4_IL^6_xL^2_{\gamma}} \|\PP_0 f_2\|_{L^4_IL^6_xL^2_v}\\
        &\quad+\|\gve\|_{\widetilde{L}^{\infty}_IL^6_xL^2_{\ga}} \|\PP_0 f_1\|_{L^4_IL^6_xL^2_{v}} \|\PP_0 f_2\|_{L^4_IL^6_xL^2_v}\\
        \lesssim&\|\gve\|_{\widetilde{L}^{\infty}_IH^1_xL^2_{v}} \|\nabla_x\PP_0 f_1\|_{L^4_IL^2_xL^2_{\ga}} \|\nabla\PP_0 f_2\|_{L^4_IL^2_xL^2_v}\\
        \lesssim&\|\gve\|_{\widetilde{L}^{\infty}_IH^1_xL^2_{v}} \||\mathrm{D}_x|^{\frac{1}{2}}\PP_0 f_1\|_{\widetilde{L}^{\infty}_IL^2_xL^2_{\ga}}^{\frac{1}{2}} \||\mathrm{D}_x|^{\frac{3}{2}}\PP_0 f_1\|_{L^2_IL^2_xL^2_{\ga}}^{\frac{1}{2}}\\
        &\quad\times\||\mathrm{D}_x|^{\frac{1}{2}}\PP_0 f_2\|_{\widetilde{L}^{\infty}_IL^2_xL^2_{v}}^{\frac{1}{2}} \||\mathrm{D}_x|^{\frac{3}{2}}\PP_0 f_2\|_{L^2_IL^2_xL^2_{v}}^{\frac{1}{2}}\\
        \lesssim&\|\gve\|_{\widetilde{L}^{\infty}_IH^1_xL^2_{v}} \|f_1\|_{\widetilde{L}^{\infty}_IH^{\frac{1}{2}}_xL^2_{v}}^{\frac{1}{2}} \|\PP_0 f_1\|_{L^2_I\dot{H}^{\frac{3}{2}}_xL^2_{\ga}}^{\frac{1}{2}}\|f_2\|_{\widetilde{L}^{\infty}_IH^{\frac{1}{2}}_xL^2_{v}}^{\frac{1}{2}} \|\PP_0 f_2\|_{L^2_I\dot{H}^{\frac{3}{2}}_xL^2_{\ga}}^{\frac{1}{2}}\\
        \lesssim&\|\gve\|_{\widetilde{L}^{\infty}_IH^1_xL^2_{v}} \|f_1\|_{\MXe}\|f_2\|_{\MXe},
    \end{aligned}\label{CB-01-es-part1}
\end{align}
where we used the Sobolev embeddings $H^1_x\hookrightarrow L^6_x$ and the embedding $L^2_{\ga}\hookrightarrow L^2_v$, and the boundedness of operator $P_0:L^2_v\to L^2_{\ga}\subset L^2_v$. Similarly,
\begin{align}
    \|\CT[\gve,\PP_0f_2,\PP_0f_1]\|_{L^2_IL^2_x(L_\gamma^2)^{\prime}}\lesssim\|\gve\|_{\widetilde{L}^{\infty}_IH^1_xL^2_{v}} \|f_1\|_{\MXe}\|f_2\|_{\MXe}.\label{CB-01-es-part2}
\end{align}
Putting together \eqref{CB-01-es-part1} and \eqref{CB-01-es-part2}, it follows that
\begin{align*}
    \CB_{01}\lesssim\ve\|\gve\|_{\widetilde{L}^{\infty}_IH^1_xL^2_{v}} \|f_1\|_{\MXe}\|f_2\|_{\MXe}\lesssim\ve^{1-\frac{1}{2}\al}\|f_1\|_{\MXe}\|f_2\|_{\MXe}\lesssim\|f_1\|_{\MXe}\|f_2\|_{\MXe},
\end{align*}
for $\al<\frac{1}{4}$. In the last step we used the smallness of $\ve$, which will be assumed throughout the subsequent estimates. In the same way, for the term $\CB_{02}$, we obtain
\begin{align}
\begin{aligned}
    &\|\CT[\gve,\PP_0f_1,\PP_0^{\perp} f_2]\|_{L^2_IL^2_x(L_\gamma^2)^{\prime}}
        \\
        \lesssim&\|\gve\|_{\widetilde{L}^{\infty}_IH^1_xL^2_{v}} \|\nabla_x\PP_0 f_1\|_{L^4_IL^2_xL^2_{\ga}} \|\PP_0^{\perp} f_2\|_{L^4_I H^1_xL^2_v}\\
        \lesssim&\|\gve\|_{\widetilde{L}^{\infty}_IH^1_xL^2_{v}} \||\mathrm{D}_x|^{\frac{1}{2}}\PP_0 f_1\|_{\widetilde{L}^{\infty}_IL^2_xL^2_{\ga}}^{\frac{1}{2}} \||\mathrm{D}_x|^{\frac{3}{2}}\PP_0 f_1\|_{L^2_IL^2_xL^2_{\ga}}^{\frac{1}{2}}\\
        &\quad\times\||\mathrm{D}_x|^{\frac{1}{2}}\PP_0^{\perp} f_2\|_{\widetilde{L}^{\infty}_IL^2_xL^2_{v}}^{\frac{1}{2}} \||\mathrm{D}_x|^{\frac{3}{2}}\PP_0^{\perp}  f_2\|_{L^2_IH^{\frac{3}{2}}_xL^2_{v}}^{\frac{1}{2}}\\
        \lesssim&\|\gve\|_{\widetilde{L}^{\infty}_IH^1_xL^2_{v}} \|f_1\|_{\widetilde{L}^{\infty}_IH^{\frac{1}{2}}_xL^2_{v}}^{\frac{1}{2}} \|\PP_0 f_1\|_{L^2_I\dot{H}^{\frac{3}{2}}_xL^2_{\ga}}^{\frac{1}{2}}\\
        &\quad\times\|f_2\|_{\widetilde{L}^{\infty}_IH^{\frac{1}{2}}_xL^2_{v}}^{\frac{1}{2}} \|\PP_0^{\perp}  f_2\|_{L^2_IH^{\frac{3}{2}}_xL^2_{\ga}}^{\frac{1}{2}}\\
        \lesssim&\ve^{\frac{1}{4}}\|\gve\|_{\widetilde{L}^{\infty}_IH^1_xL^2_{v}} \|f_1\|_{\MXe}\|f_2\|_{\MXe}.
\end{aligned}\label{CB-02-part1}
\end{align}
By the same reasoning, we have
\begin{align}
 \|\CT[\gve,\PP_0f_2,\PP_0^{\perp} f_1]\|_{L^2_IL^2_x(L_\gamma^2)^{\prime}}\lesssim\ve^{\frac{1}{4}}\|\gve\|_{\widetilde{L}^{\infty}_IH^1_xL^2_{v}} \|f_1\|_{\MXe}\|f_2\|_{\MXe}. \label{CB-02-part2}
\end{align}
From \eqref{CB-02-part1}–\eqref{CB-02-part2}, we deduce that
\begin{align}
   \CB_{02}\lesssim\ve^{1-\frac{1}{2}\al}\|f_1\|_{\MXe}\|f_2\|_{\MXe}\lesssim\|f_1\|_{\MXe}\|f_2\|_{\MXe},\nonumber
\end{align}
for $\al<\frac{1}{4}$. For terms $\CB_{03}$, applying Lemma \ref{lem-tri-estimates} and Sobolev embedding, we have
\begin{align}\label{CB-03-part1}
\begin{aligned}
   &\|\CT[\gve,\PP_0^{\perp}f_1,\PP_0 f_2]\|_{L^2_IL^2_x(L_\gamma^2)^{\prime}}\\
   \lesssim&\|\gve\|_{\widetilde{L}^{\infty}_IH^1_xL^2_v}\|\PP_0^{\perp}f_1\|_{L^2_IH^{\frac{3}{2}}_xL^2_v}\|\PP_0 f_2\|_{\widetilde{L}^{\infty}_IH^{\frac{1}{2}}_xL^2_{\ga}}\\
   &\quad+\|\gve\|_{\widetilde{L}^{\infty}_IH^{\frac{3}{2}}_xL^2_v}\|\PP_0^{\perp}f_1\|_{L^2_IH^{\frac{3}{2}}_xL^2_v}\|\PP_0 f_2\|_{\widetilde{L}^{\infty}_IH^{\frac{1}{2}}_xL^2_{\ga}}\\
   &\quad+\|\gve\|_{\widetilde{L}^{\infty}_I H^{\ell}_xL^2_v}
   \|\PP_0^{\perp}f_1\|_{L^2_IH^{1}_xL^2_v}
   \|\PP_0 f_2\|_{\widetilde{L}^{\infty}_IH^{\frac{1}{2}}_xL^2_{\ga}}\\
   \lesssim&\ve^{\frac{1}{2}}\|\gve\|_{\widetilde{L}^{\infty}_IH^{\ell}_xL^2_v}
   \|\PP_0^{\perp}f_1\|_{L^2_IH^{\frac{3}{2}}_xL^2_v}
   \|f_2\|_{\widetilde{L}^{\infty}_IH^{\frac{1}{2}}_xL^2_{v}}\\
   \lesssim&\ve^{\frac{1}{2}}\|\gve\|_{\widetilde{L}^{\infty}_IH^{\ell}_xL^2_v}\|f_1\|_{\MXe}\|f_2\|_{\MXe}.
   \end{aligned}
   \end{align}
 	A parallel argument yields
\begin{align}
  \|\CT[\gve,\PP_0^{\perp}f_2,\PP_0 f_1]\|_{L^2_IL^2_x(L_\gamma^2)^{\prime}}\lesssim\ve^{\frac{1}{2}}\|\gve\|_{L^2_IH^{\ell}_xL^2_v}\|f_1\|_{\MXe}\|f_2\|_{\MXe}.\label{CB-03-part2}
\end{align}
 Gathering \eqref{CB-03-part1} and \eqref{CB-03-part2}, we further obtain
\begin{align*}
  \CB_{03}&\lesssim  \ve\|\gve\|_{\widetilde{L}^{\infty}_I H^{\ell}_xL^2_v}
   \|\PP_0^{\perp}f_1\|_{L^2_IH^{1}_xL^2_v}
   \|\PP_0 f_2\|_{\widetilde{L}^{\infty}_IH^{\frac{1}{2}}_xL^2_{\ga}}\\
   &\lesssim\ve^{\frac{3}{2}-\al(\ell-\frac{1}{2})}\|\gve\|_{L^2_IH^{\ell}_xL^2_v}\|f_1\|_{\MXe}\|f_2\|_{\MXe}\\
   &\lesssim\|f_1\|_{\MXe}\|f_2\|_{\MXe}.
\end{align*}
Following the same line of argument, we obtain
\begin{align}
    \begin{aligned}
       &\|\CT[\gve,\PP_0^{\perp}f_1,\PP_0^{\perp} f_2]\|_{L^2_IL^2_x(L_\gamma^2)^{\prime}}\\
       \lesssim&
       \|\gve \|_{\widetilde{L}^{\infty}_I H^{1}_xL^2_{v}}
       \|\PP_0^{\perp}f_1 \|_{\widetilde{L}^{\infty}_I H^{\frac{1}{2}}_xL^2_{v}}
       \|\PP_0^{\perp}f_2 \|_{L^2_I H^{\frac{3}{2}}_x L^2_{\ga}}\\
       &\quad+\|\gve \|_{\widetilde{L}^{\infty}_I H^{1}_xL^2_{v}}
       \|\PP_0^{\perp}f_1 \|_{\widetilde{L}^{\infty}_I H^{\frac{1}{2}}_xL^2_{v}}
       \|\PP_0^{\perp}f_2 \|_{L^2_I H^{\ell}_x L^2_{\ga}}\\
       \lesssim& \ve^{\frac{1}{2}}(\ve^{\frac{1}{2}}+\ve^{\frac{1}{2}-\be})\|f_1\|_{\MXe_I}\|f_2\|_{\MXe}\\
       \lesssim& \|\gve \|_{\widetilde{L}^{\infty}_I H^{1}_xL^2_{v}}\|f_1\|_{\MXe_I}\|f_2\|_{\MXe}.
         \end{aligned}\nonumber
   \end{align}
Similarly,
\begin{align}
   &\|\CT[\gve,\PP_0^{\perp}f_1,\PP_0^{\perp} f_2]\|_{L^2_IL^2_x(L_\gamma^2)^{\prime}} \lesssim \|\gve \|_{\widetilde{L}^{\infty}_I H^{1}_xL^2_{v}}\|f_1\|_{\MXe_I}\|f_2\|_{\MXe}.\label{CB-04-part2}
\end{align}
Then,
\begin{align*}
    \CB_{04}&\lesssim \ve\|\gve \|_{\widetilde{L}^{\infty}_I H^{1}_xL^2_{v}}\|f_1\|_{\MXe_I}\|f_2\|_{\MXe}\\
   &\lesssim\ve^{1-\al(\ell-\frac{1}{2})}\|\gve\|_{L^2_IH^{\ell}_xL^2_v}\|f_1\|_{\MXe}\|f_2\|_{\MXe}\\
   &\lesssim\|f_1\|_{\MXe}\|f_2\|_{\MXe}.
\end{align*}
By repeating the same reasoning, one also finds
    \begin{align}\label{CB-05-part1}
   \begin{aligned}
       &\|\CT[f_1,f_2,\gve]\|_{L^2_IL^2_x(L_\gamma^2)^{\prime}}\\\lesssim&\|f_1\|_{\widetilde{L}^{\infty}_IL^2_xL^2_v}
       \|f_2\|_{\widetilde{L}^{\infty}_IH^{\ell}_xL^2_v}
       \|\gve\|_{L^2_IH^{\ell}_xL^2_{\ga}}+\|f_1\|_{\widetilde{L}^{\infty}_IH^{\frac{1}{2}}_xL^2_v}\|f_2\|_{\widetilde{L}^{\infty}_IH^{\frac{1}{2}}_xL^2_v}\|\gve\|_{L^2_IH^{2}_xL^2_{\ga}}\\
       \lesssim&\Big(\ve^{-\be}\|\gve\|_{L^2_IH^{\ell}_xL^2_{\ga}}+\|\gve\|_{L^2_I H^{2}_x L^2_{\ga}}\Big)\|f_1\|_{\MXe} \|f_2\|_{\MXe}.
   \end{aligned}
\end{align}
Then, repeating above progress in \eqref{CB-05-part1}, we obtain
\begin{align}
 \|\CT[f_1,f_2,\gve]\|_{L^2_IL^2_x(L_\gamma^2)^{\prime}} \lesssim\Big(\ve^{-\be}\|\gve\|_{L^2_IH^{\ell}_xL^2_{\ga}}+\|\gve\|_{L^2_I H^{2}_x L^2_{\ga}}\Big)\|f_1\|_{\MXe} \|f_2\|_{\MXe}.\label{CB-05-part2}
\end{align}
Moreover,  combining \eqref{CB-05-part1}-\eqref{CB-05-part2}, we have
\begin{align*}
 \CB_{05}&\lesssim\ve\Big(\ve^{-\be}\|\gve\|_{L^2_IH^{\ell}_xL^2_{\ga}}+\|\gve\|_{L^2_I H^{2}_x L^2_{\ga}}\Big)\|f_1\|_{\MXe} \|f_2\|_{\MXe}\\
    &\lesssim\ve^{1-\be}\|\gve\|_{L^2_IH^{\ell}_xL^2_v}\|f_1\|_{\MXe}\|f_2\|_{\MXe}\\
   &\lesssim\|f_1\|_{\MXe}\|f_2\|_{\MXe}.
\end{align*}
Collecting the estimates for $\CB_{0i},\,i=1,\cdots,5$, we conclude that
\begin{align}
    \CB_0\lesssim\|f_1\|_{\MXe}\|f_2\|_{\MXe}.\label{B0-es}
\end{align}
Next, we deal with the terms $\CB_1,\CB_2,\CB_3$. Applying Lemma \ref{lem-tri-estimates} with $m=\frac{1}{2},\frac{3}{2}$, we further have for $\CB_1$,
\begin{align*}
\CB_1&\lesssim\CB_{1}^{(1)}+\CB_{2}^{(1)},
\end{align*}
where we denote by
\begin{align*}
 \CB_{1}^{(1)}&:=\ve\|\CT[f_1,\gve,f_2]\|_{L^2_I\dot H^{\frac{1}{2}}_x(L_\gamma^2)^{\prime}}+\ve\|\CT[\gve,f_1,f_2]\|_{L^2_I\dot H^{\frac{1}{2}}_x(L_\gamma^2)^{\prime}}\\
 &\quad+\ve\|\CT[f_1,f_2,\gve]\|_{L^2_I\dot H^{\frac{1}{2}}_x(L_\gamma^2)^{\prime}}\\
 &=:\CB_{11}^{(1)}+\CB_{12}^{(1)}+\CB_{13}^{(1)},\\ \CB_{2}^{(1)}&:=\ve\|\CT[f_2,\gve,f_1]\|_{L^2_I\dot H^{\frac{1}{2}}_x(L_\gamma^2)^{\prime}}+\ve\|\CT[\gve,f_2,f_1]\|_{L^2_I\dot H^{\frac{1}{2}}_x(L_\gamma^2)^{\prime}}\\
 &\quad+\ve\|\CT[f_2,f_1,\gve]\|_{L^2_I\dot H^{\frac{1}{2}}_x(L_\gamma^2)^{\prime}}\\
  &=:\CB_{21}^{(1)}+\CB_{22}^{(1)}+\CB_{23}^{(1)}.\\
\end{align*}
Since $\PP_0$ is bounded in $L^2_v$ and is $L_\gamma^2\hookrightarrow L^2_v$, applying Lemma \ref{lem-tri-estimates} and Lemmas \ref{lem-fluid-solution}-\ref{lem-smooth-fluid-solution}, for the term $\CB_{11}^{(1)}$ we can get for $\al<\frac{1}{4}$,
\begin{align}\label{CB-1-11-ker}
      \begin{aligned}
        &\ve\|\CT[f_1,\gve,\PP_0f_2]\|_{L^2_I\dot H^{\frac{1}{2}}_x(L_\gamma^2)^{\prime}}\\
        \lesssim&\ve\|f_1\|_{L^2_I\dot H^{\frac{3}{2}}_xL^2_v}
        \|\PP_0 f_2\|_{\widetilde{L}^{\infty}_I\dot H^{\frac{1}{2}}_x L_\gamma^2}
        (\|\gve\|_{\widetilde{L}^{\infty}_I\dot H^{\frac{3}{2}}_xL^2_v}+\|\gve\|_{\widetilde{L}^{\infty}_I\dot H^{\frac{5}{2}}_xL^2_v})\\
        &\quad+\ve\|f_1\|_{L^2_I\dot H^{\frac{3}{2}}_xL^2_v}
        \|\PP_0 f_2\|_{\widetilde{L}^{\infty}_I\dot H^{\frac{1}{2}}_x L_\gamma^2}
        \|\gve\|_{\widetilde{L}^{\infty}_I\dot H^{\ell}_xL^2_v}
\\
        &\quad+\ve\|f_1\|_{\widetilde{L}^{\infty}_I\dot H^{\frac{1}{2}}_xL^2_v}
        \|\PP_0 f_2\|_{L^2_I\dot H^{\frac{3}{2}}_xL^2_v}
        \|\gve\|_{\widetilde{L}^{\infty}_I\dot H^{\ell}_xL^2_v}\\
        \lesssim& \ve(1+\sqrt{\ve})(\ve^{-2\alpha}+\ve^{-\al(\ell-\frac{1}{2})})\|f_1\|_{\MXe_I}\|f_2\|_{\MXe_I}\\
        \lesssim&\|f_1\|_{\MXe_I}\|f_2\|_{\MXe_I}.
    \end{aligned}
\end{align}
Similarly, for $\ell\in(\frac{3}{2},2]$ and $\be<\frac{1}{2}$, $\al<\frac{1}{4}$, we have
\begin{align}
    \begin{aligned}
       &\ve\|\CT[f_1,\gve,\PP_0^{\perp}f_2]\|_{L^2_I\dot H^{\frac{1}{2}}_x(L_\gamma^2)^{\prime}}\\
       \lesssim&\ve\|f_1\|_{\widetilde{L}^{\infty}_IH^{\frac{1}{2}}_xL^2_v} \|\PP_0^{\perp}f_2\|_{L^2_IH^{\ell}_x L_\gamma^2}
        (\|\gve\|_{\widetilde{L}^{\infty}_IH^{\ell}_xL^2_v}+\|\gve\|_{\widetilde{L}^{\infty}_IH^{\frac{3}{2}}_xL^2_v})\\
        &\quad+\ve\|f_1\|_{\widetilde{L}^{\infty}_IH^{\frac{1}{2}}_xL^2_v}
        \|\PP_0^{\perp}f_2\|_{L^2_IH^{\frac{3}{2}}_x L_\gamma^2}
        \|\gve\|_{\widetilde{L}^{\infty}_IH^{\ell}_xL^2_v}\\
        \lesssim&\ve(\ve^{\frac{1}{2}-\be-\al(\ell-\frac{1}{2})}+\ve^{\frac{1}{2}-\al(\ell-\frac{1}{2})})\|f_1\|_{\MXe_I}\|f_2\|_{\MXe_I}\\
        \lesssim&\|f_1\|_{\MXe_I}\|f_2\|_{\MXe_I}.
    \end{aligned}\label{ineq-CB-tri-low-3rd}
\end{align}
Combining \eqref{CB-1-11-ker} with \eqref{ineq-CB-tri-low-3rd} to yield
\begin{align}
\begin{aligned}
    \CB_{11}^{(1)}&\lesssim \ve\|\CT[f_1,\gve,\PP_0f_2]\|_{L^2_I\dot H^{\frac{1}{2}}_x(L_\gamma^2)^{\prime}}+\ve\|\CT[f_1,\gve,\PP_0^{\perp}f_2]\|_{L^2_I\dot H^{\frac{1}{2}}_x(L_\gamma^2)^{\prime}}\\
    &\lesssim\|f_1\|_{\MXe_I}\|f_2\|_{\MXe_I}.
\end{aligned}\nonumber
\end{align}
Using a similar approach, we also have
\begin{align}
\begin{aligned}
    \CB_{12}^{(1)}&\lesssim \ve\|\CT[\gve,f_1,\PP_0f_2]\|_{L^2_I\dot H^{\frac{1}{2}}_x(L_\gamma^2)^{\prime}}+\ve\|\CT[\gve,f_1,\PP_0^{\perp}f_2]\|_{L^2_I\dot H^{\frac{1}{2}}_x(L_\gamma^2)^{\prime}}\\
    &\lesssim\|f_1\|_{\MXe_I}\|f_2\|_{\MXe_I}. \nonumber
\end{aligned}
\end{align}
For the term $\CB_{13}^{(1)}$, applying Lemma \ref{lem-tri-estimates} we have
\begin{align}\label{ineq-CB-tri-low-1st}
    \begin{aligned}
     \CB_{13}^{(1)}&\lesssim\ve \|f_1\|_{L^2_I\dot H^{\frac{3}{2}}_xL^2_v}
      \|f_2\|_{\widetilde{L}^{\infty}_I\dot H^{\frac{1}{2}}_xL^2_v}
      (\|\gve\|_{\widetilde{L}^{\infty}_I\dot H^{\ell}_xL^2_v}+\|\gve\|_{\widetilde{L}^{\infty}_I\dot H^{\frac{3}{2}}_xL^2_v})\\
      &\quad+\ve\|f_1\|_{\widetilde{L}^{\infty}_I\dot H^{\frac{1}{2}}_xL^2_v}\|f_2\|_{L^2_I\dot H^{\frac{3}{2}}_xL^2_v}\|\gve\|_{\widetilde{L}^{\infty}_I\dot H^{\ell}_xL^2_v}\\
      &\quad+\ve\|f_1\|_{\widetilde{L}^{\infty}_I\dot H^{\ell}_xL^2_v}\|f_2\|_{\widetilde{L}^{\infty}_I\dot H^{\frac{1}{2}}_xL^2_v}\|\gve\|_{L^2_I\dot H^{\frac{3}{2}}_x L_\gamma^2}\\
      &\lesssim (1+\sqrt{\ve})\ve^{1-\al(\ell-\frac{1}{2})}\|f_1\|_{\MXe_I}\|f_2\|_{\MXe_I}\\
       &\lesssim \|f_1\|_{\MXe_I}\|f_2\|_{\MXe_I}.
    \end{aligned}
\end{align}

Similarly,
\begin{align}
 \ve\|\CT[f_2,f_1,\gve]\|_{L^2_IH^{\frac{1}{2}}_x(L_\gamma^2)^{\prime}}&\lesssim \|f_1\|_{\MXe_I}\|f_2\|_{\MXe_I},\nonumber
\end{align}
that is
\begin{align}
    \CB_{13}^{(1)}\lesssim\|f_1\|_{\MXe_I}\|f_2\|_{\MXe_I}.\nonumber
\end{align}
Summing up the terms $\CB_{11}^{(1)},\CB_{12}^{(1)},\CB_{13}^{(1)}$, we have
\begin{align}
    \CB_{1}^{(1)}\lesssim\|f_1\|_{\MXe_I}\|f_2\|_{\MXe_I}.\nonumber
\end{align}
Due to symmetric structure, we repeating progress as above to get
\begin{align}
    \CB_{2}^{(1)}\lesssim\|f_1\|_{\MXe_I}\|f_2\|_{\MXe_I}.\nonumber
\end{align}
Thus we have
\begin{align}\label{ineq-CB-low-1-2}
  \CB_1\lesssim\CB_{1}^{(1)}+\CB_{2}^{(1)}\lesssim\|f_1\|_{\MXe_I}\|f_2\|_{\MXe_I}.
\end{align}

For the term $\CB_2$, setting $m=\frac{3}{2}$ in Lemma \ref{lem-tri-estimates} and following the same way as in \eqref{ineq-CB-tri-low-3rd}-\eqref{ineq-CB-tri-low-1st}, we also have
\begin{align}
    \begin{aligned}
      &(\ve+\ve^{\frac{3}{2}})\|\CT[\gve,f_1,f_2]\|_{L^2_I\dot H^{\frac{3}{2}}_x(L_\gamma^2)^{\prime}}\\
      \lesssim&\ve(\|\CT[\gve,f_1,\PP_0f_2]\|_{L^2_I\dot H^{\frac{3}{2}}_x(L_\gamma^2)^{\prime}}+\|\CT[\gve,f_1,\PP_0^{\perp}f_2]\|_{L^2_I\dot H^{\frac{3}{2}}_x(L_\gamma^2)^{\prime}})\\
      \lesssim&\ve\|\gve\|_{L^2_I\dot H^{\frac{3}{2}}_xL^2_v}
      \|f_1\|_{\widetilde{L}^{\infty}_I\dot H^{\ell}_xL^2_v}
      \|\PP_0f_2\|_{\widetilde{L}^{\infty}_I\dot H^{\ell}_x L_\gamma^2}\\
      &\quad+\|\gve\|_{\widetilde{L}^{\infty}_I\dot H^{\ell}_xL^2_v}
      \|f_1\|_{L^2_I\dot H^{\frac{3}{2}}_xL^2_v}
      \|\PP_0f_2\|_{\widetilde{L}^{\infty}_I\dot H^{\ell}_x L_\gamma^2}\\
      &\quad+\|\gve\|_{\widetilde{L}^{\infty}_I\dot H^{\ell}_xL^2_v}
      \|f_1\|_{\widetilde{L}^{\infty}_I\dot H^{\ell}_xL^2_v}
      \|\PP_0f_2\|_{L^2_I\dot H^{\frac{3}{2}}_x L_\gamma^2}\\
      &\quad+\|\gve\|_{\widetilde{L}^{\infty}_I\dot H^{\ell}_xL^2_v}
      \|f_1\|_{\widetilde{L}^{\infty}_I\dot H^{\ell}_xL^2_v}
      \|\PP_0^{\perp}f_2\|_{L^2_I\dot H^{\ell}_x L_\gamma^2}\\
      \lesssim&\ve(\ve^{-2\be}+\ve^{-\beta-\al(\ell-\frac{1}{2})}(1+\sqrt{\ve}))\|f_1\|_{\MXe_I}\|f_2\|_{\MXe_I}\\
      &\quad+\ve^{1+\frac{1}{2}}\ve^{-2\be-\al(\ell-\frac{1}{2})}\|f_1\|_{\MXe_I}\|f_2\|_{\MXe_I}\\
      \lesssim&\|f_1\|_{\MXe_I}\|f_2\|_{\MXe_I}.
      \end{aligned}\nonumber
\end{align}
Similarly, due to the symmetric structure, then we have
\begin{align}
    \CB_2&\lesssim\|f_1\|_{\MXe_I}\|f_2\|_{\MXe_I}.\nonumber
\end{align}

For the high-order term $\CB_3$, applying Lemma \ref{lem-tri-estimates}, we obtain
\begin{align}
    \begin{aligned}
      &\ve^{\beta}(\ve+\ve^{\frac{3}{2}})\|\CT[\gve,f_1,f_2]\|_{L^2_I\dot H^{\ell}_x(L_\gamma^2)^{\prime}}\\\lesssim& \ve^{1+\be}\|\gve\|_{\widetilde{L}^{\infty}_I\dot H^{\ell}_xL^2_v}
      \|f_1\|_{\widetilde{L}^{\infty}_I\dot H^{\ell}_xL^2_v}
      \|f_2\|_{L^2_I\dot H^{\ell}_x L_\gamma^2}\\
      \lesssim&\ve^{1-\be-\al(\ell-\frac{1}{2})}\|f_1\|_{\MXe_I}\|f_2\|_{\MXe_I}\\
      \lesssim&\|f_1\|_{\MXe_I}\|f_2\|_{\MXe_I}.
    \end{aligned}\nonumber
\end{align}
Hence,
\begin{align}
 \CB_3\lesssim\|f_1\|_{\MXe_I}\|f_2\|_{\MXe_I}. \label{ineq-CB-high}
\end{align}

Gathering \eqref{B0-es}-\eqref{ineq-CB-high}, we obtain that
\begin{align}
   \|\mathfrak{T}^{\ve}[\gve,f_1,f_2]\|_{\MXe_I}\lesssim\|f_1\|_{\MXe_I}\|f_2\|_{\MXe_I}.\label{CB-tri-part-1}
\end{align}
We complete the proof.
\end{proof}

\subsection{\texorpdfstring{Estimates for Trilinear Operator $\mathfrak{T}^{\ve}$}{Estimates for Trilinear Operator T-epsilon}}
We show the Operator $\mathfrak{T}^{\ve}$ is continuous on $\MXe_I\times\MXe_I\times\MXe_I$. By applying Lemma \ref{lem-semi-hypocoercivity} we have
\begin{align*}
 &\|\mathfrak{T}^{\ve}[f_1,f_2,f_3]\|_{\MXe_I}\\
 \lesssim&\ve\|\CT_{\mathrm{sym}}[f_1,f_2,f_3]\|_{L^2_I\dot H^{\frac{1}{2}}_x(L_\gamma^2)^{\prime}}\\
 &\quad+(\ve+\ve^{\frac{3}{2}})\|\CT_{\mathrm{sym}}[f_1,f_2,f_3]\|_{L^2_I\dot H^{\frac{3}{2}}_x(L_\gamma^2)^{\prime}}\\
 &\quad+\ve^{\beta}(\ve+\ve^{\frac{3}{2}})\|\CT_{\mathrm{sym}}[f_1,f_2,f_3]\|_{L^2_I\dot H^{\ell}_x(L_\gamma^2)^{\prime}}\\
 &\quad+(\ve+\ve^{\frac{3}{2}})(1+\ve^{\be})\|\CT_{\mathrm{sym}}[f_1,f_2,f_3]\|_{L^2_IL^2_x(L_\gamma^2)^{\prime}}.
\end{align*}
By Lemma \ref{lem-tri-estimates}, we obtain
\begin{align*}
\begin{aligned}
&\|\CT[f_1,f_2,f_3]\|_{L^2_I\dot H^{\frac{1}{2}}_x(L_\gamma^2)^{\prime}}\\\lesssim&2\|f_1\|_{\widetilde{L}^{\infty}_I\dot H^{\frac{1}{2}}_xL^2_v}\|f_2\|_{\widetilde{L}^{\infty}_I\dot H^{\ell}_xL^2_v}\|f_3\|_{L^2_I\dot H^{\ell}_x L_\gamma^2}\\
&\quad+\|f_1\|_{\widetilde{L}^{\infty}_I\dot H^{\ell}_xL^2_v}\|f_2\|_{\widetilde{L}^{\infty}_I\dot H^{\frac{1}{2}}_xL^2_v}\|f_3\|_{L^2_I\dot H^{\ell}_x L_\gamma^2}\\
 &\quad+\|f_1\|_{\widetilde{L}^{\infty}_I\dot H^{\frac{1}{2}}_xL^2_v}\|f_2\|_{\widetilde{L}^{\infty}_I\dot H^{\ell}_xL^2_v}\|f_3\|_{L^2_I\dot H^{\frac{3}{2}}_x L_\gamma^2}\\
   \lesssim&\ve^{-2\beta}(1+\sqrt{\ve})\|f_1\|_{\MXe_I}\|f_2\|_{\MXe_I}\|f_3\|_{\MXe_I}.
\end{aligned}
\end{align*}
Due to the symmetric structure, we can directly deduce similar bounds for all permutations. Hence,
\begin{align}
    \ve\|\CT_{\mathrm{sym}}[f_1,f_2,f_3]\|_{L^2_I\dot H^{\frac{1}{2}}_x(L_\gamma^2)^{\prime}}\lesssim\|f_1\|_{\MXe_I}\|f_2\|_{\MXe_I}\|f_3\|_{\MXe_I}\quad \text{for}~\be<\frac{1}{2}.\label{ineq-tri-non-lower}
\end{align}
Applying  Lemma \ref{lem-tri-estimates} with $m=\frac{3}{2}$, and the embedding $\dot H^{\ell}_x\hookrightarrow \dot H^{\frac{3}{2}}_x$, we. also have
\begin{align*}
&\|\CT[f_1,f_2,f_3]\|_{L^2_I\dot H^{\frac{3}{2}}_x(L_\gamma^2)^{\prime}}\\\lesssim&2\|f_1\|_{\widetilde{L}^{\infty}_I\dot H^{\frac{3}{2}}_xL^2_v}\|f_2\|_{\widetilde{L}^{\infty}_I\dot H^{\ell}_xL^2_v}\|f_3\|_{L^2_I\dot H^{\ell}_x L_\gamma^2}\\
&\quad+\|f_1\|_{\widetilde{L}^{\infty}_I\dot H^{\ell}_xL^2_v}\|f_2\|_{\widetilde{L}^{\infty}_I\dot H^{\frac{3}{2}}_xL^2_v}\|f_3\|_{L^2_I\dot H^{\ell}_x L_\gamma^2}\\
 &\quad+\|f_1\|_{\widetilde{L}^{\infty}_I\dot H^{\ell}_xL^2_v}\|f_2\|_{\widetilde{L}^{\infty}_I\dot H^{\ell}_xL^2_v}\|f_3\|_{L^2_I\dot H^{\frac{3}{2}}_x L_\gamma^2}\\
 \lesssim& \|f_1\|_{\widetilde{L}^{\infty}_I\dot H^{\ell}_xL^2_v}\|f_2\|_{\widetilde{L}^{\infty}_I\dot H^{\ell}_xL^2_v}\|f_3\|_{L^2_I\dot H^{\ell}_x L_\gamma^2}\\
   \lesssim&\ve^{-\frac{3}{2}\beta}(1+\sqrt{\ve})\|f_1\|_{\MXe_I}\|f_2\|_{\MXe_I}\|f_3\|_{\MXe_I},
\end{align*}
and
\begin{align}
    \begin{aligned}
        \|\CT[f_1,f_2,f_3]\|_{L^2_I\dot H^{\ell}_x(L_\gamma^2)^{\prime}}\lesssim&\|f_1\|_{\widetilde{L}^{\infty}_I\dot H^{\ell}_xL^2_v}\|f_2\|_{\widetilde{L}^{\infty}_I\dot H^{\ell}_xL^2_v}\|f_3\|_{L^2_I\dot H^{\ell}_x L_\gamma^2}\\
 \lesssim&\ve^{-\frac{3}{2}\beta}(1+\sqrt{\ve})\|f_1\|_{\MXe_I}\|f_2\|_{\MXe_I}\|f_3\|_{\MXe_I}.
    \end{aligned}
\end{align}
Similarly, they yields
\begin{align}
(\ve+\ve^{\frac{3}{2}})\|\CT_{\mathrm{sym}}[f_1,f_2,f_3]\|_{L^2_I\dot H^{\frac{3}{2}}_x(L_\gamma^2)^{\prime}}\lesssim\|f_1\|_{\MXe_I}\|f_2\|_{\MXe_I}\|f_3\|_{\MXe_I},\quad \text{for}~\be<\frac{1}{2},
\end{align}
and
\begin{align}
\ve^{\be}(\ve+\ve^{\frac{3}{2}})\|\CT_{\mathrm{sym}}[f_1,f_2,f_3]\|_{L^2_I\dot H^{\ell}_x(L_\gamma^2)^{\prime}}\lesssim\|f_1\|_{\MXe_I}\|f_2\|_{\MXe_I}\|f_3\|_{\MXe_I},\quad \text{for}~\be<\frac{1}{2}.\label{ineq-tri-non-higher}
\end{align}
We now consider the zeroth-order terms. By symmetry of the trilinear operator, we only present the proof for one part, as the remaining cases are analogous. Putting the decomposition $f=\PP_0 f+\PP_0^{\perp} f$ in $L^2_v$, we have
    \begin{align}
      \begin{aligned}
       &\|\CT[f_1,f_2,f_3]\|_{L^2_IL^2_x(L_\gamma^2)^{\prime}}\\
       \lesssim&\|\CT[\PP_0 f_1,\PP_0 f_2,\PP_0 f_3]\|_{L^2_IL^2_x(L_\gamma^2)^{\prime}}+\|\CT[\PP_0^{\perp} f_1,\PP_0^{\perp} f_2,\PP_0^{\perp} f_3]\|_{L^2_IL^2_x(L_\gamma^2)^{\prime}}\\
      &\quad+\|\CT[\PP_0^{\perp} f_1,\PP_0^{\perp} f_2,\PP_0 f_3]\|_{L^2_IL^2_x(L_\gamma^2)^{\prime}}+\|\CT[\PP_0 f_1,\PP_0^{\perp} f_2,\PP_0 f_3]\|_{L^2_IL^2_x(L_\gamma^2)^{\prime}}\\
      &\quad+\|\CT[\PP_0 f_1,\PP_0^{\perp} f_2,\PP_0^{\perp} f_3]\|_{L^2_IL^2_x(L_\gamma^2)^{\prime}}+\|\CT[\PP_0 f_1,\PP_0 f_2,\PP_0^{\perp} f_3]\|_{L^2_IL^2_x(L_\gamma^2)^{\prime}}\\
      &\quad+\|\CT[\PP_0 f_1^{\perp},\PP_0 f_2,\PP_0^{\perp} f_3]\|_{L^2_IL^2_x(L_\gamma^2)^{\prime}}+\|\CT[\PP_0 f_1^{\perp},\PP_0 f_2,\PP_0 f_3]\|_{L^2_IL^2_x(L_\gamma^2)^{\prime}}.
      \end{aligned}
    \end{align}
    In what follows, we shall treat the terms successively. By H\"{o}lder's inequality and Grgliardo--Nienberg inequality, we have
        \begin{align}\label{0th-ker1-ker2-ker3}
            \begin{aligned}
&\|\CT[f_1,f_2,f_3]\|_{L^2_IL^2_x(L_\gamma^2)^{\prime}}\\\lesssim&\Big(\Big(\int_{\BR}\|\PP_0f_1\|_{L^2_{\gamma}}^6\|\PP_0f_2\|_{L^2_{v}}^6\ud x\Big)^{\frac{1}{3}}\ud t\Big)^{\frac{1}{2}}\|\PP_0 f_3\|_{\widetilde{L}^{\infty}_IL^3_xL^2_v}\\
            &\quad+\Big(\Big(\int_{\BR}\|\PP_0f_2\|_{L^2_{\gamma}}^6\|\PP_0f_1\|_{L^2_{v}}^6\ud x\Big)^{\frac{1}{3}}\ud t\Big)^{\frac{1}{2}}\|\PP_0 f_3\|_{\widetilde{L}^{\infty}_IL^3_xL^2_v}\\
            \lesssim&\Big(\Big(\int_{\BR}\Big|\nabla_x\Big(\|\PP_0f_1\|_{L^2_{\gamma}}\|\PP_0f_2\|_{L^2_{v}}\Big)\Big|^2\ud x\Big)\ud t\Big)^{\frac{1}{2}}\| f_3\|_{\widetilde{L}^{\infty}_IH^{\frac{1}{2}}_xL^2_v}\\
            &\quad+\Big(\Big(\int_{\BR}\Big|\nabla_x\Big(\|\PP_0f_2\|_{L^2_{\gamma}}\|\PP_0f_1\|_{L^2_{v}}\Big)\Big|^2\ud x\Big)\ud t\Big)^{\frac{1}{2}}\| f_3\|_{\widetilde{L}^{\infty}_IH^{\frac{1}{2}}_xL^2_v}.
            \end{aligned}
        \end{align}
        Applying Lemma \ref{lem-homo-bi} , we have
        \begin{align}\label{0th-ker1-ker2}
            \begin{aligned}
                &\Big(\Big(\int_{\BR}\Big|\nabla_x\Big(\|\PP_0f_1\|_{L^2_{\gamma}}\|\PP_0f_2\|_{L^2_{v}}\Big)\Big|^2\ud x\Big)\ud t\Big)^{\frac{1}{2}}\\
                \lesssim&\|\PP_0f_1\|_{L^2_I\dot H^1_xL^2_v}\|\PP_0f_2\|_{\widetilde{L}^{\infty}_I\dot H^{\ell}_xL^2_{\ga}}\\
                &\quad+\|\PP_0f_2\|_{L^2_I\dot H^1_xL^2_v}\|\PP_0f_1\|_{\widetilde{L}^{\infty}_I\dot H^{\ell}_xL^2_{\ga}}\\
                \lesssim& \|\PP_0f_1\|_{L^2_I\dot H^1_xL^2_v}\|f_2\|_{\widetilde{L}^{\infty}_I\dot H^{\ell}_xL^2_{v}}\\
                &\quad+ \|\PP_0f_2\|_{L^2_I\dot H^1_xL^2_v}\|f_1\|_{\widetilde{L}^{\infty}_I\dot H^{\ell}_xL^2_{v}}\\
                \lesssim&\|f_1\|_{\MXe}\|f_2\|_{\MXe}.
            \end{aligned}
        \end{align}
        Putting \eqref{0th-ker1-ker2} into \eqref{0th-ker1-ker2-ker3}, we obtain
        \begin{align}\label{0th-ker1-ker2-ker3-final}
           \|\CT[\PP_0f_1,\PP_0f_2,\PP_0f_3]\|_{L^2_IL^2_x(L_\gamma^2)^{\prime}}\lesssim \|f_1\|_{\MXe}\|f_2\|_{\MXe}\|f_3\|_{\MXe}.
        \end{align}
         Next, we apply Lemma~\ref{lem-tri-estimates} with $m=0$, and then deal with each term on the right-hand side of the above expression by choosing suitable parameters $(\alpha_i,\beta_i,\gamma_i,r_i,\delta_i),\,i=1,\cdots,5$, respectively obtaining the following estimates:
\begin{align}
    \begin{aligned}
        &\|\CT[\PP_0^{\perp} f_1,\PP_0^{\perp} f_2,\PP_0f_3]\|_{L^2_IL^2_x(L_\gamma^2)^{\prime}}\\
        \lesssim&\|\PP_0^{\perp} f_1\|_{\widetilde{L}^{\infty}_IL^2_xL^2_v}\|\PP_0^{\perp} f_2\|_{L^2_IH^{\ell}_xL^2_v}\|\PP_0 f_3\|_{\widetilde{L}^{\infty}_IH^{\ell}_xL^2_{\ga}}\\
        &\quad+\|\PP_0^{\perp} f_2\|_{\widetilde{L}^{\infty}_IL^2_xL^2_v}\|\PP_0^{\perp} f_1\|_{L^2_IH^{\ell}_xL^2_v}\|\PP_0 f_3\|_{\widetilde{L}^{\infty}_I H^{\ell}_xL^2_{\ga}}\\
        &\quad+\|\PP_0^{\perp} f_1\|_{L^2_I H^{\ell}_xL^2_v}\|\PP_0^{\perp} f_2\|_{\widetilde{L}^{\infty}_I H^{\ell}_xL^2_v}\|\PP_0 f_3\|_{\widetilde{L}^{\infty}_IL^2_xL^2_{\ga}}\\
        \lesssim&\|f_1\|_{\widetilde{L}^{\infty}_IL^2_xL^2_v}\|\PP_0^{\perp} f_2\|_{L^2_IH^{\ell}_xL^2_v}\|f_3\|_{\widetilde{L}^{\infty}_IH^{\ell}_xL^2_{v}}\\
        &\quad+\|f_2\|_{\widetilde{L}^{\infty}_I H^{\frac{1}{2}}_xL^2_v}
        \|\PP_0^{\perp} f_1\|_{L^2_IH^{\ell}_xL^2_{\ga}}\|f_3\|_{\widetilde{L}^{\infty}_I H^{\ell}_xL^2_{v}}\\
        &\quad+\|\PP_0^{\perp} f_1\|_{L^2_I H^{\ell}_xL^2_{\ga}}\|f_2\|_{\widetilde{L}^{\infty}_I H^{\ell}_xL^2_v}\|f_3\|_{\widetilde{L}^{\infty}_IL^2_xL^2_{v}}\\
        \lesssim&\ve^{\frac{1}{2}-2\be}\|f_1\|_{\MXe_I}\|f_2\|_{\MXe_I}\|f_3\|_{\MXe_I},
    \end{aligned}
\end{align}
and
\begin{align}
     \begin{aligned}
        &\|\CT[\PP_0^{\perp} f_1,\PP_0^{\perp} f_2,\PP_0^{\perp} f_3]\|_{L^2_IL^2_x(L_\gamma^2)^{\prime}}\\
        \lesssim&\|\PP_0^{\perp} f_1\|_{\widetilde{L}^{\infty}_IH^{\frac{1}{2}}_xL^2_v}
        \|\PP_0^{\perp} f_2\|_{\widetilde{L}^{\infty}_IH^{\ell}_xL^2_v}
        \|\PP_0^{\perp} f_2\|_{L^2_IH^{\ell}_xL^2_{\ga}}\\
        &\quad+\|\PP_0^{\perp} f_2\|_{\widetilde{L}^{\infty}_IH^{\frac{1}{2}}_xL^2_v}
        \|\PP_0^{\perp} f_1\|_{\widetilde{L}^{\infty}_IH^{\ell}_xL^2_v}
        \|\PP_0^{\perp} f_2\|_{L^2_IH^{\ell}_xL^2_{\ga}}\\
        \lesssim&\ve^{\frac{1}{2}-2\be}\|f_1\|_{\MXe_I}\|f_2\|_{\MXe_I}\|f_3\|_{\MXe_I}.
    \end{aligned}
\end{align}
Besides, we have
\begin{align}
      \begin{aligned}
        &\|\CT[\PP_0 f_1,\PP_0^{\perp} f_2,\PP_0^{\perp}f_3]\|_{L^2_IL^2_x(L_\gamma^2)^{\prime}}\\
        \lesssim&\|\PP_0f_1\|_{\widetilde{L}^{\infty}_IL^2_xL^2_v}\|\PP_0^{\perp} f_2\|_{\widetilde{L}^{\infty}_IH^{\ell}_xL^2_v}\|\PP_0^{\perp} f_3\|_{L^2_IH^{\ell}_xL^2_{\ga}}\\
        &\quad+\|\PP_0f_2\|_{\widetilde{L}^{\infty}_IL^2_xL^2_v}\|\PP_0^{\perp} f_1\|_{\widetilde{L}^{\infty}_IH^{\ell}_xL^2_v}\|\PP_0^{\perp} f_3\|_{L^2_IH^{\ell}_xL^2_{\ga}}\\
        \lesssim&\|f_1\|_{\widetilde{L}^{\infty}_IL^2_xL^2_v}\|f_2\|_{\widetilde{L}^{\infty}_IH^{\ell}_xL^2_v}\|\PP_0^{\perp} f_3\|_{L^2_IH^{\ell}_xL^2_{\ga}}\\
        &\quad+\|f_2\|_{\widetilde{L}^{\infty}_IL^2_xL^2_v}\|f_1\|_{\widetilde{L}^{\infty}_IH^{\ell}_xL^2_v}\|\PP_0^{\perp} f_3\|_{L^2_IH^{\ell}_xL^2_{\ga}}\\
        \lesssim&\ve^{\frac{1}{2}-2\be}\|f_1\|_{\MXe}\|f_2\|_{\MXe}\|f_3\|_{\MXe}.
    \end{aligned}
\end{align}
By the same reasoning, one can derive
\begin{align}
& \|\CT[\PP_0 f_1,\PP_0 f_2,\PP_0^{\perp}f_3]\|_{L^2_IL^2_x(L_\gamma^2)^{\prime}}\lesssim\ve^{\frac{1}{2}-2\be}\|f_1\|_{\MXe}\|f_2\|_{\MXe}\|f_3\|_{\MXe},\\
& \|\CT[\PP_0 f_2,\PP_0^{\perp} f_1,\PP_0f_3]\|_{L^2_IL^2_x(L_\gamma^2)^{\prime}}\lesssim\ve^{\frac{1}{2}-2\be}\|f_1\|_{\MXe_I}\|f_2\|_{\MXe_I}\|f_3\|_{\MXe_I}.
\end{align}
Repeating the above argument yields
\begin{align}\label{0th-oth1-ker2-ker3}
   \|\CT[\PP_0 f_1^{\perp},\PP_0 f_2,\PP_0^{\perp} f_3]\|_{L^2_IL^2_x(L_\gamma^2)^{\prime}}&\lesssim\ve^{\frac{1}{2}-2\be}\|f_1\|_{\MXe_I}\|f_2\|_{\MXe_I}\|f_3\|_{\MXe_I},\\
   \|\CT[\PP_0 f_1^{\perp},\PP_0 f_2,\PP_0 f_3]\|_{L^2_IL^2_x(L_\gamma^2)^{\prime}}&\lesssim\ve^{\frac{1}{2}-2\be}\|f_1\|_{\MXe_I}\|f_2\|_{\MXe_I}\|f_3\|_{\MXe_I}.
\end{align}
From \eqref{0th-ker1-ker2-ker3-final}-\eqref{0th-oth1-ker2-ker3}, it follows that
\begin{align}\label{ineq-tri-non-0th}
\begin{aligned}
 (1+\ve^{\be})(\ve+\ve^{\frac{3}{2}})\|\CT[f_1,f_2,f_3]\|_{L^2_IL^2_x(L_\gamma^2)^{\prime}}&\lesssim\ve^{\frac{3}{2}-2\be}\|f_1\|_{\MXe_I}\|f_2\|_{\MXe_I}\|f_3\|_{\MXe_I}\\
 &\lesssim\|f_1\|_{\MXe_I}\|f_2\|_{\MXe_I}\|f_3\|_{\MXe_I},
\end{aligned}
\end{align}
provided that $\beta<\tfrac{1}{2}$ and $\varepsilon$ is sufficiently small.
Collecting the estimates in \eqref{ineq-tri-non-lower}-\eqref{ineq-tri-non-higher} with \eqref{ineq-tri-non-0th}, we conclude the estimate \eqref{tri-non-esti} for trilinear term in Proposition \ref{prop-fixed-point-lem-conditions}-(5).

\bibliographystyle{abbrv}
\bibliography{references}

\end{document}